\documentclass[11pt,leqno]{article}
\usepackage{a4wide}
\usepackage{amssymb}
\usepackage{amsthm}
\usepackage{amsmath}
\usepackage{amscd}
\usepackage{array}
\usepackage[curve]{xypic}
\usepackage{hyperref}
\usepackage[latin2]{inputenc}
\usepackage{t1enc}
\usepackage{enumerate}
\usepackage[mathscr]{eucal}
\usepackage[british]{babel}
\usepackage[dvips]{graphicx}
\usepackage[dvips]{epsfig}
\usepackage{epsf}
\usepackage{indentfirst}
\usepackage{url}

\newcommand{\Ker}{\mathrm{Ker}}
\newcommand{\Res}{\mathrm{Res}}
\newcommand{\Aut}{\mathrm{Aut}}

\newcommand{\Hom}{\mathrm{Hom}}

\newcommand{\val}{\mathrm{val}}

\newcommand{\End}{\mathrm{End}}
\newcommand{\GL}{\mathrm{GL}}
\newcommand{\bb}{{\bf b}}
\newcommand{\bk}{{\bf k}}

\newcommand{\id}{\mathrm{id}}

\newcommand{\bg}{(\hspace{-0.06cm}(}
\newcommand{\jg}{)\hspace{-0.06cm})}

\newcommand{\bs}{[\hspace{-0.04cm}[}
\newcommand{\js}{]\hspace{-0.04cm}]}
\newtheorem{thm}{Theorem}[section]
\newtheorem{pro}[thm]{Proposition}
\newtheorem{lem}[thm]{Lemma}
\newtheorem{cor}[thm]{Corollary}
\newtheorem{df}[thm]{Definition}
\newtheorem{exm}[thm]{Example}
\newtheorem{rem}[thm]{Remark}

\theoremstyle{definition}

\begin{document}
\title{$(\varphi,\Gamma)$-modules over noncommutative overconvergent and Robba rings} 
\author{Gergely Z\'abr\'adi}
\date{\today}
\maketitle
\begin{abstract}
We construct noncommutative multidimensional versions of overconvergent power
series rings and Robba rings. We show that the category of \'etale
$(\varphi,\Gamma)$-modules over certain completions of these rings are equivalent to the category of
\'etale $(\varphi,\Gamma)$-modules over the corresponding classical
overconvergent, resp.\ Robba rings (hence also to the category of $p$-adic
Galois representations of $\mathbb{Q}_p$). Moreover, in the case of Robba
rings, the assumption of \'etaleness is not necessary, so there exists a
notion of trianguline objects in this sense.
\end{abstract}

\section{Introduction}

In recent years it has become increasingly clear that some kind of $p$-adic
version of the local Langlands correspondence should exist. In fact, Colmez
\cite{Co1,Co2} constructed such a correspondence for
$\GL_2(\mathbb{Q}_p)$. His construction is done in several steps using
$(\varphi,\Gamma)$-modules (the category of which is well-known \cite{F} to be
equivalent to the category of $p$-adic Galois representations of
$\mathbb{Q}_p$).  We briefly recall Colmez's correspondence here. Let $K$
be a finite extension of $\mathbb{Q}_p$ with ring of integers $o_K$ and
uniformizer $p_K$. 

The so-called ``Montreal-functor'' associates to a smooth $o_K$-torsion
representation of the standard Borel subgroup $B_2(\mathbb{Q}_p)$ of
$\GL_2(\mathbb{Q}_p)$ an $o_K$-torsion $(\varphi,\Gamma)$-module over
Fontaine's ring $\mathcal{O_E}$. If we are given a unitary Banach space
representation $\Pi$ over the field $K$ of the group $\GL_2(\mathbb{Q}_p)$
then it admits an $o_K$-lattice $L(\Pi)$ which is invariant under
$\GL_2(\mathbb{Q}_p)$. Hence $L(\Pi)/p_K^r$ is a smooth $o_K$-torsion
representation that we restrict now to $B_2(\mathbb{Q}_p)$. The
$(\varphi,\Gamma)$-module associated to $\Pi$ is the projective limit (as
$r\to\infty$) of the $(\varphi,\Gamma)$-modules associated to $L(\Pi)/p_K^r$
via the Montreal functor. This is generalized in \cite{SVi} to general
reductive groups over $\mathbb{Q}_p$.  

The reverse direction, how one adjoins a unitary continuous $p$-adic
representation to a $2$-dimensional $(\varphi,\Gamma)$-module $D$ over Fontaine's ring, is
even more subtle. One first constructs a unitary $p$-adic Banach space
representation $\Pi(D)$ to each $2$-dimensional \emph{trianguline} $(\varphi,\Gamma)$-module $D$ over $\mathcal{E}=\mathcal{O_E}[p^{-1}]$ using some
kind of parabolic induction. This Banach space is well described as a quotient
of the space of $p$-adic functions satisfying certain properties by a certain
$\GL_2(\mathbb{Q}_p)$-invariant subspace (see \cite{Co3} and \cite{Br} for
details), however, a priori it is not clear whether or not it is
nontrivial. On the other hand, there is a general construction of a (somewhat 
bigger) $\GL_2(\mathbb{Q}_p)$-representation $D\boxtimes_\delta\mathbb{P}^1$ which
is in fact the space of global sections of a $\GL_2(\mathbb{Q}_p)$-equivariant
sheaf $U\mapsto D\boxtimes_\delta U$ ($U\subseteq \mathbb{P}^1$ open) on the projective space
$\mathbb{P}^1(\mathbb{Q}_p)\cong\GL_2(\mathbb{Q}_p)/B_2(\mathbb{Q}_p)$ for any (not necessarily $2$-dimensional) $(\varphi,\Gamma)$-module $D$ and any unitary character $\delta\colon\mathbb{Q}_p^{\times}\to o_K^{\times}$. This sheaf has the following properties: $(i)$ the centre of $\GL_2(\mathbb{Q}_p)$ acts via $\delta$ on $D\boxtimes_\delta \mathbb{P}^1$; $(ii)$ we have $D\boxtimes_\delta \mathbb{Z}_p\cong D$ as a module over the monoid $\begin{pmatrix}\mathbb{Z}_p\setminus\{0\}&\mathbb{Z}_p\\0&1\end{pmatrix}$ (where we regard $\mathbb{Z}_p$ as an open subspace in $\mathbb{P}^1=\mathbb{Q}_p\cup\{\infty\}$). (See
\cite{SVZ} for a generalization of this construction to general reductive
groups.) Then Colmez shows that in case $D$ is $2$-dimensional and trianguline, then there exists a unitary character $\delta$ (namely $\delta=\chi^{-1}\det D$ where $\chi$ is the cyclotomic character and $\det D$ is the character associated to the $1$-dimensional $(\varphi,\Gamma)$-module $\bigwedge^2 D$ via Fontaine's equivalence composed with class field theory) such that a certain subspace $D^\natural\boxtimes_\delta \mathbb{P}^1$ (for the definition see \cite{Co2}) of $D\boxtimes_\delta\mathbb{P}^1$ is isomorphic to the dual of the Banach space representation $\Pi(\check{D})$ associated
earlier to the dual $(\varphi,\Gamma)$-module $\check{D}$---therefore showing in particular that the previous construction is
nonzero. This subspace makes sense also in case $D$ is not
trianguline (nor of rank $2$), but a priori only known to be
$B_2(\mathbb{Q}_p)$-invariant. Moreover, whenever $D$ is indecomposable and $2$-dimensional, then the above $\delta$ is unique \cite{P}, and whenenever $D$ is absolutely irreducible and $\geq 3$-dimensional, then there does not exist \cite{P} such a character $\delta$ (so that the subspace $D^\natural\boxtimes_\delta \mathbb{P}^1$ is $\GL_2(\mathbb{Q}_p)$-invariant). Since the construction of $D\mapsto D^{\natural}\boxtimes_\delta \mathbb{P}^1$ behaves well in families (see chapter II in \cite{Co1}) and the trianguline Galois-representations are Zariski-dense in the deformation space of $2$-dimensional $(\varphi,\Gamma)$-modules with given reduction mod $p$ \cite{Ki}, Colmez \cite{Co1} shows
that this subspace is not only $B_2(\mathbb{Q}_p)$, but also
$\GL_2(\mathbb{Q}_p)$-invariant for general $2$-dimensional $(\varphi,\Gamma)$-modules. Moreover, for $\delta=\chi^{-1}\det D$ (in this case we omit the subscript $\delta$ from the notation) we have a short exact sequence
\begin{equation*}
0\to \Pi(\check{D})\to D\boxtimes\mathbb{P}^1\to \Pi(D)\to 0
\end{equation*}
where $\Pi(D)$ is the unitary Banach-space representation associated to $D$ via the $p$-adic Langlands correspondence.

Colmez (\cite{Co1}, chapter V and VI) also identifies the space $\Pi(D)^{an}$ of locally analytic and the space $\Pi(D)^{alg}$ of locally algebraic vectors in the Banach-space representation $\Pi(D)$. These play a crucial role in the proof of the compatibility of the $p$-adic and the classical local Langlands correspondence. In fact, we have $\Pi(D)^{an}=(D^{\dagger}\boxtimes\mathbb{P}^1)/K\cdot (D^{\natural}\boxtimes\mathbb{P}^1)$ where $D^\dagger\boxtimes\mathbb{P}^1$ is the subspace of elements $x\in D\boxtimes\mathbb{P}^1$ such that both $\Res_{\mathbb{Z}_p}^{\mathbb{P}^1}(x)$ and $\Res_{\mathbb{Z}_p}^{\mathbb{P}^1}\left(\begin{pmatrix}0&1\\1&0\end{pmatrix}x\right)$ lie in the subspace of overconvergent elements $D^\dagger\subset D\cong D\boxtimes\mathbb{Z}_p$. $D^\dagger$ is an \'etale $(\varphi,\Gamma)$-module over the ring $\mathcal{E}^\dagger$ of overconvergent power series with coefficients in $K$ such that $D\cong \mathcal{E}\otimes_{\mathcal{E}^\dagger}D^\dagger$ \cite{CC}.

Let now $G$ be the group of $\mathbb{Q}_p$-points of a connected $\mathbb{Q}_p$-split reductive group and $P=TN$ a Borel subgroup of $G$. Further denote by $\Phi^+$ the set of positive roots with respect to $P$ and $\Delta\subset \Phi^+$ the set of simple roots. The above noted generalizations of Colmez's work (\cite{SVi} and \cite{SVZ}) both use a certain microlocalisation $\Lambda_\ell(N_0)$ (constructed originally in \cite{SVe}) of the Iwasawa algebra $\Lambda(N_0)$ of a compact open subgroup $N_0$ of $N$. This can be thought of as the noncommutative analogue of Fontaine's ring $\mathcal{O_E}$. On the other hand, Colmez's $p$-adic Langlands correspondence heavily relies on the theory of trianguline $(\varphi,\Gamma)$-modules. A $(\varphi,\Gamma)$-module over the Robba ring is a free module $D_{rig}^\dagger$ over $\mathcal{R}$ together with commuting semilinear actions of the operator $\varphi$ and the group $\Gamma$ such that $\varphi$ takes a basis of the free module to another basis. Such a $(\varphi,\Gamma)$-module $D_{rig}^\dagger$ is said to be \'etale (or of slope $0$) if there is a basis of $D_{rig}^\dagger$ such that the matrix of $\varphi$ in this basis is an invertible matrix over the subring $\mathcal{O}_{\mathcal{E}}^\dagger\subset\mathcal{R}$ of overconvergent Laurent series. An \'etale $(\varphi,\Gamma)$-module over $\mathcal{R}$ is \emph{trianguline} if it admits a filtration of (not necessarily \'etale) $(\varphi,\Gamma)$-modules over $\mathcal{R}$ with subquotients of rank $1$ possibly after a finite base change $E\otimes_K\cdot$. The fact that the Robba ring and the ring of overconvergent Laurent series play such a role in the construction of the $p$-adic Langlands correspondence for $\GL_2(\mathbb{Q}_p)$ and also in the identification of the locally analytic vectors is the motivation for the construction of noncommutative analogues of these rings---as they will most probably be needed for a future correspondence for reductive groups other than $\GL_2(\mathbb{Q}_p)$.

The motivation of this paper is twofold. On one hand, we reinterpret the ring $\Lambda_\ell(N_0)$ as follows. Instead of localising and completing the Iwasawa algebra $\Lambda(N_0)$, one may construct $\Lambda_\ell(N_0)$ as the projective limit of certain skew group rings over $\mathcal{O_E}$. The only assumptions on the ring $R=\mathcal{O_E}$ such that this new  construction of $\Lambda_\ell(N_0)$ can be carried out are that $R$ admits an inclusion $\chi\colon\mathbb{Z}_p\to R^{\times}$ of the additive group $\mathbb{Z}_p$ into its group of invertible elements and an \'etale action of an operator $\varphi$ that is compatible with $\chi$. The noncommutative ring that is constructed is a completed skew group ring $R\bs H_1,\ell\js$ of a closed normal subgroup $H_1$ of a pro-$p$ group $H_0$ such that $\ell\colon H_0\twoheadrightarrow \mathbb{Z}_p$ is a homomorphism with kernel $H_1$ (hence $H_0/H_1\cong\mathbb{Z}_p$). The main result in this direction is Prop.\ \ref{equivcat} showing that the category of $\varphi$-modules over $R$ is equivalent to the category of $\varphi$-modules over the completed skew group ring $R\bs H_1,\ell\js$. This can be applied  also to the ring $R=\mathcal{O}_\mathcal{E}^\dagger$ of overconvergent Laurent series with coefficients in $o_K$ and the Robba ring $\mathcal{R}$. The other motivation (probably also the more important one) is the construction of the right noncommutative analogues of $\mathcal{O}_\mathcal{E}^\dagger$ and $\mathcal{R}$. The elements of the rings $\mathcal{R}\bs H_1,\ell\js$ and
$\mathcal{O}_{\mathcal{E}}^\dagger\bs H_1,\ell\js$, however, are not necessarily
convergent in any open annulus since they are obtained by taking an inverse
limit. Therefore in section \ref{partialrobba} we construct the rings
$\mathcal{R}(H_1,\ell)$ and $\mathcal{R}^{int}(H_1,\ell)$ as
direct limits of certain microlocalisations of the distribution algebra. The
elements of these are convergent in a region of the form
\begin{equation*}
\left\{\rho_2<|b_\alpha|<1,\ |b_\beta|<|b_\alpha|^{r}\text{ for }\beta\in\Phi^+\setminus\{\alpha\}\right\}
\end{equation*}
for some $p^{-1}<\rho_2<1$ and $1\leq r\in\mathbb{Z}$. In section \ref{relate} we show that $\mathcal{R}\bs N_1,\ell\js$ (resp.\ $\mathcal{O}_\mathcal{E}^\dagger\bs N_1,\ell\js$) is a certain completion of $\mathcal{R}(N_1,\ell)$ (resp.\ $\mathcal{R}^{int}(N_1,\ell)$). Note that although the natural map $j_{int}\colon\mathcal{R}^{int}(N_1,\ell)\to\mathcal{O}_\mathcal{E}^\dagger\bs N_1,\ell\js$ is injective, the map $j\colon \mathcal{R}(N_1,\ell)\to\mathcal{R}\bs N_1,\ell\js$ is not. Both rings $\mathcal{R}(N_1,\ell)$ and its integral version admit an \'etale action of the monoid $T_+=\{t\in T\mid tN_0t^{-1}\subseteq N_0\}$. However, it is an open question whether the categories of \'etale $T_+$-modules over these rings are equivalent to the \'etale $T_+$-modules over their completions.

In my opinion, the right noncommutative analogues of the ring $\mathcal{R}$ (resp.\ $\mathcal{O}_\mathcal{E}^\dagger$) is $\mathcal{R}(N_1,\ell)$ (resp.\ $\mathcal{R}^{int}(N_1,\ell)$) in the context of $\mathbb{Q}_p$-split reductive groups $G$ over $\mathbb{Q}_p$ as both rings admit an \'etale action of the monoid $T_+$ and their elements converge in certain polyannuli. However, it might still be useful to also consider the rings $\mathcal{R}\bs H_1,\ell\js$ and
$\mathcal{O}_{\mathcal{E}}^\dagger\bs H_1,\ell\js$, as they can help us compare the category of usual $(\varphi,\Gamma)$-modules with the category of $T_+$-modules over $\mathcal{R}(N_1,\ell)$ (resp.\ over $\mathcal{R}^{int}(N_1,\ell)$) using the equivalence of categories in Proposition \ref{equivcat}. Note that only one variable is inverted in these rings in contrast to the rings constructed in \cite{SZ}. The reasons for this are the following: $(i)$ this way $\mathcal{R}^{int}(N_1,\ell)$ is a subring of $\Lambda_\ell(N_0)$; $(ii)$ the equivalence of categories in Proposition \ref{equivcat} holds for rings in which only one variable is inverted; $(iii)$ all the usual $(\varphi,\Gamma)$-modules are overconvergent, ie.\ descend to $\mathcal{O}_{\mathcal{E}}^\dagger$ already in one variable. However, if $\mathbb{Q}_p$ is replaced by a finite unramified extension $F$ then one might have to consider Lubin-Tate $(\varphi,\Gamma_F)$-modules (with $\Gamma_F\cong o_F^{\times}$) instead so that the monoid $\varphi^{\mathbb{N}}\Gamma_F$ is isomorphic to $o_F\setminus\{0\}$. These $(\varphi,\Gamma_F)$-modules are not overconvergent in general but they might still correspond to objects over certain multivariable Robba rings (in which all the variables are inverted). For a first result in this direction see \cite{Be2}. It is plausible to expect that for general reductive groups $G$ over $F$ one has to invert exactly $|F:\mathbb{Q}_p|$ ($\mathbb{Q}_p$-)variables that correspond to the root subgroup $N_\alpha\cong F\cong\mathbb{Q}_p^{|F:\mathbb{Q}_p|}$ for a given simple root $\alpha$.

\subsection{Acknowledgements}

This research was supported by the Hungarian OTKA grant K-100291. The proof of
Prop.\ \ref{equivcat} is a direct generalization of Thm.\ 8.20 in \cite{SVZ}
and grew out of related discussions with Marie-France Vigneras and Peter
Schneider. I express my gratitude to both of them for allowing me to include
this proof here and for other valueable discussions. My debt is especially due
to Peter Schneider from whom I learnt the most that I know about $p$-adic
representation theory. I would like to thank Torsten Schoeneberg for reading
the manuscript and for his comments. I also benefited from enlightening
conversations with Kiran Kedlaya. Finally, I thank the referee for a careful reading of the manuscript.

\section{Completed skew group rings}\label{phiring}

Let $R$ be a commutative ring (with identity) with the following properties: 
\begin{enumerate}[$(i)$]
\item There exists a group homomorphism $\chi\colon\mathbb{Z}_p\hookrightarrow
  R^{\times}$.
\item The ring $R$ admits an \'etale action of the $p$-Frobenius $\varphi$
  that is compatible with $\chi$. More precisely there is an injective ring
  homomorphism $\varphi\colon R\hookrightarrow R$ such that
  $\varphi(\chi(x))=\chi(px)$ and
\begin{equation*}
R=\bigoplus_{i=0}^{p-1}\chi(i)\varphi(R).
\end{equation*}
In particular, $R$ is free of rank $p$ over $\varphi(R)$.
\end{enumerate}

We remark first of all that one may iterate $(ii)$ $c$ times for any positive integer $c$ to obtain
\begin{equation}
R=\bigoplus_{i=0}^{p^c-1}\chi(i)\varphi^c(R)\ .
\label{iterate}
\end{equation}
Indeed, by induction we may assume that \eqref{iterate} holds for $c-1$ and obtain
\begin{equation*}
R=\bigoplus_{k=0}^{p^{c-1}-1}\chi(k)\varphi^{c-1}(R)=
\bigoplus_{k=0}^{p^{c-1}-1}\chi(k)\varphi^{c-1}\left(\bigoplus_{j=0}^{p-1}\chi(j)\varphi(R)\right)=
\bigoplus_{k=0}^{p^{c-1}-1}\bigoplus_{j=0}^{p-1}\chi(k+p^{c-1}j)\varphi^c(R)\ 
\end{equation*}
since $\varphi^{c-1}$ takes direct sums to direct sums as it is injective. Now the claim follows from noting that any integer $0\leq i\leq p^c-1$ can be uniquely written in the form $i=k+p^{c-1}j$ with $0\leq k\leq p^{c-1}-1$ and $0\leq j\leq p-1$.

Moreover, for any $x\in\mathbb{Z}_p$ we have $\chi(p^cx)=\varphi^c(\chi(x))\in\varphi^c(R)^\times$. Hence $\chi(i)\varphi^c(R)=\chi(i+p^cx)\varphi^c(R)$ and we may replace each value of $i$ in the formula \eqref{iterate} by any element in the coset $i+p^c\mathbb{Z}_p$.

\begin{df}\label{phiringdef}
We call a ring $R$ with the above properties $(i)$ and $(ii)$ a $\varphi$-ring
over $\mathbb{Z}_p$ or often just a $\varphi$-ring. 
\end{df}
For example, if
$K/\mathbb{Q}_p$ is a finite extension with ring of integers $o$ and
uniformizer $p_K$ then the Iwasawa algebra $o\bs T\js$ is a $\varphi$-ring
with the homomorphism  
\begin{eqnarray*}
\chi\colon\mathbb{Z}_p&\rightarrow&o\bs T\js\\
1&\mapsto&1+T
\end{eqnarray*}
and Frobenius $\varphi(T)=(T+1)^p-1$. Similarly with the same $\chi$ and
$\varphi$, Fontaine's ring $\mathcal{O_E}$, its field of fractions $\mathcal{E}$, the Robba ring $\mathcal{R}$ and
the rings $\mathcal{E}^\dagger$, $\mathcal{O}_{\mathcal{E}}^{\dagger}$ of
overconvergent power series are also $\varphi$-rings (for the definitions see
the paragraph before Lemma \ref{isom} (for $\mathcal{O_E}$ and $\mathcal{E}$), section
\ref{partialrobba} (for $\mathcal{R}$, $\mathcal{O}_{\mathcal{E}}^{\dagger}$,
and $\mathcal{E}^\dagger$)).

\begin{lem}\label{polynomial}
For any positive integer $c$ we have a ring isomorphism
\begin{eqnarray*}
\varphi^c(R)[X]/(X^{p^c}-\chi(p^c))&\overset{\sim}{\longrightarrow}&R\\
X&\longmapsto&\chi(1)\ .
\end{eqnarray*}
\end{lem}
\begin{proof}
Since the polynomial ring $\varphi^c(R)[X]$ is a free object in the category of commutative $\varphi^c(R)$-algebras, we may extend the natural inclusion homomorphism $f\colon\varphi^c(R)\hookrightarrow R$ given by $(ii)$ to a ring homomorphism $\tilde{f}\colon\varphi^c(R)[X]\rightarrow R$ by any free choice for the value $\tilde{f}(X)$, in particular such that $\tilde{f}(X):=\chi(1)\in R$ and, of course, $\tilde{f}_{|\varphi^c(R)}:=f$. We need to show that $\tilde{f}$ is surjective with kernel equal to the ideal generated by $X^{p^c}-\chi(p^c)$. Note that $\chi(p^c)=\varphi^c(\chi(1))$ lies in $\varphi^c(R)$ so the claim makes sense. 

By $(i)$ the map $\chi$ is a group homomorphism, so $\chi(r)=\chi(1)^r=\tilde{f}(X)^r=\tilde{f}(X^r)$ lies in the image of $\tilde{f}$ for any positive integer $r$. Hence we obtain the surjectivity from \eqref{iterate} by noting that $\varphi^c(R)$ also lies in the image of $\tilde{f}$.

Using again $\chi(r)=\tilde{f}(X^r)$ with the choice of $r=p^c$ we see immediately that $X^{p^c}-\chi(p^c)$ lies in the kernel of $\tilde{f}$. Moreover, note that $\varphi^c(R)[X]/(X^{p^c}-\chi(p^c))$ is a free module of rank $p^c$ over $\varphi^c(R)$ with generators the classes of $\{X^r\}_{r=0}^{p^c-1}$ in the quotient. On the other hand, $R$ is also a free module of rank $p^c$ with generators $\{\chi(r)\}_{r=0}^{p^c-1}$ by \eqref{iterate} and these two sets of generators correspond to each other under the map $\tilde{f}$ hence the isomorphism.
\end{proof}

Let $H_0$ be a pro-$p$ group of finite rank (therefore a compact $p$-adic Lie group by Corollary 4.3 and Theorem 8.18 in \cite{DDMS}) without elements of order $p$ admitting a continuous surjective group homomorphism $\ell\colon
H_0\twoheadrightarrow\mathbb{Z}_p$ with kernel $H_1:=\Ker(\ell)$. We assume
further the following
\begin{enumerate}[$(A)$]
\item $H_0$ also admits an injective group endomorphism $\varphi\colon
H_0\hookrightarrow H_0$ with finite cokernel and compatible with $\ell$ in the
sense that $\ell(\varphi(h))=\varphi(\ell(h))=p\ell(h)$. In particular, we
have $\varphi(H_1)\subseteq H_1$.
\item $\bigcap_{n\geq 1}\varphi^n(H_0)=\{1\}$ and the subgroups $\varphi^n(H_0)$ form a system of neighbourhoods of $1$ in $H_0$.
\end{enumerate}

We remark first of all that by a Theorem of Serre (Thm.\ 1.17 in \cite{DDMS}) any finite index subgroup in $H_0$ is open. Hence the homomorphism $\varphi$ is automatically continuous and the subgroups $\varphi^n(H_0)$ are open.

Note that $H_1$ is a closed subgroup of $H_0$ hence it is also a pro-$p$ group of finite rank. By assumption $(B)$ we also have in particular that the subgroups
$\varphi^n(H_1)$ form a system of open neighbourhoods of $1$ in $H_1$. Note
that the subgroups $\varphi^n(H_1)$ may not be normal in either $H_1$ or
$H_0$. Hence we define the normal subgroups $H_k\lhd
H_0$ as the normal subgroup of $H_0$ generated by $\varphi^{k-1}(H_1)$. Since
$H_1$ is normal in $H_0$ we automatically have $H_k\subseteq H_1$ for any
$k\geq 1$. Moreover, since the $p$-adic Lie group $H_1$ has a system of
neighbourhoods of $1$ containing only characteristic subgroups, the $H_k$ also form
a system of neighbourhoods of $1$ in $H_1$. On the other hand, we have by
definition that $\varphi(H_k)\subseteq H_{k+1}\subseteq H_k$ for each $k\geq 1$. In
particular, we have an induced $\varphi$ action on the quotient group
$H_0/H_k$. This is, of course, no longer injective.

Since the group $\mathbb{Z}_p$ is topologically generated by one element, we
may find a splitting $\iota\colon\mathbb{Z}_p\hookrightarrow H_0$ for the
group homomorphism $\ell$. We fix this splitting $\iota$, too. Assume further that

\begin{enumerate}[$(C)$]
\item the group homomorphism $\iota$ is $\varphi$-equivariant, ie.\ we have $\iota(\varphi(x))=\varphi(\iota(x))$ for all $x\in \mathbb{Z}_p$.
\end{enumerate}

We define the skew group ring $R[H_1/H_k,\ell,\iota]$ as follows. We put
\begin{equation}
R[H_1/H_k,\ell,\iota]:=\bigoplus_{h\in H_1/H_k}Rh\ . \label{defR[H_1/H_k]}
\end{equation}
as left $R$-modules. Note that since $H_1$ is a normal subgroup in $H_0$, we also have $H_1/H_k\lhd H_0/H_k$. Therefore we obtain a conjugation action of $\mathbb{Z}_p$ on $H_1/H_k$ given by 
\begin{eqnarray*}
\rho\colon\mathbb{Z}_p&\to&\Aut(H_1/H_k)\\
z&\mapsto&(h\mapsto\iota(z)h\iota(z)^{-1}, h\in H_1/H_k)\ .
\end{eqnarray*}
Since $H_1/H_k$ is a finite $p$-group, $|\Aut(H_1/H_k)|<\infty$ and we have
an integer $c_k\geq1$ such that $p^{c_k}\mathbb{Z}_p\subseteq\Ker(\rho)$. The multiplication is defined so that $\varphi^{c_k}(R)$
commutes with elements $h$ in $H_1/H_k$ and $\chi(i)$ acts on $H_1/H_k$
via $\iota\circ\chi^{-1}$ and conjugation. More precisely for
$r_1,r_2\in R$ and $h_1,h_2\in H_1/H_k$ we may write 
\begin{equation}
r_2=\sum_{i=0}^{p^{c_k}-1}\chi(i)\varphi^{c_k}(r_{i,2})\label{expand}
\end{equation}
and put
\begin{equation}\label{mult}
(r_1h_1)(r_2h_2):=\sum_{i=0}^{p^{c_k}-1}r_1\chi(i)\varphi^{c_k}(r_{i,2})\left((\iota(i)^{-1}h_1\iota(i))h_2\right)\in\bigoplus_{h\in H_1/H_k}Rh.
\end{equation}
Note that in case $r_2=1$ we have $(r_1h_1)h_2=r_1(h_1h_2)$ and in case $h_1=1$ we have
$r_1(r_2h_2)=(r_1r_2)h_2$. Moreover, by the choice of $c_k$,
$\iota(p^{c_k}\mathbb{Z}_p)$ lies in the centre of $H_0/H_k$. So we may use
any set of representatives of $\mathbb{Z}_p/p^{c_k}\mathbb{Z}_p$ instead of
$\{0,1,\dots,p^{c_k}-1\}$ in \eqref{expand} in order to compute
\eqref{mult}. Indeed, if $i\equiv i'\pmod{p^{c_k}}$ then we have
$\chi(i)\varphi^{c_k}(r_{i,2})=\chi(i')\varphi^{c_k}(\chi(\frac{i-i'}{p^{c_k}})r_{i,2})$
and $\iota(i)^{-1}h_1\iota(i)=\iota(i')^{-1}h_1\iota(i')$.

\begin{lem}\label{lemmamult}
The multipliction \eqref{mult} equips $R[H_1/H_k,\ell,\iota]$ with a ring
structure. 
\end{lem}
\begin{proof}
There exists an easy, but rather long computation showing this. However, there
is another, more conceptual description of the ring $R[H_1/H_k,\ell,\iota]$
pointed out by Torsten Schoeneberg that proves this lemma without any
computations. Let $S$ be the group ring $S:=\varphi^{c_k}(R)[H_1/H_k]$ and
$\sigma$ be the automorphism of $S$ trivial on $\varphi^{c_k}(R)$ and acting
by conjugation with $\iota(1)$ on $H_1/H_k$, ie.\ for $h\in H_1/H_k$ put
$\sigma(h):=\iota(1)^{-1}h\iota(1)$. Now define the skew polynomial ring
$S[X,\sigma]$ by the relation $aX=X\sigma(a)$ for $a\in S$. Note that by the definition of $\sigma$, the subring $\varphi^{c_k}(R)$ lies in the centre of $S[X,\sigma]$ therefore so does $\chi(p^{c_k})= \varphi^{c_k}(\chi(1))\in \varphi^{c_k}(R)$. On the other hand, we have $aX^{p^{c_k}}=X^{p^{c_k}}\sigma^{p^{c_k}}(a)=X^{p^{c_k}}a$ for all $a\in S$ since $\sigma^{p^{c_k}} $ is the conjugation by the central element $\iota(1)^{p^{c_k}}=\iota(p^{c_k})\in H_0/H_k$ on $H_1/H_k$ and is trivial by definition on $\varphi^{c_k}(R)$ hence $\sigma^{p^{c_k}}=\id_S$. This shows that  $X^{p^{c_k}}-\chi(p^{c_k})$ is central and that $S[X,\sigma](X^{p^{c_k}}-\chi(p^{c_k}))=(X^{p^{c_k}}-\chi(p^{c_k}))S[X,\sigma]$ is a two-sided ideal in $S[X,\sigma]$. So we may form the quotient ring and compute (as left $\varphi^{c_k}(R)$-modules)
\begin{align*}
S[X,\sigma]/(X^{p^{c_k}}-\chi(p^{c_k}))\cong\left(\bigoplus_{r=0}^{\infty}\bigoplus_{h\in
  H_1/H_k}X^r\varphi^{c_k}(R)h\right)/ (X^{p^{c_k}}-\chi(p^{c_k}))\cong\\
\cong\bigoplus_{h\in
  H_1/H_k}\left(\varphi^{c_k}(R)[X]/(X^{p^{c_k}}-\chi(p^{c_k}))\right)h\cong \bigoplus_{h\in
  H_1/H_k}Rh
\end{align*}
using Lemma \ref{polynomial} in the middle. Note that on the component $h=1$ in the above direct sum the identification is even multiplicative as Lemma \ref{polynomial} gives an isomorphism of rings, not just $\varphi^{c_k}(R)$-modules. Hence $S[X,\sigma]/(X^{p^{c_k}}-\chi(p^{c_k}))$ contains $R$ as a subring and the isomorphism above is an isomorphism of left $R$-modules. 
The transport of ring structure gives back the definition \eqref{mult} of
multiplication on the right hand side. Indeed, we have
\begin{align*}
(r_1h_1)(r_2h_2)=\sum_{i=0}^{p^{c_k}-1}r_1h_1\chi(i)\varphi^{c_k}(r_{i,2})h_2=\sum_{i=0}^{p^{c_k}-1}r_1h_1X^i\varphi^{c_k}(r_{i,2})h_2=\\
=\sum_{i=0}^{p^{c_k}-1}r_1X^i\sigma^i(h_1)\varphi^{c_k}(r_{i,2})h_2=\sum_{i=0}^{p^{c_k}-1}r_1\chi(i)\varphi^{c_k}(r_{i,2})\left((\iota(i)^{-1}h_1\iota(i))h_2\right)
\end{align*}
since $\chi(i)$ corresponds to $X^i$ under the isomorphism in Lemma \ref{polynomial}.
\end{proof}

We further have a natural action of $\varphi$ on
$R[H_1/H_k,\ell,\iota]$ coming from the $\varphi$-action on both $R$ and
$H_1/H_k$ by putting $\varphi(rh):=\varphi(r)\varphi(h)$ for $r\in R$ and $h\in H_1/H_k$.

\begin{lem}
The map $\varphi\colon R[H_1/H_k,\ell,\iota]\to R[H_1/H_k,\ell,\iota]$ defined above is a ring homomorphism.
\end{lem}
\begin{proof}
The additivity is clear, so it suffices to check the multiplicativity. Using \eqref{mult} we compute
\begin{align*}
\varphi((r_1h_1)(r_2h_2))=\sum_{i=0}^{p^{c_k}-1}\varphi\left(r_1\chi(i)\varphi^{c_k}(r_{i,2})\right)\varphi\left((\iota(i)^{-1}h_1\iota(i))h_2\right)=\\
=\sum_{i=0}^{p^{c_k}-1}\varphi(r_1)\chi(pi)\varphi^{c_k+1}(r_{i,2})\left((\iota(pi)^{-1}\varphi(h_1)\iota(pi))\varphi(h_2)\right)=\\
=\sum_{i=0}^{p^{c_k}-1}\varphi(r_1)\varphi^{c_k+1}(r_{i,2})\varphi(h_1)\chi(pi)\varphi(h_2)=\varphi(r_1)\varphi(h_1)\sum_{i=0}^{p^{c_k}-1}\chi(pi)\varphi^{c_k+1}(r_{i,2})\varphi(h_2)=\\
=\varphi(r_1h_1)\varphi\left(\sum_{i=0}^{p^{c_k}-1}\chi(i)\varphi^{c_k}(r_{i,2})h_2\right)=\varphi(r_1h_1)\varphi(r_2h_2)\ .
\end{align*}
\end{proof}

 Moreover, the map $\chi$ and the inclusion of the group $H_1/H_k$
in the multiplicative group of $R[H_1/H_k,\ell,\iota]$ are compatible in the
sense that they glue together to a $\varphi$-equivariant group
homomorphism $\chi_k\colon H_0\to R[H_1/H_k,\ell,\iota]^{\times}$ (with
kernel $\Ker\chi_k=H_k$) making the diagram 
\begin{equation}\label{chicomm}
\xymatrix{
\mathbb{Z}_p\ar[r]_{\iota} \ar[d]_{\chi} &H_0 \ar[d]_{\chi_{k}}\\
R\ar[r]_{\iota_{R,k}} & R[H_1/H_k,\ell]
}
\end{equation}
commutative where $\iota_{R,k}$ is the natural inclusion of $R$ in $R[H_1/H_k,\ell]$. Indeed, $H_0\cong\iota(\mathbb{Z}_p)\ltimes H_1$, so we put
$\chi_k(\iota(i)h):=\chi(i)(hH_k)$ for $i\in\mathbb{Z}_p$, $h\in H_1$ and compute
\begin{align*}
\chi_k(\iota(i_1)h_1\iota(i_2)h_2)=\chi_k(\iota(i_1+i_2)\iota(i_2)^{-1}h_1\iota(i_2)h_2)=\chi(i_1+i_2)(\iota(i_2)^{-1}h_1\iota(i_2)h_2)H_k=\\
=\chi(i_1)(h_1H_k)\chi(i_2)(h_2H_k)=\chi_k(\iota(i_1)h_1)\chi_k(\iota(i_2)h_2)
\end{align*}
showing that $\chi_k$ is indeed a group homomorphism. The commutativity
of the diagram \eqref{chicomm} is clear by definition. Moreover,
$\chi_k$ is $\varphi$-equivariant, since we have  
\begin{equation*}
\chi_k\circ\varphi(\iota(i)h)=\chi_k(\iota(pi)\varphi(h))=\chi(pi)\varphi(h)H_k=\varphi(\chi(i)hH_k)=\varphi\circ\chi_k(\iota(i)h)\ .
\end{equation*}

\begin{lem}
The above definition of $R[H_1/H_k,\ell,\iota]$ does not depend on the choice
of the section $\iota$ up to natural isomorphism.
\end{lem}
\begin{proof}
Let $\iota'\colon\mathbb{Z}_p\hookrightarrow H_0$ be another section of
$\ell$. Note that the integer $c_k$ depends on $\iota$ but we also have another integer $c_k'$ such that $\iota'(p^{c_k'})$ acts trivially by conjugation on $H_1/H_k$, ie.\ $\iota'(p^{c_k'})$ lies in the centre of $H_0/H_k$. On the other hand, we may choose $m_k\geq 0$ so that $H_1^{p^{m_k}}\subseteq H_k$ since $H_1/H_k$ is a finite $p$-group. From $\ell\circ\iota=\id_{\mathbb{Z}_p}=\ell\circ\iota'$ we see that
$\iota^{-1}\iota'(\mathbb{Z}_p)\subseteq\Ker(\ell)=H_1$ hence for any $x\in\mathbb{Z}_p$ we have 
\begin{align*}
\iota^{-1}(p^{m_k+\max(c_k,c_k')}x)\iota'(p^{m_k+\max(c_k,c_k')}x)= \iota^{-1}(p^{\max(c_k,c_k')}x)^{p^{m_k}}\iota'(p^{\max(c_k,c_k')}x)^{p^{m_k}}=\\ 
=\left(\iota^{-1}(p^{\max(c_k,c_k')}x)\iota'(p^{\max(c_k,c_k')}x)\right)^{p^{m_k}}\in H_1^{p^{m_k}}\subseteq H_k\ . 
\end{align*}
Therefore for $m\geq m_k+\max(c_k,c_k')$ the map 
\begin{eqnarray}
\iota_k'\colon R&\hookrightarrow& R[H_1/H_k,\ell,\iota]\notag\\
\iota'_k\left(\sum_{i=0}^{p^m-1}\chi(i)\varphi^m(r_i)\right)&:=&\sum_{i=0}^{p^m-1}\chi(i)\varphi^m(r_i)(\iota(i)^{-1}\iota'(i))\label{iota'}
\end{eqnarray}
extends to an isomorphism
\begin{eqnarray*}
\iota_k'\colon R[H_1/H_k,\ell,\iota']&\rightarrow& R[H_1/H_k,\ell,\iota]\\
rh&\mapsto&\iota'_k(r)h
\end{eqnarray*}
of $\varphi$-rings. Indeed, the map $\iota'_k$ is clearly additive and bijective.
We claim that it is multiplicative and $\varphi$-equivariant. We first
 show the latter statement and compute 
\begin{align*}
\varphi\circ\iota'_k\left(\sum_{i=0}^{p^m-1}\chi(i)\varphi^m(r_i)\right)=\sum_{i=0}^{p^m-1}\varphi(\chi(i)\varphi^m(r_i)(\iota(i)^{-1}\iota'(i)))=\\
\sum_{i=0}^{p^m-1}\chi(pi)\varphi^{m+1}(r_i)(\iota(pi)^{-1}\iota'(pi))
=\iota'_k\left(\sum_{i=0}^{p^m-1}\chi(pi)\varphi^{m+1}(r_i)\right)=\iota'_k\circ\varphi\left(\sum_{i=0}^{p^m-1}\chi(i)\varphi^m(r_i)\right)\ .
\end{align*} 

Note that since $m\geq \max(c_k,c_k')$ the subring
$\varphi^m(R)$ lies in the centre of both $R[H_1/H_k,\ell,\iota]$ and
$R[H_1/H_k,\ell,\iota']$. Therefore---in view of the associativity
 (Lemma \ref{lemmamult})---we may compute the multiplication
 \eqref{mult} by expanding elements of $R$ to degree $m$. So we write 
\begin{equation*}
r_1=\sum_{j=0}^{p^m-1}\chi(j)\varphi^m(r_{j,1})\ ,\quad
 r_2=\sum_{i=0}^{p^m-1}\chi(i)\varphi^m(r_{i,2})\ .
\end{equation*}
Moreover, we may compute \eqref{iota'} using 
any set of representatives of $\mathbb{Z}_p/p^m\mathbb{Z}_p$ (e.g.\
 $\{j,j+1,\dots,j+p^m-1\}$ instead of $\{0,1,\dots,p^m-1\}$) since
$\iota^{-1}\iota'(p^m\mathbb{Z}_p)\subseteq H_k$. Hence we obtain
\begin{align*}
\iota'_k\left((r_1h_1)(r_2h_2)\right)=\iota_k'\left(\sum_{i=0}^{p^m-1}r_1\chi(i)\varphi^m(r_{i,2})(\iota'(i)^{-1}h_1\iota'(i)h_2)\right)=\\
=\iota_k'\left(\sum_{i,j=0}^{p^m-1}\chi(j)\varphi^m(r_{j,1})\chi(i)\varphi^m(r_{i,2})(\iota'(i)^{-1}h_1\iota'(i)h_2)\right)=\\
=\sum_{i=0}^{p^m-1}\iota_k'\left(\sum_{j=0}^{p^m-1}\chi(i+j)\varphi^m(r_{j,1}r_{i,2})\right)\iota'(i)^{-1}h_1\iota'(i)h_2=\\
=\sum_{i,j=0}^{p^m-1}\chi(i+j)\varphi^m(r_{j,1}r_{i,2})\iota(i+j)^{-1}\iota'(i+j)\iota'(i)^{-1}h_1\iota'(i)h_2=\\
=\sum_{i,j=0}^{p^m-1}\chi(i+j)\varphi^m(r_{j,1}r_{i,2})\iota(i+j)^{-1}\iota'(j)h_1\iota'(i)h_2)=\\
=\sum_{i,j=0}^{p^m-1}\chi(j)\varphi^m(r_{j,1})\chi(i)\varphi^m(r_{i,r})\iota(i)^{-1}(\iota(j)^{-1}\iota'(j)h_1)\iota(i)(\iota(i)^{-1}\iota'(i)h_2)=\\
=\left(\sum_{j=0}^{p^m-1}\chi(j)\varphi^m(r_{j,1})(\iota(j)^{-1}\iota'(j)h_1)\right)
\left(\sum_{i=0}^{p^m-1}\chi(i)\varphi^m(r_{i,2})(\iota(i)^{-1}\iota'(i)h_2)\right)=\\
\left(\sum_{j=0}^{p^m-1}\varphi^m(r_{j,1})\iota'(j)h_1\right)
\left(\sum_{i=0}^{p^m-1}\varphi^m(r_{i,2})\iota'(i)h_2\right)=\iota'_k(r_1h_1)\iota_k'(r_2h_2)\ .
\end{align*}
\end{proof}

In view of the above Lemma we omit $\iota$ from the notation from now on. This
construction is compatible with the natural surjective homomorphisms 
$H_1/H_{k+1}\twoheadrightarrow H_1/H_k$ therefore the rings
$R[H_1/H_k,\ell]$ form an inverse system for the induced maps. So we may
define the completed skew group ring $R\bs H_1,\ell\js$ as the
projective limit
\begin{equation*}
R\bs H_1,\ell\js:=\varprojlim_{k}R[H_1/H_k,\ell].
\end{equation*}
We denote by $I_k$ the kernel of the canonical surjective homomorphism from
$R\bs H_1,\ell\js$ to $R[H_1/H_k,\ell]$.

Whenever $R$ is a topological ring we equip $R\bs H_1,\ell\js$ with the
projective limit topology of the product topologies on each $\bigoplus_{h\in H_1/H_k}Rh$.

The augmentation map $H_1\twoheadrightarrow1$ induces a ring homomorphism $\ell:=\ell_R\colon
R\bs H_1,\ell\js\twoheadrightarrow R$. This also has a section $\iota:=\iota_R=\varprojlim\iota_{R,k}\colon R\hookrightarrow R\bs
H_1,\ell\js$ (whenever clear we omit the subscript $_R$), ie.\ $\ell_R\circ\iota_R=\id_R$. Moreover, by \eqref{chicomm} the group homomorphism $\chi\colon
\mathbb{Z}_p\to R^{\times}$ extends to a group homomorphism $\chi_{H_0}\colon H_0\to
R\bs H_1,\ell\js^{\times}$ making the diagram
\begin{equation*}
\xymatrix{
\mathbb{Z}_p\ar[r]_{\iota} \ar[d]_{\chi} &H_0 \ar[d]_{\chi_{H_0}}\\
R\ar[r]_{\iota_R} & R\bs H_1,\ell\js
}
\end{equation*}
commutative.

The operator $\varphi$ acts naturally on this projective limit. If $R$ is a
topological ring and $\varphi$ acts continuously on $R$ then $\varphi$ also
acts continuously on each $R[H_1/H_k,\ell]$ and by taking the limit also on $R\bs H_1,\ell\js$. For an open subgroup $H'$ of a profinite group $H$ we use the notation $J(H/H')$ for a set of representatives of
the left cosets of $H'$ in $H$. Similarly, we use $J(H'\setminus H)$ for a set
of representatives of the right cosets $H'\setminus H$.

\begin{lem}\label{cosets}
\begin{enumerate}[$a)$]
\item Let $L\leq K\leq H$ be groups. Then the set $J(H/K)J(K/L)$ (resp.\ $J(L\setminus K)J(K\setminus H)$) is a set of representatives for the cosets $H/L$ (resp.\ for $L\setminus H$).
\item Let $K\leq H$ be groups and $N\lhd H$ a \emph{normal} subgroup. Then $J((K\cap N)\setminus N)$ is also a set of representatives for $K\setminus KN$. 
\end{enumerate}
\end{lem}
\begin{proof}
These are well-known facts in group theory, however, for the convenience
 of the reader we recall their proofs here. Note that in $b)$ we need
 $N$ to be a normal subgroup so that $KN$ is a subgroup of $H$. Also
 note that $J(K\setminus KN)$ might not lie in $N$ in general.

$a)$ Let $h_1,h_2\in J(H/K)$ and $k_1,k_2\in J(K/L)$. Suppose we have
 $h_1k_1L=h_2k_2L$. Then we also have $h_1^{-1}h_2\in K$ so $h_1=h_2$,
 whence $k_1^{-1}k_2\in L$ so $k_1=k_2$. So the elements of the set
 $J(H/K)J(K/L)$ are in distinct left cosets of $L$. On the other hand,
 if $hL\in H/L$ is a left coset, then we may first choose $h_1\in
 J(H/K)$ so that $h_1^{-1}h\in K$ and then $k_1\in J(K/L)$ so that
 $k_1^{-1}h_1^{-1}h\in L$, ie.\ $hL=h_1k_1L$.

$b)$ If $n_1\neq n_2\in J((K\cap N)\setminus N)$ are distinct then
 $Kn_1\neq Kn_2$ as $n_1n_2^{-1}$ does not lie in $K\cap N$, but it lies
 in $N$. On the other hand, if $kn\in KN$ then we may find $n_1\in
 J((K\cap N)\setminus N)$ such that $nn_1^{-1}\in K\cap N$ hence
 $knn_1^{-1}\in K$.
\end{proof}

\begin{pro}
The map $\varphi\colon R\bs H_1,\ell \js\rightarrow R\bs
H_1,\ell \js$ is injective. Moreover, we have
\begin{equation*}
R\bs H_1,\ell \js=\bigoplus_{h\in J(\varphi(H_0)\setminus H_0)}\varphi(R\bs H_1,\ell \js)h.
\end{equation*}
In particular, $R\bs H_1,\ell \js$ is a free (left) module of rank
$[H_0:\varphi(H_0)]$ over itself via $\varphi$.
\end{pro}
\begin{proof}
\emph{Step 1.} Let $k$ be an integer and denote by $A_k$ the kernel of the map $\varphi\colon
R[H_1/H_k,\ell]\to R[H_1/H_k,\ell]$ so that we have a short exact sequence of
abelian groups
\begin{equation*}
0\longrightarrow A_k \longrightarrow R[H_1/H_k,\ell]
\overset{\varphi}{\longrightarrow} \varphi(R[H_1/H_k,\ell])\longrightarrow 0.
\end{equation*}
We are going to show that the sequence $A_k$ satisfies the trivial
Mittag-Leffler condition. From this the injectivity of $\varphi$ follows, and we obtain 
\begin{equation}
\varprojlim_{k}\varphi(R[H_1/H_k,\ell])\cong \varphi(\varprojlim_k
R[H_1/H_k,\ell])=\varphi(R\bs H_1,\ell\js).\label{lim}
\end{equation}
Take a fixed positive integer $k$. Since $\varphi\colon
H_1\rightarrow H_1$ is an open map (bijective and continuous between the compact sets $H_1$ and $\varphi(H_1)$ hence a homeomorphism) and the subgroups $H_l$ form a system of neighbourhoods we find an integer $l> k$ such that
$H_k\supseteq\varphi^{-1}(H_{l})$. In view of Lemma \ref{cosets} we put 
\begin{equation*}
J(H_1/H_l):=J(H_1/\varphi^{-1}(H_l))J(\varphi^{-1}(H_l)/H_l) 
\end{equation*}
for $J(H_1/\varphi^{-1}(H_l))$ and
$J(\varphi^{-1}(H_l)/H_l)$ arbitrarily fixed sets of representatives for
 the cosets of $H_1/\varphi^{-1}(H_l)$ and of $\varphi^{-1}(H_l)/H_l$, respectively.

Let now $\sum_{h\in J(H_1/H_{l})}r_h\chi_l(h)$ be an
element in $A_l$ and denote by $f_{k,l}$ the natural surjection from
$R[H_1/H_l,\ell]\twoheadrightarrow R[H_1/H_k,\ell]$. We have
\begin{align*}
0=\varphi\left(\sum_{h\in J(H_1/H_l)}r_h\chi_l(h)\right)=
\sum_{h_1\in
  J(H_1/\varphi^{-1}(H_l))}\sum_{h_2\in
  J(\varphi^{-1}(H_l)/H_l)}\varphi(r_{h_1h_2})\chi_l(\varphi(h_1h_2))=\\
=\sum_{h_1}\sum_{h_2}\varphi(r_{h_1h_2})\chi_l(\varphi(h_1))=\sum_{h_1\in  J(H_1/\varphi^{-1}(H_l))}\varphi\left(\sum_{h\in
J(H_1/H_l)\cap h_1\varphi^{-1}(H_l)}r_h\right)\chi_l(\varphi(h_1)).
\end{align*}
Note that for $h_1\neq h_1'\in J(H_1/\varphi^{-1}(H_l))$ we have
$\varphi(h_1)H_l\neq \varphi(h_1')H_l$. Since $R[H_1/H_l,\ell]$ is
 defined as a direct sum, we obtain 
\begin{equation*}
\varphi\left(\sum_{h\in J(H_1/H_l)\cap h_1\varphi^{-1}(H_l)}r_h\right)=0\
 \text{, whence}\ 
\sum_{h\in J(H_1/H_l)\cap h_1\varphi^{-1}(H_l)}r_h=0 
\end{equation*}
for any fixed $h_1\in J(H_1/\varphi^{-1}(H_l))$ as $\varphi$ is
 injective on $R$. On the other hand, we have
\begin{align*}
f_{k,l}\left(\sum_{h\in J(H_1/H_l)}r_h\chi_l(h)\right)
=\sum_{h_1\in J(H_1/H_k)}(\sum_{h\in
  J(H_1/H_l)\cap h_1H_k}r_h)\chi_k(h_1)=\sum_{h_1\in J(H_1/H_k)}0\chi_k(h_1)=0
\end{align*}
as $h_1H_k$ is a disjoint union of cosets of $\varphi^{-1}(H_l)$ by the choice
of $l$. This shows that $f_{k,l}(A_{l})=0$ as claimed. Therefore \eqref{lim} follows as discussed above.

\emph{Step 2.} Since $\varphi(H_0)\cap H_1$ is open in $H_1$, there exists an integer
$k_0\geq 2$ such that for $k\geq k_0$ we have $H_k\subseteq\varphi(H_0)$. (We
remark here that we may not be able to take $k_0=2$ because $H_k$ is the \emph{normal} subgroup
generated by $\varphi(H_1)$ which does have elements outside $\varphi(H_0)$ in
general.) We claim now the decomposition
\begin{equation}\label{flat}
R[H_1/H_k,\ell ]=\bigoplus_{h\in J(\varphi(H_0)\setminus H_0)}\varphi(R[H_1/H_k,\ell ])\chi_k(h)
\end{equation}
for $k\geq k_0$. Note that since $H_k$ is a normal subgroup of $H_0$
contained in $\varphi(H_0)$ the elements $\chi_k(h)$ above are distinct. 

For the proof of \eqref{flat} we apply Lemma \ref{cosets} $b)$ in the situation
$K:=\varphi(H_0)$, $N:=H_1$, and $H:=H_0$ to be able to choose
 $J(\varphi(H_0)\setminus\varphi(H_0)H_1):=J((\varphi(H_0)\cap
 H_1)\setminus H_1)$. Moreover, by the injectivity of $\varphi$ on
 $H_0/H_1$ we see that $\varphi(H_0)\cap H_1=\varphi(H_1)$. On the other
 hand, $\iota(\{0,1,\dots,p-1\})$ is a set of representatives for the
 cosets $H_1\varphi(H_0)\setminus H_0$. Therefore (using Lemma \ref{cosets} $a)$ with  $L:=\varphi(H_0)$,
 $K:=\varphi(H_0)H_1$, and $H:=H_0$) we may choose
\begin{equation*}
J(\varphi(H_0)\setminus H_0):=J(\varphi(H_0)\setminus
 \varphi(H_0)H_1)J(\varphi(H_0)H_1\setminus H_0)=J(\varphi(H_1)\setminus
 H_1)\iota(\{0,1,\dots,p-1\})\ .
\end{equation*}

We are going to use this specific set $J(\varphi(H_0)\setminus H_0)$ in
 order to compute the right hand side of \eqref{flat}. Let $\sum_{h\in
  J(H_1/H_k)}r_h\chi_k(h)$ be an arbitrary element in $R[H_1/H_k,\ell]$. By the
\'etaleness of the action of $\varphi$ on $R$ (noting that $R$ is
 commutative) we may uniquely decompose
\begin{equation*}
r_h=\sum_{i=0}^{p-1}\chi(i)\varphi(r_{i,h})=\sum_{i=0}^{p-1}\varphi(r_{i,h})\chi(i)\ .
\end{equation*}
On the other hand, we write $\iota(i)h\iota(i)^{-1}=\varphi(u_{i,h})v_{i,h}$ with unique $u_{i,h}\in H_1$
and $v_{i,h}\in J(\varphi(H_1)\setminus H_1)$. Therefore we have
\begin{align*}
\sum_{h\in J(H_1/H_k)}r_h\chi_k(h)=
\sum_{h\in
  J(H_1/H_k)}\sum_{i=0}^{p-1}\varphi(r_{i,h})\chi(i)\chi_k(\iota(i)^{-1}\varphi(u_{i,h})v_{i,h}\iota(i))=\\
=\sum_{h\in
  J(H_1/H_k)}\sum_{i=0}^{p-1}\varphi(r_{i,h}\chi_k(u_{i,h}))\chi_k(v_{i,h}\iota(i))\in
\sum_{h\in J(\varphi(H_0)\setminus
 H_0)}\varphi(R[H_1/H_k,\ell])\chi_k(h)
\end{align*}
as $\chi(i)=\chi_k(\iota(i))$ and
 $\chi_k\circ\varphi=\varphi\circ\chi_k$ by \eqref{chicomm}. 

It remains to show that the sum in \eqref{flat} is indeed direct. For this
we may expand any element $x_{i,h}\in R[H_1/H_k,\ell]$ as
\begin{equation*}
x_{i,h}=\sum_{m\in J(H_1/H_k)}r_{i,h,m}\chi_k(m)
\end{equation*}
and compute
\begin{align}
\sum_{i=0}^{p-1}\sum_{h\in J(\varphi(H_1)\setminus
  H_1)}\varphi(x_{i,h})\chi_k(h\iota(i))=\sum_{i,h,m}\varphi(r_{i,h,m})\chi_k(\varphi(m)h\iota(i))=\label{unique}\\
=\sum_{i,h,m}\chi(i)\varphi(r_{i,h,m})\chi_k(\iota(i)^{-1}\varphi(m)h\iota(i))=\notag\\
=\sum_{i,h}\sum_{m_0\in J(\varphi(H_1)/H_k)}\chi(i)\left(\sum_{m\in J(H_1/H_k)\cap\varphi^{-1}(m_0H_k)}\varphi(r_{i,h,m})\right)\chi_k(\iota(i)^{-1}m_0h\iota(i)).\notag
\end{align}
Assume now that the left hand side of \eqref{unique} is 0. The set
$J(\varphi(H_1)/H_k)J(\varphi(H_1)\setminus H_1)$ is a set of representatives
of $H_k\setminus H_1$ because $H_k$ is normal in $H_1$ whence
 $\varphi(H_1)/H_k=H_k\setminus \varphi(H_1)$. This shows that the
 elements $m_0h$ are distinct in $H_1/H_k$ on the right hand side of
 \eqref{unique}. Moreover, the conjugation by $\iota(i)$ is an
 automorphism of $H_1/H_k$ therefore the elements
 $\iota(i)^{-1}m_0h\iota(i)$ are also distinct for any fixed
 $i\in\{0,1,\dots,p-1\}$. On the other hand, by the \'etaleness of  
$\varphi$ on $R$ and by \eqref{defR[H_1/H_k]} we obtain
\begin{equation*}
R[H_1/H_k,\ell]=\bigoplus_{i=0}^{p-1}\bigoplus_{h_1\in
 H_1/H_k}\chi(i)\varphi(R)h\ .
\end{equation*}
Hence we have
\begin{equation*}
\sum_{m\in J(H_1/H_k)\cap\varphi^{-1}(m_0H_k)}\varphi(r_{i,h,m})=0
\end{equation*}
for any fixed $m_0$, $i$, and $h$. In particular, we also have
\begin{align*}
\varphi(x_{i,h})=\sum_{m\in
  J(H_1/H_k)}\varphi(r_{i,h,m})\chi_k(\varphi(m))=\\
=\sum_{m_0\in J(\varphi(H_1)/H_k)}\left(\sum_{m\in J(H_1/H_k)\cap\varphi^{-1}(m_0H_k)}\varphi(r_{i,h,m})\right)\chi_k(m_0)=0
\end{align*}
showing that the sum in \eqref{flat} is direct.

\emph{Step 3.} The result follows by taking the projective limit of \eqref{flat} using \eqref{lim}.
\end{proof}
\begin{rem}
The above lemma holds for replacing left and right, as well, ie.\ we also have
\begin{equation*}
R\bs H_1,\ell\js=\bigoplus_{h\in J(H_0/\varphi(H_0))}h\varphi(R\bs H_1,\ell\js).
\end{equation*}
\end{rem}

Let $S$ be a (not necessarily commutative) ring (with identity) with the following properties: 
\begin{enumerate}[$(i)$]
\item There exists a group homomorphism $\chi\colon H_0\hookrightarrow
  S^{\times}$.
\item The ring $S$ admits an \'etale action of the $p$-Frobenius $\varphi$
  that is compatible with $\chi$. More precisely there is an injective ring
  homomorphism $\varphi\colon S\hookrightarrow S$ such that
  $\varphi(\chi(x))=\chi(\varphi(x))$ and
\begin{equation*}
S=\bigoplus_{h\in \varphi(H_0)\setminus H_0}\varphi(S)\chi(h)=\bigoplus_{h\in H_0/\varphi(H_0)}\chi(h)\varphi(S)\ .
\end{equation*}
In particular, $S$ is free of rank $|H_0:\varphi(H_0)|$ as a left, as well as a right module over $\varphi(S)$.
\end{enumerate}
\begin{df}\label{phiringH_0}
We call a ring $S$ with the above properties $(i)$ and $(ii)$ a $\varphi$-ring
over $H_0$. 
\end{df}
\begin{cor}
The map $R\mapsto R\bs N_1,\ell\js$ is a functor from the category of $\varphi$-rings over $\mathbb{Z}_p$ to the category of $\varphi$-rings over $H_0$.
\end{cor}

\begin{rem}\label{pronilp}
We have $\varphi(I_k)\subseteq I_{k+1}$ for all $k\geq 1$.
\end{rem}
\begin{proof}
Take $x\in I_k$ and write $x+I_{k+1}\in R[H_1/H_{k+1},\ell]$ as
 $x+I_{k+1}=\sum_{h\in J(H_1/H_{k+1})}r_h\chi_{k+1}(h)$. Since $x\in
 I_k$ we have 
\begin{equation*}
0=\sum_{h\in J(H_1/H_{k+1})}r_h\chi_{k}(h)=\sum_{h_1\in
 J(H_1/H_k)}\sum_{h\in J(H_1/H_{k+1})\cap h_1H_k}r_h\chi_k(h_1)
\end{equation*}
hence $\sum_{h\in J(H_1/H_{k+1})\cap h_1H_k}r_h=0$ for any fixed $h_1\in
 J(H_1/H_k)$. So we compute
\begin{align*}
\varphi(x)+I_{k+1}=\sum_{h_1\in
 J(H_1/H_k)}\sum_{h\in J(H_1/H_{k+1})\cap
 h_1H_k}\varphi(r_h)\varphi(\chi_k(h))=\\ 
=\sum_{h_1\in
 J(H_1/H_k)}\sum_{h\in J(H_1/H_{k+1})\cap
 h_1H_k}\varphi(r_h)\chi_k(\varphi(h_1))=\sum_{h_1\in
 J(H_1/H_k)}0\chi_k(\varphi(h_1))=0
\end{align*}
since $\varphi(H_k)\subseteq H_{k+1}$ whence
 $\varphi(h_1)=\varphi(h)$ above.
\end{proof}

Recall that Fontaine's ring $\mathcal{O_E}:=\varprojlim_{n}(o\bs
T\js[T^{-1}])/p_K^n$ is defined as the $p$-adic completion of the ring of
formal Laurent-series over $o$. It is a complete discrete valuation ring with
maximal ideal $p_K\mathcal{O_E}$, residue field $k\bg T\jg$, and field of
fractions $\mathcal{E}=\mathcal{O_E}[p_K^{-1}]$. We are going to show
that the completed skew group ring
$\mathcal{O_E}\bs H_1,\ell\js$ is isomorphic to the previously constructed
(\cite{SVe}, see also section 8 of \cite{SVi}, \cite{SVZ}, \cite{Z})
microlocalized ring $\Lambda_\ell(H_0)$ of the Iwasawa algebra
$\Lambda(H_0)$. ($H_0=N_0$ in the notations of \cite{SVi}, \cite{SVZ}, and
\cite{Z}.) For the convenience of the reader we recall the definition 
here. Let $\Lambda(H_0):=o\bs H_0\js$ be the Iwasawa algebra of the pro-$p$
group $H_0$. It is shown in \cite{CFKSV} that
$S:=\Lambda(H_0)\setminus(p_K,H_1-1)$ is a left and right Ore set in
$\Lambda(H_0)$ so that the localization $\Lambda(H_0)_S$ exists. The ring
$\Lambda_\ell(H_0)$ is defined as the $(p_K,H_1-1)$-adic completion of
$\Lambda(H_0)_S$ (the so called ``microlocalization''). Note that since $\varphi\colon H_0\to H_0$ is a continuous group homomorphism, it induces a continuous ring homomorphism $\varphi\colon \Lambda(H_0)\to\Lambda(H_0)$ of the Iwasawa algebra. Moreover, since $\varphi(S)\subset S$, $\varphi$ extends to a ring homomorphism $\varphi\colon\Lambda(H_0)_S\to\Lambda(H_0)_S$ and by continuity to its completion $\Lambda_\ell(H_0)$ (see section 8 of \cite{SVi} for more details).

\begin{rem}\label{remark}
Let $R$ be a $\varphi$-ring containing (as a $\varphi$-subring) the Iwasawa
algebra $o\bs T\js\cong\Lambda(\mathbb{Z}_p)$. Then using \eqref{iterate} we compute
\begin{align*}
R[H_1/H_k,\ell ] \cong (R\otimes_{\Lambda(\mathbb{Z}_p)}\Lambda(\mathbb{Z}_p))[H_1/H_k,\ell]\cong
R\otimes_{\Lambda(\mathbb{Z}_p)}(\Lambda(\mathbb{Z}_p)[H_1/H_k,\ell])\\
\cong R\otimes_{\Lambda(\mathbb{Z}_p),\iota}\Lambda(H_0/H_k)\cong
(\varphi^{c_k}(R)\otimes_{\varphi^{c_k}(\Lambda(\mathbb{Z}_p))}\Lambda(\mathbb{Z}_p))\otimes_{\Lambda(\mathbb{Z}_p),\iota}\Lambda(H_0/H_k)\\
\cong\varphi^{c_k}(R)\otimes_{\varphi^{c_k}(\Lambda(\mathbb{Z}_p)),\iota}\Lambda(H_0/H_k)
\end{align*}
for any $k\geq 1$.
\end{rem}

\begin{lem}\label{isom}
We have a $\varphi$-equivariant ring-isomorphism $\mathcal{O_E}\bs
H_1,\ell\js\cong\Lambda_\ell(H_0)$. 
\end{lem}
\begin{proof}
The ring $\Lambda_\ell(H_0)$ is complete and Hausdorff with respect to the
filtration by the ideals generated by $(H_k-1)$ since these ideals are
closed with intersection zero in the pseudocompact ring $\Lambda_\ell(H_0)$
(cf.\ Thm.\ 4.7 in \cite{SVe}). So it remains to show that
$\Lambda_\ell(H_0/H_k)$ is naturally isomorphic to the skew group ring
$\mathcal{O_E}[H_1/H_k,\ell]$. At first we show that
 $\Lambda(H_0/H_k)\cong\Lambda(\mathbb{Z}_p)[H_1/H_k,\ell]$. Both sides
 are free modules of rank $|H_1/H_k|$ over $\Lambda(\mathbb{Z}_p)$ with
 generators $h\in H_1/H_k$ so there is an obvious isomorphism between
 them as $\Lambda(\mathbb{Z}_p)$-modules. Moreover,
 $\varphi^{c_k}(\Lambda(\mathbb{Z}_p))$ lies in the centre of both
 rings. However, the obvious map above is also 
 multiplicative since the multiplication on
 $\Lambda(\mathbb{Z}_p)[H_1/H_k,\ell]$ is uniquely determined by
 \eqref{mult} so that \eqref{chicomm} is satisfied and
 $\varphi^{c_k}(\Lambda(\mathbb{Z}_p))$ lies in the centre. 

Now by Remark \ref{remark} we have
\begin{align*}
\mathcal{O_E}[H_1/H_k,\ell ]\cong\varphi^{c_k}(\mathcal{O_E})\otimes_{\varphi^{c_k}(\Lambda(\mathbb{Z}_p)),\iota}\Lambda(H_0/H_k)
\end{align*}
for any $k\geq 1$.

Since $\iota(\varphi^{p^{c_k}}(\Lambda(\mathbb{Z}_p)))$ lies in the
 centre of $\Lambda(H_0/H_k)$, the right hand side above is the
 localisation of $\Lambda(H_0/H_k)$ inverting the central element
 $\varphi^{p^{c_k}}(T)$ and taking the $p$-adic completion
 afterwards (ie.\ ``microlocalisation'' at $\varphi^{p^{c_k}}(T)$). However, in a $p$-adically complete ring $T$ is invertible
 if and only if so is $\varphi^{p^{c_k}}(T)$. Indeed, we have
\begin{equation*} 
T\mid\varphi^{p^{c_k}}(T)=(T+1)^{p^{c_k}}-1=\sum_{i=1}^{p^{c_k}}\binom{p^{c_k}}{i}T^i\in
 T^{p^{c_k}}(1+po\bs T\js[T^{-1}])\ .
\end{equation*}
Hence we obtain
\begin{equation*}
\varphi^{c_k}(\mathcal{O_E})\otimes_{\varphi^{c_k}(\Lambda(\mathbb{Z}_p)),\iota}\Lambda(H_0/H_k)\cong\Lambda_\ell(H_0/H_k)
\end{equation*}
as both sides are the microlocalisation of $\Lambda(H_0/H_k)$ at $T$.
\end{proof}

\section{Equivalence of categories}\label{equiv}

Let $S$ be a $\varphi$-ring over any pro-$p$ group $H_0$ satisfying $(A)$, $(B)$, and $(C)$ (for now it would suffice to assume that $S$ has an injective ring-endomorphism $\varphi\colon S\to S$). We define a $\varphi$-module over $S$ to be a free $S$-module $D$ of finite rank together with a semilinear action of $\varphi$ such that the map
\begin{eqnarray}
1\otimes\varphi\colon S\otimes_{S,\varphi}D&\rightarrow& D\notag\\
r\otimes d&\mapsto& r\varphi(d)\label{etaledef}
\end{eqnarray}
is an isomorphism. Note that for rings $S$ in which $p$ is not invertible
(such as $S=\mathcal{O_E}$ and $\mathcal{O}^\dagger_{\mathcal{E}}$) this is the definition of an
\emph{\'etale} $\varphi$-module. However, for rings in which $p$ is invertible
(such as the Robba ring $\mathcal{R}$) this is the usual definition of a
$\varphi$-module. We use this definition for both $S=R$ and $S=R\bs H_1,\ell\js$---the former being a $\varphi$-ring over $\mathbb{Z}_p$ and the latter being a $\varphi$-ring over $H_0$. We denote the category of $\varphi$-modules over $R$ (resp.\ over $R\bs H_1,\ell\js$) by $\mathfrak{M}(R,\varphi)$ (resp.\ by $\mathfrak{M}(R\bs H_1,\ell\js,\varphi)$). These are clearly additive categories. However, they are not abelian in general, as the kernel and cokernel might not be a free module over $R$, resp.\ over $R\bs H_1,\ell\js$.

Note that for modules $M$ over $R\bs H_1,\ell\js$ \eqref{etaledef} (with $D=M$) is equivalent to saying that each element $m\in
M$ is uniquely decomposed as
\begin{equation*}
m=\sum_{u\in J(H_0/\varphi^k(H_0))}u \varphi^{k}(m_{u,k})
\end{equation*}
for $k=1$, or equivalently, for all $k\geq 1$.

There is an obvious functor in both directions induced by $\ell_R$ and $\iota_R$
that we denote by
\begin{eqnarray*}
\mathbb{D}:=R\otimes_{R\bs H_1,\ell\js,\ell}\cdot\colon&\mathfrak{M}(R\bs
H_1,\ell\js,\varphi)&\rightarrow\mathfrak{M}(R,\varphi)\\
\mathbb{M}:=R\bs
H_1,\ell\js\otimes_{R,\iota}\cdot\colon&\mathfrak{M}(R,\varphi)&\rightarrow\mathfrak{M}(R\bs
H_1,\ell\js,\varphi)\ .
\end{eqnarray*}

The following is a generalization of Thm.\ 8.20 in \cite{SVZ}. The proof is also
similar, but we include it here for the convenience of the reader.

\begin{pro}\label{equivcat}
The functors $\mathbb{D}$ and $\mathbb{M}$ are quasi-inverse equivalences of categories.
\end{pro}
\begin{proof}
We first note that since $\ell\circ\iota=\id_{R}$ we also have
$\mathbb{D}\circ\mathbb{M}\cong\id_{\mathfrak{M}(R,\varphi)}$. So it remains
to show that $\mathbb{D}$ is full and faithful.

For the faithfulness of $\mathbb{D}$ let $f\colon M_{1}\to M_{2}$ be a
morphism in  ${\mathfrak M}(R\bs H_1,\ell\js,\varphi)$ such that
$\mathbb{D}(f) =0 $ which means that $f(M_1)\subseteq I_1M_2$.
Let $m\in M_{1}$. For any $k\in \mathbb N$ we write $m=\sum_{u\in J(H_0/\varphi^k(H_0))}u \varphi^{k}(m_{u,k})$ and
\begin{equation*}
f(m)=\sum_{u\in J(H_0/\varphi^k(H_0))}u \varphi^{k}f(m_{u,k})\in
\varphi^{k}(I_1M_{2})\subseteq I_{k+1}M_2
\end{equation*}
by Remark \ref{pronilp}. Therefore $f(M_1)\subseteq I_{k+1}M_2$ for any $k\geq0$ and therefore $f=0$
since $M_2$ is a finitely generated free module over $R\bs H_1,\ell\js$ and
$\bigcap_{k\geq 0}I_{k+1}=0$ since $R\bs H_1,\ell\js\cong\varprojlim R\bs H_1,\ell\js/I_k$.

Now we prove that for any object $M$ in $\mathfrak{M}(R\bs
H_1,\ell\js,\varphi)$ we have an isomorphism
$\mathbb{M}\circ\mathbb{D}(M)\rightarrow M$. We start with an arbitrary finite
$R\bs H_1,\ell\js$-basis $(\epsilon_{i})_{1\leq i\leq d}$ of $M$ (where $d$ is
the rank of $M$). As $R$-modules we have
\begin{equation*}
 M=(\bigoplus_{1\leq i\leq d} \iota (R)\epsilon_{i})\oplus
 (\bigoplus_{1\leq i\leq d} I_1\epsilon_{i}) \ .
\end{equation*}
 It is clear that the $R\bs H_1,\ell\js$-linear map from $M$ to
 $\mathbb{M}(\mathbb{D}(M))$ sending $\epsilon_{i}$ to $1\otimes (1\otimes
 \epsilon_{i})$ is bijective. It is $\varphi$-equivariant if and only if $\bigoplus_{1\leq
   i\leq d} \iota (R)\epsilon_{i}$ is $\varphi$-stable which is, of course,
 not true in general. We always have
\begin{equation*}      
\varphi(\epsilon_{i})=\sum_{1\leq j\leq d}(a_{i,j} +b_{i,j}) \epsilon_{j} \ \ {\rm where } \ a_{i,j}
\in \iota (R) \ , \ b_{i,j}\in I_1  \ .
\end{equation*}
   If  the $b_{i,j}$ are not all $0$,
   we will find elements $ x_{i,j}\in I_1$  such that
\begin{equation*}  
\eta_{i} := \epsilon_{i}     +\sum_{1\leq j\leq d}  x_{i,j} \epsilon_{j}
\end{equation*}
  satisfies
\begin{equation*}
\varphi (\eta_{i}) = \sum_{1\leq j\leq d}a_{i,j}\eta_{j} \ \ {\rm for } \ i\in I  \ .
\end{equation*}
The conditions on the matrix    $X:=(x_{i,j})_{1\leq i,j\leq d}$  are :
\begin{equation*}    
\varphi (\id+X) (A+B) = A (\id +X) \   
\end{equation*}
   for the matrices $A:=(a_{i,j})_{1\leq i,j\leq d} \ , B:=(b_{i,j})_{1\leq
     i,j\leq d}$. The coefficients  of  $A$ belong to the commutative ring
   $\iota (R) $. The matrix $A+B$ is invertible because the 
 $R\bs H_1,\ell\js$-endomorphism $f$ of $M$ defined by
\begin{equation*}
  f(\epsilon_{i})=\varphi (\epsilon_{i})  \ \ {\rm for } \ 1\leq i\leq d \ ,
\end{equation*}
   is an automorphism of $M$ as $M$ lies in $\mathfrak{M}(R\bs H_1,\ell\js,\varphi)$. Therefore the matrix $A=\ell(A+B)$ is also invertible. We are reduced to solve the equation
\begin{equation*}
  A^{-1}B+A^{-1} \varphi (X) (A+B ) =    X
\end{equation*}
 in the indeterminate $X$.
We are looking for the solution $X$ in the form of an infinite sum
\begin{equation*}
X= A^{-1}B +  \ldots + (A^{-1}\varphi  (A^{-1})\ldots   \varphi ^{k-1}(A^{-1})
\varphi ^{k}(A^{-1}B)  \varphi ^{k-1}(A+B) \ldots  \varphi  (A+B)(A+ B) )+
\ldots  \ .
\end{equation*}
 The coefficients of $A^{-1}B $ belong to the two-sided ideal $I_1$ of $R\bs H_1,\ell\js$ and the coefficients of the $k$-th term of the series
\begin{equation*}
(A^{-1}\varphi  (A^{-1})\ldots   \varphi ^{k-1}(A^{-1}) \varphi ^{k}(A^{-1}B)
  \varphi ^{k-1}(A+B) \ldots  \varphi  (A+B)(A+ B) ) 
\end{equation*}
belong to  $\varphi^{k}(I_1)\subseteq I_{k+1}$.  Hence the series converges
since $R\bs H_1,\ell\js\cong\varprojlim_k R\bs H_1,\ell\js/I_k$. Its limit $X$
is the unique solution of the equation. The coefficients of every term in the
series belong to $I_1$ and  $I_1$ is closed in $R\bs H_1,\ell\js$, hence
$x_{i,j}\in I_1$ for $1\leq i,j\leq d$. 

We still need to show that the set $(\eta_{i})_{1\leq i\leq d}$ is an $R\bs
H_1,\ell\js$-basis of $M$. Similarly to the above equation we may find a
matrix $Y$ with coefficients in $I_1$ such that we have
\begin{equation*}
(A+B)(\id+Y)=\varphi(\id+Y)A.
\end{equation*}
Therefore we obtain
\begin{equation*}
(A+B)(\id+Y)(\id+X)=\varphi\left((\id+Y)(\id+X)\right)(A+B)
\end{equation*}
which means that the map
\begin{eqnarray*}
(\id+Y)(\id+X)\colon M&\rightarrow&M\\
\epsilon_i&\mapsto& (\id+Y)(\id+X)\epsilon_i
\end{eqnarray*}
is a $\varphi$-equivariant map such that $\mathbb{D}((\id+Y)(\id+X))=\id$,
hence $(\id+Y)(\id+X)=\id$ by the faithfulness of $\mathbb{D}$. By a similar
computation we also obtain
$A(\id+X)(\id+Y)=\varphi\left((\id+X)(\id+Y)\right)A$ showing that
$(\id+X)(\id+Y)$ is a $\varphi$-equivariant endomorphism of
$\mathbb{M}\circ\mathbb{D}(M)$ reducing to the identity modulo $I_1$. Hence
$(\id+Y)$ is a twosided inverse to the map $(\id+X)$, in particular
$(\eta_i)_{1\leq i\leq d}$ is an $R\bs H_1,\ell\js$-basis of $M$. So we obtain
an isomorphism in $\mathfrak{M}(R\bs H_1,\ell\js,\varphi)$, 
\begin{equation*}
\Theta \ : \ M\to\mathbb  \mathbb{M}(\mathbb{D}(M)) \  \ , \ \ \Theta
(\eta_{i})=1\otimes (1 \otimes \eta_{i}) \ \ \hbox{for} \ 1\leq i\leq d \ ,
\end{equation*}
such that $\mathbb{D}(\Theta)$ is the identity morphism of $\mathbb{D}(M)$.

Now if $f\colon \mathbb{D}(M_1)\rightarrow\mathbb{D}(M_2)$ then for
\begin{equation*}
\mathbb{M}(f)\colon M_1\cong\mathbb{M}\circ\mathbb{D}(M_1)\rightarrow
\mathbb{M}\circ\mathbb{D}(M_2)\cong M_2 
\end{equation*}
we have
$\mathbb{D}\circ\mathbb{M}(f)=f$ therefore $\mathbb{D}$ is full.
\end{proof}

\begin{rem}
There is a small mistake in Lemma 1 of \cite{Z}. The map $\omega$ is in fact
not a $p$-valuation, since assertion $(iii)$ stating that
$\omega(g^p)=\omega(g)+1$ is false. It is only true in the weaker form
$\omega(g^p)\geq \omega(g)+1$. However, this does not influence the validity
of the rest of the paper as $N_{0,n}:=\{g\in N_0\mid \omega(g)\geq n\}$ is
still a subgroup satisfying Lemma 2. Alternatively, it is possible to modify
$\omega$ so that one truely obtains a $p$-valuation. I would like to take this
opportunity to thank Torsten Schoeneberg for pointing this out to me.
\end{rem}
\begin{rem} 
Note that in the case of $R=\mathcal{O_E}$ we may end the proof of
  Proposition \ref{equivcat} by saying that $\id+X$ is invertible since $X$
  lies in $I_1^{d\times d}$ and $\mathcal{O_E}\bs
  H_1,\ell\js\cong\Lambda_\ell(H_0)$ is $I_1$-adically complete. However, in
  the general situation $R\bs H_1,\ell\js$ may not be complete
  $I_1$-adically. The reason for this is the fact that the ideals
  $(I_k)_{k\geq 1}$ are only cofinal with the ideals $I_1^k$ whenever $R$ is
  killed by a power of $p$. Therefore if $R$ is not $p$-adically complete, we
  do not have $R\bs H_1,\ell\js\cong \varprojlim R\bs H_1,\ell\js/I_1^k$ in
  general. Moreover, in case of $R=\mathcal{O_E}$ Proposition \ref{equivcat} holds for not necessarily free modules, as well. See \cite{SVZ} for the proof of this.
\end{rem}
\begin{rem}
The matrix $Y$ in the proof of Prop.\ \ref{equivcat} is given by a
  convergent sum of the terms
\begin{equation*}
-(A+B)^{-1}\varphi((A+B)^{-1})\ldots\varphi^{k-1}((A+B)^{-1}) \varphi ^{k}((A+B)^{-1}B)
  \varphi ^{k-1}(A) \ldots  \varphi  (A)A 
\end{equation*}
for $k\geq 0$ and a direct computation also shows that $(\id+Y)(\id+X)=\id=(\id+X)(\id+Y)$.
\end{rem}

\subsection{Reductive groups over $\mathbb{Q}_p$ and Whittaker functionals}

Let $p$ be a prime number let $\mathbb{Q}_p\subseteq K$ be
a finite extension with ring of integer $o_K$, uniformizer $p_K$, and residue
field $k=o_K/p_K$. This field will only play the role of coefficients, the
reductive groups will all be defined over $\mathbb{Q}_p$.
Following \cite{SVi}, let $G$ be the $\mathbb{Q}_p$-rational points of a
$\mathbb{Q}_p$-split connected reductive group over $\mathbb{Q}_p$. In
particular, $G$ is a locally $\mathbb{Q}_p$-analytic group. Moreover, we
assume that the centre of $G$ is connected. We fix a Borel subgroup $P=TN$ in
$G$ with maximal split torus $T$ and unipotent radical $N$. Let $\Phi^+$
denote, as usual, the set of positive roots of $T$ with respect to $P$ and let
$\Delta\subseteq\Phi^+$ be the subset of simple roots. For any
$\alpha\in\Phi^+$ we have the root subgroup $N_{\alpha}\subseteq N$. We recall
that $N=\prod_{\alpha\in\Phi^+}N_{\alpha}$ (set-theoretically) for any total
ordering of $\Phi^+$. Let $T_0\subseteq T$ be the maximal compact subgroup. We
fix a compact open subgroup $N_0\subseteq N$ which is totally
decomposed, in other words $N_0=\prod_{\alpha}(N_0\cap N_{\alpha})$ for any
total ordering of $\Phi^+$. Hence $P_0:=T_0N_0$ is a group. We introduce the
submonoid $T_+\subseteq T$ of all $t\in T$ such that $tN_0t^{-1}\subseteq
N_0$, or equivalently, such that $|\alpha(t)|\leq1$ for any
$\alpha\in\Delta$. Obviously, $P_+:=N_0T_+=P_0T_+P_0$ is then a submonoid of
$P$. 

We fix once and for all isomorphisms of algebraic groups
\begin{equation*}
\iota_{\alpha}\colon N_{\alpha}\overset{\cong}{\rightarrow}\mathbb{Q}_p
\end{equation*}
for $\alpha\in\Delta$, such that
\begin{equation*}
\iota_{\alpha}(tnt^{-1})=\alpha(t)\iota_{\alpha}(n)
\end{equation*}
for any $n\in N_{\alpha}$ and $t\in T$. We normalize these isomorphisms so that $\iota_\alpha(N_0\cap N_\alpha)=\mathbb{Z}_p\subset \mathbb{Q}_p$. Since
$\prod_{\alpha\in\Delta}N_{\alpha}$ is naturally a quotient of $N/[N,N]$ we
may view any homomorphism
\begin{equation*}
\ell\colon \prod_{\alpha\in\Delta}N_{\alpha}\rightarrow\mathbb{Q}_p
\end{equation*}
as a functional on $N$. We fix once and for all a homomorphism $\ell$
such that we have $\ell(N_0)=\mathbb{Z}_p$. Let $X^*(T):=\Hom_{alg}(T,\mathbb{G}_m)$ (resp.\  $X_*(T):=\Hom_{alg}(\mathbb{G}_m,T)$) be the group of algebraic characters (resp.\ cocharacters) of $T$. Since we assume that the 
centre of $G$ is connected, the quotient
$X^*(T)/\bigoplus_{\alpha\in\Delta}\mathbb{Z}\alpha$ is free. Hence we find a
cocharacter $\xi$ in $X_*(T)$ such that $\alpha\circ\xi=\id_{\mathbb{G}_m}$
for any $\alpha$ in $\Delta$. It is injective and uniquely determined up to a
central cocharacter. We fix once and for all such a $\xi$. It satisfies 
\begin{equation*}
\xi(\mathbb{Z}_p\setminus\{0\})\subseteq T_+
\end{equation*}
and
\begin{equation}\label{compat}
\ell(\xi(a)n\xi(a^{-1}))=a\ell(n)
\end{equation}
for any $a$ in $\mathbb{Q}_p^{\times}$ and $n$ in $N$ since $\ell$ is a linear
functional on the space $\prod_{\alpha\in\Delta}N_\alpha$ and therefore can be
written as a linear combination of the isomorphisms $\iota_\alpha\colon
N_\alpha\rightarrow \mathbb{Q}_p$. 

For example, if $G=\GL_n(\mathbb{Q}_p)$, $T$ is the group of diagonal
matrices, and $N$ is the group of unipotent upper triangular matrices, then we
could choose $\xi\colon \mathbb{G}_m(\mathbb{Q}_p)=\mathbb{Q}_p^\times\to
T=(\mathbb{Q}_p^{\times})^n$, $\xi(x):=\begin{pmatrix}x^{n-1}&&&\\&x^{n-2}&&\\&&\ddots&\\&&&1\end{pmatrix}$.

Put $\Gamma:=\xi(\mathbb{Z}_p^{\times})$ and $s:=\xi(p)$. The element $s$ acts
by conjugation on the group $N_0$ such that $\bigcap_ks^kN_0s^{-k}=\{1\}$. We denote this action by
$\varphi:=\varphi_s$. This is compatible with the functional $\ell$ in the
sense $\ell\circ\varphi=p\ell$ (see section \ref{phiring}) by \eqref{compat}. 
Therefore we may apply the theory of the preceding sections to any
$\varphi$-ring $R$ with the homomorphism $\ell\colon
N_0\rightarrow\mathbb{Z}_p$ and $N_1:=\Ker(\ell_{\mid N_0})$. We are going to
apply the theory of section \ref{phiring} in the setting $H_0:=N_0$ and $H_1:=N_1$.

Note that in
\cite{SVi} and \cite{Z} $\ell$ is assumed to be generic---we do not assume this
here, though. We remark that for any $\alpha\in\Delta$ the restriction of
$\ell$ to a fixed $N_{\alpha}$ is either zero or an isomorphism of $N_\alpha$
with $\mathbb{Q}_p$ and put $a_\alpha:=\ell(\iota_\alpha^{-1}(1))$. By
the assumption $\ell(N_0)=\mathbb{Z}_p$ we obtain $a_\alpha\in\mathbb{Z}_p$
for all $\alpha\in\Delta$ and $a_\alpha\in\mathbb{Z}_p^{\times}$ for at least
one $\alpha$ in $\Delta$. We put $T_{+,\ell}:=\{t\in T_+\mid
tN_1t^{-1}\subseteq N_1\}$. The monoid $T_{+,\ell}$ acts on the group $\mathbb{Z}_p$
via $\ell\colon N_0\rightarrow\mathbb{Z}_p$, too. 

A $(\varphi,\Gamma)$-ring $R$ is by definition a $\varphi$-ring (in the sense of
section \ref{phiring}) together with an action of $\Gamma\cong\mathbb{Z}_p^{\times}$ commuting with
$\varphi$ and satisfying $\gamma(\chi(x))=\chi(\xi^{-1}(\gamma)x)$. For
example
$\mathcal{O_E},\mathcal{O}^\dagger_{\mathcal{E}},\mathcal{E}^\dagger,\mathcal{R}$
are $(\varphi,\Gamma)$-rings. Note that
the endomorphism ring $\End(\mathbb{Z}_p)$ of the $p$-adic integers (as a
topological abelian group) is
isomorphic to $\mathbb{Z}_p$. On the other hand, the multiplicative monoid
$\mathbb{Z}_p\setminus\{0\}$ is isomorphic to
$\varphi^{\mathbb{N}}\Gamma$. Now having an action of $\varphi$ and
$\Gamma$ on $R$ we obtain an action of $T_{+,\ell}$ on $R$ since the map
$\ell\colon N_0\rightarrow\mathbb{Z}_p$ induces a monoid homorphism
$T_{+,\ell}\rightarrow\mathbb{Z}_p\setminus\{0\}\cong\varphi^{\mathbb{N}}\Gamma$. We
denote the kernel of this monoid homomorphism by $T_{0,\ell}$. 
Similarly, we have a natural action of $T_{+,\ell}$
on the ring $R\bs N_1,\ell\js$ by conjugation. Indeed, if $t\in T_{+,\ell}$
then since $T$ is commutative we have 
\begin{equation*}
t\varphi^k(N_1)t^{-1}=ts^kN_1s^{-k}t^{-1}=s^ktN_1t^{-1}s^{-k}=\varphi^k(tN_1t^{-1})\subseteq\varphi^k(N_1),
\end{equation*}
whence $tN_kt^{-1}\subseteq N_k$. Hence $t$ acts naturally on the skew
group ring $R[N_1/N_k,\ell]$ and by taking the limit we also obtain an action
on $R\bs N_1,\ell\js$. We denote the map on both $R$ and $R\bs
N_1,\ell\js$ induced by the action of $t\in T_{+,\ell}$ by $\varphi_t$.

Now a $T_{+,\ell}$-module over $R$ (resp.\ over $R\bs N_1,\ell\js$) is a finitely
generated free $R$-module $D$ (resp.\ $R\bs N_1,\ell\js$-module $M$) with a semilinear
action of $T_{+,\ell}$ (denoted by $\varphi_t\colon D\rightarrow D$,
resp.\ $\varphi_t\colon M\rightarrow M$ for any $t\in T_{+,\ell}$) such that
the restriction of the $T_{+,\ell}$-action to $s\in T_{+,\ell}$ 
defines a $\varphi$-module over $R$ (resp.\ over $R\bs N_1,\ell\js$). We
denote the category of $T_{+,\ell}$-modules over $R$ (resp.\ over $R\bs N_1,\ell\js$)
by $\mathfrak{M}(R,T_{+,\ell})$ (resp.\ by $\mathfrak{M}(R\bs
N_1,\ell\js,T_{+,\ell})$). 

\begin{lem}
Let $M$ be in $\mathfrak{M}(R\bs N_1,\ell\js,T_{+,\ell})$ and $D$ be in
$\mathfrak{M}(R,T_{+,\ell})$. Then the maps
\begin{eqnarray*}
1\otimes\varphi_t\colon R\bs N_1,\ell\js\otimes_{R\bs N_1,\ell\js,\varphi_t}M&\rightarrow& M\\
r\otimes m&\mapsto& r\varphi_t(m)
\end{eqnarray*}
and
\begin{eqnarray*}
1\otimes\varphi_t\colon R\otimes_{R,\varphi_t}D&\rightarrow& D\\
r\otimes d&\mapsto& r\varphi_t(d)
\end{eqnarray*}
are isomorphisms for any $t\in T_{+,\ell}$.
\end{lem}
\begin{proof}
We only prove the statement for $M$ (the statement for $D$ is entirely
analogous). First note that the subgroups $s^kN_0s^{-k}$ (resp.\ $s^kN_1s^{-k}$) form a
system of neighbourhoods of $1$ in $N$ (resp.\ in $\Ker(\ell)$). On
the other hand, if $t$ is in $T_{+,\ell}$ then 
\begin{equation*}
t\Ker(\ell_{\mid N})t^{-1}=t\left(\bigcup_{k\in\mathbb{Z}} s^k
N_1s^{-k}\right)t^{-1}=\bigcup_{k\in\mathbb{Z}} s^ktN_1t^{-1}s^{-k}= \Ker(\ell_{\mid N}).
\end{equation*} 
since $tN_1t^{-1}$ has finite index in $N_1$. Now since $t^{-1}N_0t$ and
 on $t^{-1}N_1t$ are compact, we find $k_0>0$ so that
 $t^{-1}N_0t\subseteq s^{-k_0}N_0s^{k_0}$ and $t^{-1}N_1t\subseteq
 s^{-k_0}N_1s^{k_0}$ whence $s^{k_0}t^{-1}$ lies in $T_{+,\ell}$. Since $M$ is a $\varphi$-module over
$R\bs N_1,\ell\js$, the map
\begin{eqnarray*}
1\otimes\varphi_{s^{k_0}}\colon R\bs N_1,\ell\js\otimes_{R\bs
  N_1,\ell\js,\varphi_{s^{k_0}}}M&\rightarrow& M\\
r\otimes m&\mapsto&r\varphi_{s^{k_0}}(m)
\end{eqnarray*}
is an isomorphism. Moreover, under the identifications 
\begin{align*}
R\bs N_1,\ell\js\otimes_{R\bs N_1,\ell\js,\varphi_t}(R\bs
N_1,\ell\js\otimes_{R\bs N_1,\ell\js,\varphi_{s^{k_0}t^{-1}}}M)\cong R\bs
N_1,\ell\js\otimes_{R\bs N_1,\ell\js,\varphi_{s^{k_0}}}M\cong\\ \cong R\bs
N_1,\ell\js\otimes_{R\bs N_1,\ell\js,\varphi_{s^{k_0}t^{-1}}}(R\bs
N_1,\ell\js\otimes_{R\bs N_1,\ell\js,\varphi_t}M) 
\end{align*}
we have
\begin{equation*}
(1\otimes\varphi_t)\circ(1\otimes(1\otimes\varphi_{s^{k_0}t^{-1}}))=1\otimes\varphi_{s^{k_0}}=(1\otimes\varphi_{s^{k_0}t^{-1}})\circ(1\otimes(1\otimes\varphi_t)), 
\end{equation*}
so $1\otimes\varphi_t$ is surjective by the equality on the left and
 injective by the equality on the right.
\end{proof}

\begin{rem}
Note that the action of $T_{0,\ell}$ on a
$T_{+,\ell}$-module $D$ over $R$ is linear since $T_{0,\ell}$ acts trivially
on $R$. Therefore this action extends (uniquely) to the subgroup $T_\ell\leq T$
generated by the monoid $T_{0,\ell}$.
\end{rem}
\begin{proof}
By the \'etaleness of the action of $\varphi_t$ for $t\in T_{0,\ell}$ we see
immediately that $\varphi_t$ is an automorphism of $D$ since $\varphi_t\colon
R\rightarrow R$ is the identity map. Therefore $\varphi_t$ has a (left and right) inverse (as
a linear transformation of the $R$-module $D$) which we denote by
$\varphi_{t^{-1}}$. The remark follows noting that $T_\ell$ consists of the
quotients of elements of $T_{0,\ell}$.
\end{proof}

In the case when $\ell=\ell_\alpha$ given by the projection of
$\prod_{\beta\in\Delta}N_\beta$ to $N_\alpha$ for some fixed simple root
$\alpha\in\Delta$ it is clear that $T_{+,\ell}=T_+$ as $N_\beta$ is
$T_+$-invariant for each $\beta\in\Phi^+$ and $\Ker(\ell)=\prod_{\alpha\neq\beta\in\Phi^+}N_\beta$. Therefore
$T_\ell\cong(\mathbb{Q}_p^{\times})^{n-1}$ where $n=\dim T$. This is the case
in which a $G$-equivariant sheaf on $G/P$ is constructed in \cite{SVZ}
associated to any object $D$ in $\mathfrak{M}(\mathcal{O_E},T_{+,\ell})$. So an object in
$\mathfrak{M}(\mathcal{O_E},T_{+,\ell})$ is nothing else but a $(\varphi,\Gamma)$-module
over $\mathcal{O_E}$ with an additional linear action of the group
$T_\ell$ (once we fixed the cocharacter $\xi$). In case of
$G=\GL_2(\mathbb{Q}_p)$ this additional action is just an action of the centre
$Z=T_\ell$ of $G$. In the work of Colmez \cite{Co1,Co2} on the $p$-adic
Langlands correspondence for $\GL_2(\mathbb{Q}_p)$ the action of $Z$ on an
irreducible $2$-dimensional \'etale $(\varphi,\Gamma)$-module $D$ is given by the
determinant (ie.\ the action of $\mathbb{Q}_p^{\times}\cong Z$ on
$\bigwedge^2D$). It is unclear at this point whether the action of $T_\ell$
can be chosen canonically (in a similar fashion) for a given $n$-dimensional
irreducible \'etale $(\varphi,\Gamma)$-module $D$.

As a corollary of Prop.\ \ref{equivcat} we obtain

\begin{pro}\label{T_+equiv}
The functors $\mathbb{D}=R\otimes_{R\bs N_1,\ell\js,\ell}\cdot$ and
$\mathbb{M}=R\bs N_1,\ell\js\otimes_{R,\iota}\cdot$ are quasi-inverse equivalences of
categories between $\mathfrak{M}(R\bs N_1,\ell\js,T_{+,\ell})$ and $\mathfrak{M}(R,T_{+,\ell})$.
\end{pro}
\begin{proof}
Since we clearly have
$\mathbb{D}\circ\mathbb{M}\cong\id_{\mathfrak{M}(R,T_{+,\ell})}$ and the faithfulness
of $\mathbb{D}$ is a formal consequence of Prop.\ \ref{equivcat},
it suffices to show that the isomorphism $\Theta\colon
M\rightarrow\mathbb{M}\circ\mathbb{D}(M)$ is $T_{+,\ell}$-equivariant whenever $M$
lies in $\mathfrak{M}(R\bs N_1,\ell\js,T_{+,\ell})$. Let $t\in T_{+,\ell}$ be
arbitrary and for an $m\in M$ write $m=\sum_{u\in
  J(N_0/\varphi_s^k(N_0))}u\varphi_s^k(m_{u,k})$. Since
$\mathbb{D}(\Theta)=\id_{\mathbb{D}(M)}$, we have
$(\Theta\circ\varphi_t-\varphi_t\circ\Theta)(M)\subseteq I_1\mathbb{M}\circ\mathbb{D}(M)$. We compute
\begin{align*}
(\Theta\circ\varphi_t-\varphi_t\circ\Theta)(m)=\sum_{u\in
    J(N_0/\varphi_s^k(N_0))}\varphi_t(u)\varphi_s^k\circ(\Theta\circ\varphi_t-\varphi_t\circ\Theta)(m_{u,k})\subseteq\\ \subseteq\varphi_s^k(I_1\mathbb{M}\circ\mathbb{D}(M))\subseteq
  I_{k+1}\mathbb{M}\circ\mathbb{D}(M)
\end{align*}
for all $k\geq 0$ showing that $\Theta$ is $\varphi_t$-equivariant.
\end{proof}

\section{The case of overconvergent and Robba rings}\label{microlocal}

\subsection{The locally analytic distribution algebra}

Let $p$ be a prime and put $\epsilon_p=1$ if $p$ is odd
and $\epsilon_p=2$ if $p=2$. If $H$ is a compact locally $\mathbb{Q}_p$-analytic group then we
denote by $D(H,K)$ the algebra of $K$-valued locally analytic
distributions on $H$. Recall that $D(H,K)$ is equal to the strong
dual of the locally convex vector space $C^{an}(H,K)$ of
$K$-valued locally $\mathbb{Q}_p$-analytic functions on $H$ with the
convolution product.

Recall that a topologically finitely generated pro-$p$ group $H$ is uniform, if it is powerful (ie.\ $H/\overline{H^{p^{\epsilon_p}}}$ is abelian) and for all $i\geq 1$ we have $|P_i(H):P_{i+1}(H)|=|H:P_2(H)|$ where $P_1(H)=H$ and $P_{i+1}(H)=\overline{P_i(H)^p[P_i(H),H]}$ (see \cite{DDMS} for more details). Now if $H$ is uniform, then it has a bijective global chart
\begin{eqnarray*}
\mathbb{Z}_p^{d}&\rightarrow&H\\
(x_1,\dots,x_{d})&\mapsto&h_1^{x_1}\cdots h_{d}^{x_{d}}
\end{eqnarray*}
where $h_{1},\dots,h_{d}$ is a fixed (ordered) minimal set of
topological generators of $H$. Putting
$b_i:=h_{i}-1\in\mathbb{Z}[G]$, ${\bf
b}^{\bf k}:=b_{1}^{k_1}\cdots b_{d}^{k_d}$ for
${\bf k}=(k_i)\in\mathbb{N}^{d}$ we can identify $D(H,K)$ with the ring
of all formal series
\begin{equation*}
\lambda=\sum_{{\bf k}\in\mathbb{N}^{d}}d_{\bf k}{\bf b}^{\bf k}
\end{equation*}
with $d_{\bf k}$ in $K$ such that the set
$\{|d_{\bf k}|\rho^{\epsilon_p|{\bf k}|}\}_{\bf k}$ is bounded for all
$0<\rho<1$. Here the first $|\cdot|$ is the normalized absolute
value on $K$ and the second one denotes the degree of ${\bf k}$,
that is $\sum_{i}k_{i}$. For any $\rho$ in $p^{\mathbb{Q}}$
with $p^{-1}<\rho<1$, we have a multiplicative norm
$\|\cdot\|_{\rho}$ on $D(H,K)$ \cite{ST1} given by
\begin{equation*}
\|\lambda\|_{\rho}:=\sup_{\bf k}|d_{\bf k}|\rho^{\epsilon_p|{\bf k}|}\ .
\end{equation*}
The family of norms $\|\cdot\|_{\rho}$ defines the Fr\'echet
topology on $D(H,K)$. The completion with respect to the norm
$\|\cdot\|_{\rho}$ is denoted by $D_{[0,\rho]}(H,K)$.

\subsection{Microlocalization}\label{partialrobba}

Let $G$ be the group of $\mathbb{Q}_p$-points of a $\mathbb{Q}_p$-split
connected reductive group with a fixed Borel subgroup $P=TN$. We also choose a
simple root $\alpha$ for the Borel subgroup $P$ and let $\ell=\ell_\alpha$ be
the functional given by the projection 
\begin{equation*}
\ell_\alpha\colon N\rightarrow
N/[N,N]\rightarrow\prod_{\beta\in\Delta}N_\beta\rightarrow N_\alpha\overset{\iota_\alpha}{\rightarrow}\mathbb{Q}_p\ .
\end{equation*}
Therefore we have $T_{+,\ell}=T_+$ as $N_\beta$ is
$T_+$-invariant for each $\beta\in\Phi^+$. We assume further that $N_0$ is \emph{uniform}. 

Let us begin by recalling the definition of the
classical Robba ring for the group $\mathbb{Z}_p$. The
distribution algebra $D(\mathbb{Z}_p,K)$ of $\mathbb{Z}_p$ can clearly be
identified with the ring of power series (in variable $T$) with
coefficients in $K$ that are convergent in the $p$-adic open unit
disc. Now put
\begin{equation*}
\mathcal{A}_{[\rho,1)}:=\hbox{the ring of all Laurent series
}f(T)=\sum_{n\in\mathbb{Z}}a_nT^n\hbox{ that converge for
}\rho\leq|T|<1.
\end{equation*}
For $\rho\leq\rho'$ we have a natural inclusion
$\mathcal{A}_{[\rho,1)}\hookrightarrow\mathcal{A}_{[\rho',1)}$ so
we can form the inductive limit
\begin{equation*}
\mathcal{R}:=\varinjlim_{\rho\rightarrow1}\mathcal{A}_{[\rho,1)}
\end{equation*}
defining the Robba ring. $\mathcal{R}$ is a $(\varphi,\Gamma)$-ring over $\mathbb{Z}_p$
with the maps $\chi\colon\mathbb{Z}_p\rightarrow\mathcal{R}^{\times}$ and
$\varphi\colon\mathcal{R}\rightarrow\mathcal{R}$ such that $\chi(1)=1+T$,
$\varphi(T)=(T+1)^p-1$, and $\gamma(T)=(1+T)^{\xi^{-1}(\gamma)}-1$ for $\gamma\in\Gamma$.

Recall that the ring 
\begin{equation*}
\mathcal{O}_{\mathcal{E}}^\dagger:=\{\sum_{n\in\mathbb{Z}}a_nT^n\mid a_n\in o_K\text{
  and there exists a }\rho<1\text{ s.t.\ }|a_n|\rho^n\to 0\text{ as }n\to-\infty\}
\end{equation*}
is called the ring of overconvergent power series. It is a subring of both
$\mathcal{O_E}$ and $\mathcal{R}$. We put
$\mathcal{E}^\dagger:=K\otimes_{o_K}\mathcal{O}_{\mathcal{E}}^\dagger$ which is
also a subring of the Robba ring. These rings are also
$(\varphi,\Gamma)$-rings.

The rings $\mathcal{O}_{\mathcal{E}}^\dagger\bs N_1,\ell\js$ and
$\mathcal{R}\bs N_1,\ell\js$ constructed in the previous sections are only
overconvergent (resp.\ Robba) in the variable $b_\alpha$ for the fixed simple
root $\alpha$. In all the other variables $b_\beta$ they behave like the
Iwasawa algebra $\Lambda(N_1)$ since we took the completion with respect to
the ideals generated by $(N_k-1)$. Moreover, in the projective limit
$\mathcal{O}_{\mathcal{E}}^\dagger\bs N_1,\ell\js\cong
\varprojlim_k\mathcal{O}_{\mathcal{E}}^\dagger[N_1/N_k,\ell]$ the terms are
not forced to share a common region of convergence. In this section we
construct the rings $\mathcal{R}^{int}(N_1,\ell)$ and
$\mathcal{R}(N_1,\ell)$ with better analytic properties.

We start with constructing a ring $\mathfrak{R}_0=\mathfrak{R}_0(N_0,K,\alpha)$ as a certain microlocalization
of the distribution algebra $D(N_0,K)$. We fix the topological
generator $n_{\alpha}$ of $N_0\cap N_{\alpha}$ such that $\ell_{\alpha}(n_{\alpha})=1$. This is possible since we normalized $\iota_\alpha\colon N_\alpha\overset{\sim}{\to}\mathbb{Q}_p$ so that $\iota_\alpha(N_0\cap N_\alpha)=\mathbb{Z}_p$.
Further, we fix topological generators $n_{\beta}$ of $N_0\cap
N_{\beta}$ for each $\alpha\neq\beta\in\Phi^+$. Since $N_0$ is uniform of
dimension $|\Phi^+|$, the set $A:=\{n_{\beta}\mid\beta\in\Phi^+\}$
is a minimal set of topological generators of the group $N_0$.
Moreover, $A\setminus\{n_{\alpha}\}$ is a minimal set of
generators of the group $N_1=\Ker(\ell)\cap N_0$. Further, we put
$b_{\beta}:=n_\beta-1$. For any real number $p^{-1}<\rho<1$ in $p^{\mathbb{Q}}$ the formula
$\|b_\beta\|_\rho:=\rho$ (for all $\beta\in\Phi^+$) defines a multiplicative
norm on $D(N_0,K)$. The completion of $D(N_0,K)$ with respect to this norm is
a Banach algebra which we denote by $D_{[0,\rho]}(N_0,K)$. Let now
$p^{-1}<\rho_1<\rho_2<1$ be real numbers in $p^{\mathbb{Q}}$. We take the
generalized microlocalization (cf. the Appendix of \cite{SZ}) of the Banach algebra
$D_{[0,\rho_2]}(N_0,K)$ at the multiplicatively closed set
$\{(n_{\alpha}-1)^i\}_{i\geq1}$ with respect to the pair of
norms $(\rho_1,\rho_2)$. This provides us with the Banach algebra
$D_{[\rho_1,\rho_2]}(N_0,K,\alpha)$. Recall that the elements of this
Banach algebra are equivalence classes of Cauchy sequences
$((n_\alpha-1)^{-k_n}x_n)_n$ (with $x_n\in D_{[0,\rho_2]}(N_0,K)$) with
respect to the norm $\|\cdot\|_{\rho_1,\rho_2}:=\max(\|\cdot\|_{\rho_1},\|\cdot\|_{\rho_2})$. 

Letting $\rho_2$ tend to
$1$ we define
$D_{[\rho_1,1)}(N_0,K,\alpha):=\varprojlim_{\rho_2\rightarrow
1}D_{[\rho_1,\rho_2]}(N_0,K,\alpha)$. This is a Fr\'echet-Stein
algebra (the proof is completely analogous to that of Theorem 5.5
in \cite{SZ}, but it is not a formal consequence of that). However, we will not
need this fact in the sequel so we omit the proof. Now the partial Robba ring
$\mathfrak{R}_0:=\mathfrak{R}_0(N_0,K,\alpha):=\varinjlim_{\rho_1\rightarrow1}D_{[\rho_1,1)}(N_0,K,\alpha)$
is defined as the injective limit of these Fr\'echet-Stein
algebras. We equip $\mathfrak{R}_0$ with the
inductive limit topology of the Fr\'echet topologies of
$D_{[\rho_1,1)}(N_0,K,\alpha)$. By the following parametrization
the partial Robba ring can be thought of as a skew Laurent series
ring on the variables $b_{\beta}$ ($\beta\in\Phi^+$) with
certain convergence conditions such that only the variable $b_\alpha$ is
invertible. Note that in \cite{SZ} a ``full'' Robba ring is constructed such
that all the variables $b_\beta$ are invertible. We denote the corresponding
``fully'' microlocalized Banach algebras by $D_{[\rho_1,\rho_2]}(N_0,K)$. In all these rings we will often omit $K$ from the notation if it is clear from the context.

\begin{rem}
The microlocalization of quasi-abelian normed algebras (Appendix of \cite{SZ})
is somewhat different from the microlocalisation constructing
$\Lambda_\ell(N_0)$ where first a localization (with respect to an Ore set) is
constructed and then the completion is taken. The set we are inverting here
does not satisfy the Ore property, so the localization in the usual sense does
not exist. However, we may complete and localize at the same time in order to
obtain a microlocalized ring directly.
\end{rem}

In order to be able to work with these rings we will show that
their elements can be viewed as Laurent series. The discussion below is
completely analogous to the discussion before Prop.\ A.24 in
\cite{SZ}. However, for the convenience of the reader, we explain the
method specialized to our case here. We introduce the affinoid domain
\begin{equation*}
    A_{\alpha}[\rho_1,\rho_2] := \{(z_\beta)_{\beta\in\Phi^+} \in
     \mathbb{C}_p^{\Phi^+} : \rho_1 
    \leq |z_{\alpha}| \leq \rho_2, 0\leq |z_\beta/z_\alpha|\leq 1\
    \hbox{for}\ \alpha\neq\beta\in\Phi^+\} \ .
\end{equation*}
This has the affinoid subdomain
\begin{equation*}
    X^{\Phi^+}_{[\rho_1,\rho_2]} := \{(z_\beta)_{\beta\in\Phi^+} \in \mathbb{C}_p^{\Phi^+} : \rho_1 \leq |z_{\beta_1}| = \ldots = |z_{\beta_{|\Phi^+|}}| \leq \rho_2\}
\end{equation*}
(where $\{\beta_1,\dots,\beta_{|\Phi^+|}\}=\Phi^+$) as defined in
\cite{SZ} (Prop.\ A.24). 

\begin{lem}
The ring $\mathcal{O}_K(A_\alpha[\rho_1,\rho_2])$ of $K$-analytic functions
on $A_\alpha[\rho_1,\rho_2]$ is the ring of all Laurent series
\begin{equation*}
    f(\mathbf{Z}) = \sum_{\bk \in \mathbb{Z}^{\{\alpha\}}\times\mathbb{N}^{\Phi^+\setminus\{\alpha\}}} d_\bk
    \mathbf {Z}^\bk
\end{equation*}
with $d_\bk \in K$ and such that $\lim_{\bk \rightarrow
\infty} |d_\bk| \rho^{\bk} = 0$ for any $\rho_1 \leq \rho \leq
\rho_2$.
Here
\begin{equation*}
    \mathbf{Z}^\bk := \prod_{\beta\in\Phi^+}Z_\beta^{k_\beta}
     \qquad\text{and}\qquad \rho^\bk :=\rho^{\sum_{\beta\in\Phi^+}k_\beta}
\end{equation*}
and ${\bf k}\to \infty$ means that $\sum_{\beta\in\Phi^+}|k_\beta|\to\infty$. This is the subring of $\mathcal{O}_K(X^{\Phi^+}_{[\rho_1,\rho_2]})$
 consisting of elements in which the variables $Z_\beta$ appear only
 with nonnegative exponent for all $\alpha\neq\beta\in\Phi^+$. 
\end{lem}
\begin{proof}
Since $X^{\Phi^+}_{[\rho_1,\rho_2]}\subseteq A_\alpha[\rho_1,\rho_2]$, we
 clearly have $\mathcal{O}_K(A_\alpha[\rho_1,\rho_2])\subseteq
 \mathcal{O}_K(X^{\Phi^+}_{[\rho_1,\rho_2]})$. Moreover, the power
 series in  $\mathcal{O}_K(A_\alpha[\rho_1,\rho_2])$ converge for
 $z_\beta=0$ ($\beta\neq \alpha$), hence these variables appear with
 nonnegative exponent. On the other hand, if we have a power series
 $f(\mathbf{Z})\in \mathcal{O}_K(X^{\Phi^+}_{[\rho_1,\rho_2]})$ such
 that the variables $Z_\beta$ have nonnegative exponent for all
 $\alpha\neq\beta\in\Phi^+$ then it also converges in the region $
 A_\alpha[\rho_1,\rho_2]$ as we have the trivial estimate
 $|\prod_{\beta\in\Phi^+}z_\beta^{k_\beta}|\leq
 |z_\alpha|^{\sum_{\beta\in\Phi^+}k_\beta}$ in this case.
\end{proof}

Since $\rho^\bk \leq
\max(\rho_1^\bk ,\rho_2^\bk)$ for any $\rho_1 \leq \rho \leq \rho_2$ and any
$\bk \in \mathbb
{Z}^{(\alpha)}\times\mathbb{N}^{\Phi^+\setminus\{\alpha\}}$ the
convergence condition on $f$ is equivalent to
\begin{equation*}
    \lim_{\bk \rightarrow
\infty} |d_\bk| \rho_1^{\bk} = \lim_{\bk \rightarrow \infty}
|d_\bk| \rho_2^{\bk} = 0 \ .
\end{equation*}
The spectral norm on the affinoid algebra
$\mathcal{O}_K(A_\alpha[\rho_1,\rho_2])$ (for the definition of these notions see \cite{FvdP}) is given by
\begin{align*}
    \|f\|_{A_\alpha[\rho_1,\rho_2]} & = \sup_{\rho_1 \leq \rho \leq
 \rho_2} \max_{\bk \in \mathbb{Z}^{\{\alpha\}}\times\mathbb{N}^{\Phi^+\setminus\{\alpha\}}} |d_\bk|\rho^\bk \\
    & = \max( \max_{\bk \in \mathbb{Z}^{\{\alpha\}}\times\mathbb{N}^{\Phi^+\setminus\{\alpha\}}} |d_\bk|\rho_1^\bk,
    \max_{\bk \in \mathbb{Z}^{\{\alpha\}}\times\mathbb{N}^{\Phi^+\setminus\{\alpha\}}} |d_\bk| \rho_2^\bk) \ .
\end{align*}
Setting $\bb^\bk := \prod_{\beta\in\Phi^+}b_\beta^{ k_\beta} $ for some
fixed ordering of $\Phi^+$ and for any
$\bk = ( k_\beta)_{\beta\in\Phi^+} \in \mathbb{Z}^{\{\alpha\}}\times\mathbb{N}^{\Phi^+\setminus\{\alpha\}}$ we claim that $f(\bb) := \sum_{\bk \in
\mathbb{Z}^{\{\alpha\}}\times\mathbb{N}^{\Phi^+\setminus\{\alpha\}}} d_\bk \bb^\bk$ converges in $D_{[\rho_1,\rho_2]}(N_0,K,\alpha)$ for $f\in\mathcal{O}_K(A_\alpha[\rho_1,\rho_2])$.
As a consequence of Prop.\ A.21 and Lemma A.7.iii in \cite{SZ}
we have
\begin{equation*}
    \|\bb^\bk\|_{\rho_1,\rho_2} = \max(\rho_1^\bk,\rho_2^\bk)
\end{equation*}
for any $\bk \in \mathbb{Z}^{\{\alpha\}}\times\mathbb{N}^{\Phi^+\setminus\{\alpha\}}$. Hence
\begin{equation*}
    \lim_{\bk \rightarrow
\infty} \|d_\bk\bb^\bk\|_{\rho_1,\rho_2} = \lim_{\bk \rightarrow
\infty} \max(|d_\bk| \rho_1^\bk, |d_\bk| \rho_2^\bk) = \max(
\lim_{\bk \rightarrow \infty} |d_\bk| \rho_1^\bk,
\lim_{\bk \rightarrow \infty} |d_\bk| \rho_2^\bk) = 0 \ .
\end{equation*}
Therefore
\begin{align*}
    \mathcal{O}_K(A_\alpha[\rho_1,\rho_2]) & \longrightarrow
    D_{[\rho_1,\rho_2]}(N_0,K,\alpha) \\
    f & \longmapsto f(\bb)
\end{align*}
is a well defined $K$-linear map. In order to investigate this map
we introduce the filtration
\begin{equation*}
    F^iD_{[\rho_1,\rho_2]}(N_0,K,\alpha) := \{ e \in D_{[\rho_1,\rho_2]}(N_0,K,\alpha) : \|e\|_{\rho_1,\rho_2}
    \leq |p|^i \} \qquad\text{for $i \in \mathbb{R}$}
\end{equation*}
on $D_{[\rho_1,\rho_2]}(N_0,K,\alpha)$. Since $K$ is discretely valued and $\rho_1,\rho_2 \in
p^{\mathbb{Q}}$ this filtration is quasi-integral in the sense of
\cite{ST1} \S1. The corresponding graded ring $gr^\cdot
D_{[\rho_1,\rho_2]}(N_0,K,\alpha)$, by Prop.\ A.21 in \cite{SZ}, is commutative. We let
$\sigma(e) \in gr^\cdot D_{[\rho_1,\rho_2]}(N_0,K,\alpha)$ denote the principal
symbol of any element $e \in D_{[\rho_1,\rho_2]}(N_0,K,\alpha)$.

\begin{pro}\label{4}
\begin{itemize}
    \item[i.] $gr^\cdot D_{[\rho_1,\rho_2]}(N_0,K,\alpha)$ is a free $gr^\cdot
K$-module with basis $\{ \sigma(\bb^\bk) : \bk \in
\mathbb{Z}^{\{\alpha\}}\times\mathbb{N}^{\Phi^+\setminus\{\alpha\}} \}$.
    \item[ii.] The map
\begin{align*}
    \mathcal{O}_K(A_\alpha[\rho_1,\rho_2]) & \xrightarrow{\;\cong\;}
    D_{[\rho_1,\rho_2]}(N_0,K,\alpha) \\
    f & \longmapsto f(\bb)
\end{align*}
is a $K$-linear isometric bijection.
\end{itemize}
\end{pro}
\begin{proof}
 Since $\{ b_\alpha^{-l}\mu
: l \geq 0 , \mu \in D_{[0,\rho_1]}(N_0,K) \}$ is dense in $D_{[\rho_1,\rho_2]}(N_0,K,\alpha)$
every element in the graded ring $gr^\cdot D_{[\rho_1,\rho_2]}(N_0,K,\alpha)$ is of
the form $\sigma(b_\alpha^{-l}\mu)$. Suppose that $\mu =
\sum_{\bk\in\mathbb{N}_0^{d}} d_{\bk}\bb^{\bk}$. Then $b_\alpha^{-l}\mu =
\sum_{\bk\in\mathbb{N}_0^{d}}
d_{\bk}b_\alpha^{-l}\bb^{\bk}$ and, using Lemma
A.7.iii \cite{SZ} we compute
\begin{align*}
    \|b_\alpha^{-l}\mu\|_{\rho_1,\rho_2}  = \max(
 \|b_\alpha\|_{\rho_1}^{-l} \|\mu\|_{\rho_1},
 \|b_\alpha\|_{\rho_2}^{-l}\|\mu\|_{\rho_2}) 
     = \max( \max_{\bk \in \mathbb{N}_0^d} |d_\bk| \rho_1^{\bk
    - l}, \max_{\bk \in \mathbb{N}_0^d} |d_\bk| \rho_2^{\bk
    - l}) \\
     = \max_{\bk \in \mathbb{N}_0^d} |d_\bk| \max(\rho_1^{\bk
    - l}, \rho_2^{\bk - l})
     = \max_{\bk \in \mathbb{N}_0^d} |d_\bk|
    |b_\alpha^{-l}\bb^\bk|_{\rho_1,\rho_2} \ .
\end{align*}
It follows that $gr^\cdot D_{[\rho_1,\rho_2]}(N_0,K,\alpha)$ as a $gr^\cdot
K$-module is generated by the principal symbols
$\sigma(b_\alpha^{-l}\bb^\bk)$ with $\bk \in
\mathbb{N}_0^d$, $l\geq 0$. But it also follows that, for a fixed $l\geq
 0$, the principal symbols
$\sigma(b_\alpha^{-l}\bb^\bk)$ with $\bk$ running over
$\mathbb{N}_0^d$ are linearly independent over $gr^\cdot K$. By
Prop.\ A.21 in \cite{SZ} we may permute the factors in
$\sigma(b_\alpha^{-l}\bb^\bk)$ arbitrarily. Hence $gr^\cdot
D_{[\rho_1,\rho_2]}(N_0,K,\alpha)$ is a free $gr^\cdot
K$-module with basis $\{ \sigma(\bb^\bk) : \bk \in
\mathbb{Z}^{\{\alpha\}}\times\mathbb{N}^{\Phi^+\setminus\{\alpha\}} \}$.

On the other hand, we of course have
\begin{align*}
    \|f(\bb)\|_{\rho_1,\rho_2} & \leq \max_{\bk \in \mathbb{Z}^{\{\alpha\}}\times\mathbb{N}^{\Phi^+\setminus\{\alpha\}}}
    |d_\bk||\bb^\bk|_{\rho_1,\rho_2} \\
    & = \max_{\bk \in \mathbb{Z}^{\{\alpha\}}\times\mathbb{N}^{\Phi^+\setminus\{\alpha\}}} |d_\bk|
    \max(\rho_1^\bk,\rho_2^\bk) \\
    & = \max( \max_{\bk \in \mathbb{Z}^{\{\alpha\}}\times\mathbb{N}^{\Phi^+\setminus\{\alpha\}}}
    |d_\bk|\rho_1^\bk, \max_{\bk \in \mathbb{Z}^{\{\alpha\}}\times\mathbb{N}^{\Phi^+\setminus\{\alpha\}}}
    |d_\bk|\rho_2^\bk) \\
    & = |f|_{A_\alpha[\rho_1,\rho_2]} \ .
\end{align*}
This means that if we introduce on $\mathcal{O}_K(A_\alpha[\rho_1,\rho_2])$
the filtration defined by the spectral norm then the asserted map
respects the filtrations, and by the above reasoning it induces an
isomorphism between the associated graded rings. Hence, by
completeness of these filtrations, it is an isometric bijection.
\end{proof}

Now we turn to the construction of $\mathcal{R}(N_1,\ell)$. The problem with
(na\"ive) microlocalization is that the ring $\mathfrak{R}_0$
is not finitely generated over $\varphi(\mathfrak{R}_0)$. The reason for
this is that $\varphi$ improves the order of convergence for a power series in
$\mathfrak{R}_0$. In the case $G\neq \GL_2(\mathbb{Q}_p)$ the operator
$\varphi=\varphi_s$ acts by conjugation on $N_{\beta}$ 
by raising to the $\beta(s)$-th power. Whenever $\beta\in \Phi^+\setminus
\Delta$ is not a simple root then
$\beta(s)=p^{m_\beta}>\alpha(s)=p$ where $m_\beta$ is the degree
of the map $\beta\circ\xi\colon\mathbb{G}_m\to\mathbb{G}_m$. 
\begin{lem}\label{trivest}
We have $\|b_\beta\|_\rho=\|b_\alpha\|_\rho=\rho$ and
\begin{equation*}
\|\varphi(b_\beta)\|_\rho=\max_{0\leq j\leq
 m_{\beta}}(\rho^{p^j}p^{j-m_\beta})<\max(\rho^p,p^{-1}\rho)=\|\varphi(b_\alpha)\|_\rho
\end{equation*}
for any $p^{-1}<\rho<1$. In general, we have
 $\|\varphi_t(b_\beta)\|_\rho=\max\limits_{0\leq j\leq \mathrm{val}_p(\beta(t))}(\rho^{p^j}p^{j-\mathrm{val}_p(\beta(t))})$.
\end{lem}
\begin{proof}
We compute 
\begin{equation*}
\|\varphi_t(b_\beta)\|_\rho=\|(1+b_\beta)^{\beta(t)}-1\|_\rho=\|\sum_{i=1}^\infty \binom{\beta(t)}{i}b_\beta^{i}\|_\rho=\max_{0\leq j\leq \val_p(\beta(t))}(\rho^{p^j}p^{j-\mathrm{val}_p(\beta(t))})\
 .
\end{equation*}
Here we use the trivial estimate
 $\mathrm{val}_p\binom{n}{k}=\val_p(\frac{n}{k}\binom{n-1}{k-1})\geq
 \mathrm{val}_p(n)-\val_p(k)$ for $n:=\beta(t)\in\mathbb{Z}_p$ and $k\in
 \mathbb{N}$. We see immediately that whenever
 $m_\beta>1$ then 
 $\rho^{p^{j}}p^{j-m_\beta}<\rho^p$ for $1\leq j\leq m_\beta$ and
$p^{-m_\beta}\rho<p^{-1}\rho$.
\end{proof}

Now choose an ordering $<$ on $\Phi^+$ such that $(i)$ $m_{\beta_1}<m_{\beta_2}$
implies $\beta_1>\beta_2$ and $(ii)$ $\alpha>\beta$ for any $\alpha\neq
\beta,\beta_1,\beta_2\in\Phi^+$. Then by Prop.\ \ref{4} any element in
$\mathfrak{R}_0$ has a skew Laurent-series expansion
\begin{equation*}
f({\bf b})=\sum_{{\bf
    k}\in\mathbb{Z}^{\{\alpha\}}\times\mathbb{N}^{\Phi^+\setminus\{\alpha\}}}c_{\bf k}{\bf b}^{\bf k}
\end{equation*}
such that there exists $p^{-1}<\rho<1$ such that for all $\rho<\rho_1<1$ we
have $|c_{\bf k}|_p\rho_1^{\sum k_\beta}\to0$ as $\sum |k_\beta|\to\infty$. By Lemma
\ref{trivest} and the discussion above we clearly have the following 

\begin{exm}\label{ex}
Let $\beta\in\Phi^+\setminus\Delta$ be a non-simple root. Then the series
$\sum_{n=1}^\infty b_{\beta}^nb_\alpha^{-n}$ does not belong to
$\mathfrak{R}_0(N_0)$. However, the series
$\sum_{n=1}^\infty\varphi(b_\beta^nb_\alpha^{-n})$ converges in each
$D_{[\rho_1,\rho_2]}(N_0,\alpha)$ (for arbitrary $p^{-1}<\rho_1<\rho_2<1$)
hence defines an element in $\mathfrak{R}_0(N_0)$. Therefore we cannot have a continuous left inverse $\psi$ to $\varphi$ on $\mathfrak{R}_0(N_0)$ as otherwise $\psi(\sum_{n=1}^{\infty}\varphi(b_{\beta}^nb_\alpha^{-n}))=\sum_{n=1}^\infty b_{\beta}^nb_\alpha^{-n}$ would converge. In particular, we cannot write $\mathfrak{R}_0(N_0)$ as the topological direct sum $ \bigoplus_{u\in N_0/\varphi(N_0)}u\varphi(\mathfrak{R}_0(N_0))$ of closed subspaces in $\mathfrak{R}_0(N_0)$ as otherwise the operator 
\begin{eqnarray*}
\psi\colon \mathfrak{R}_0(N_0)&\to& \mathfrak{R}_0(N_0)\\
\sum_{u\in J(N_0/\varphi(N_0))}u\varphi(f_u)&\mapsto& \varphi^{-1}(u_0)f_{u_0}
\end{eqnarray*}
for the unique $u_0\in J(N_0/\varphi(N_0))\cap \varphi(N_0)$ would be a continuous left inverse to $\varphi$. In fact, we even have $\mathfrak{R}_0(N_0)\neq \bigoplus_{u\in N_0/\varphi(N_0)}u\varphi(\mathfrak{R}_0(N_0))$ algebraically, however, the proof of this requires the forthcoming machinery (see Remark \ref{notetale}).
\end{exm}

In order to overcome the above counter-example we are going to consider the
ring $\mathcal{R}(N_1,\ell)$ of all the skew power series of the form $f({\bf b})$
such that $f(\varphi_t({\bf b}))$ is convergent in $\mathfrak{R}_0$ for some
$t\in T_+$. A priori it is not clear that these series form a
ring, so we are going to give a more conceptual construction. 

Take an arbitrary element $t\in T_+$. The conjugation by $t$ on $N_0$ gives an isomorphism $\varphi_t\colon
N_0\to\varphi_t(N_0)$ of pro-$p$ groups (since it is injective). Hence
$\varphi_t(N_0)$ is also a uniform pro-$p$ group with minimal set of generators
$\{\varphi_t(n_\beta)\}_{\beta\in\Phi^+}$. So we may define the distribution
algebra $D(\varphi_t(N_0)):=D(\varphi_t(N_0),K)$. The inclusion
$\varphi_t(N_0)\hookrightarrow N_0$ induces an injective homomorphism of
Fr\'echet-algebras $\iota_{1,t}\colon D(\varphi_t(N_0))\hookrightarrow
D(N_0)$. It is well-known \cite{ST1} 
that we have
\begin{equation*}
D(N_0)=\bigoplus_{n\in J(N_0/\varphi_t(N_0))}n\iota_{1,t}(D(\varphi_t(N_0)))
\end{equation*}
as right $D(\varphi_t(N_0))$-modules. Moreover, the direct summands are closed
in $D(N_0)$. For each real number $p^{-1}<\rho<1$ the
$\rho$-norm on $D(N_0)$ defines a norm $r_t(\rho)$ on $D(\varphi_t(N_0))$ by
restriction. Note that this is different from the $\rho$-norm on
$D(\varphi_t(N_0))$ (using the uniform structure on $\varphi_t(N_0)$). However,
the family $(r_t(\rho))_\rho$ of norms defines the Fr\'echet topology on
$D(\varphi_t(N_0))$. On the other hand, whenever $r$ is a norm on
$D(\varphi_t(N_0))$ then we may extend $r$ to a norm $q_t(r)$ on $D(N_0)$ by
putting 
\begin{equation*}
\|\sum_{n\in J(N_0/\varphi_t(N_0))}n\iota_{1,t}(x_n)\|_{q_t(r)}:=\max(\|x_n\|_r).
\end{equation*}
These norms define the Fr\'echet topology on $D(N_0)$. More precisely, if $\beta(t)=p^{m(\beta,t)}u(\beta,t)$ with
$m(\beta,t):=\val_p(\beta(t))\geq 0$ integer and 
$u(\beta,t)\in\mathbb{Z}_p^{\times}$, then we have

\begin{lem}\label{equivnorm}
\begin{equation*}
\|x\|_\rho\leq \|x\|_{q_t(r_t(\rho))}\leq \rho^{-\sum_{\beta\in\Phi^+}(p^{m(\beta,t)}-1)}\|x\|_\rho
\end{equation*}
for any $p^{-\frac{1}{\max_{\beta\in\Phi^+}p^{m(\beta,t)}}}<\rho<1$ and $x\in D(N_0)$. In particular, the norms $\rho$ and
$q_t(r_t(\rho))$ define the same topology.
\end{lem}

\begin{proof}
The inequality on the left is clear from the triangle inequality. For
 the other inequality note that our assumption on $\rho$ implies in
 particular that 
\begin{equation*}
\rho^{p^{m(\beta,t)}}= \rho^{p^j}\rho^{p^{m(\beta,t)}-p^j}>\rho^{p^j}p^{-\frac{p^{m(\beta,t)}-p^j}{p^{m(\beta,t)}}}>\rho^{p^j}p^{j-m(\beta,t)}
\end{equation*}
for all $0\leq j<m(\beta,t)$. Hence by Lemma \ref{trivest}, we have
 $\rho^{p^{m(\beta,t)}}=\|\binom{\beta(t)}{p^{m(\beta,t)}}b_\beta^{p^{m(\beta,t)}}\|_\rho=\|\varphi_t(b_\beta)\|_\rho$. Moreover,
 there exists an \emph{invertible} element $y$ in the Iwasawa algebra
 $\Lambda(N_{0,\beta})$ such that
 $y\varphi_t(b_\beta)\equiv\binom{\beta(t)}{p^{m(\beta,t)}}b_\beta^{p^{m(\beta,t)}}\pmod{p}$
 (as both sides have the same principal term). However,
 by the choice of $\rho$,
 $|p|=1/p<\rho^{p^{m(\beta,t)}}=\|\varphi_t(b_\beta)\|_\rho=\|\varphi_t(b_\beta)\|_{q_t(r_t(\rho))}$. Therefore
 we also have 
\begin{align*}
\rho^{p^{m(\beta,t)}}=\|b_{\beta}^{p^{m(\beta,t)}}\|_\rho=\|\binom{\beta(t)}{p^{m(\beta,t)}}b_\beta^{p^{m(\beta,t)}}\|_\rho=\|\varphi_t(b_{\beta})\|_\rho=
\|\varphi_t(b_\beta)\|_{q_t(r_t(\rho))}=\\
=\|y\varphi_t(b_\beta)\|_{q_t(r_t(\rho))}=\|\binom{\beta(t)}{p^{m(\beta,t)}}b_\beta^{p^{m(\beta,t)}}\|_{q_t(r_t(\rho))}=\|b_{\beta}^{p^{m(\beta,t)}}\|_{q_t(r_t(\rho))}
\end{align*}
whence 
\begin{equation}
\|b_\beta^{k_\beta}\|_{q_t(r_t(\rho))}=\|b_\beta^{p^{m(\beta,t)}}\|_{q_t(r_t(\rho))}^{k_{1,\beta}}\|b_\beta^{k_{2,\beta}}\|_{q_t(r_t(\rho))}\leq
\rho^{k_{1,\beta}p^{m(\beta,t)}}\leq\rho^{-p^{m(\beta,t)}+1}\|b_\beta^{k_\beta}\|_\rho\label{est}
\end{equation}
where $k_\beta=p^{m(\beta,t)}k_{1,\beta}+k_{2,\beta}$ with $0\leq
 k_{2,\beta}\leq p^{m(\beta,t)}-1$ and $k_{1,\beta}$ nonnegative integers.

Now consider an element of $D(N_0)$ of the form
\begin{equation*}
x=\sum_{{\bf k}=(k_\beta)\in\mathbb{N}^{\Phi^+}}c_{\bf k}\prod_{\beta\in\Phi^+}b_\beta^{k_\beta}\ . 
\end{equation*}
We may assume without loss of generality that
$J(N_0/\varphi_t(N_0))=\{\prod_{\beta\in\Phi^+}n_\beta^{j_\beta}\mid 0\leq j_\beta\leq p^{m(\beta,t)}-1\}$
where the product is taken in the reversed order. Let
$\eta\in\Phi^+$ be the largest root (with respect to the ordering $<$
defined after Lemma \ref{trivest}) such
that there exists a ${\bf k}\in\mathbb{N}^{\Phi^+}$ with $c_{\bf k}\neq 0$
and $k_{\eta}\neq0$. We are going to show the estimate
\begin{equation*}
\|x\|_{q_t(r_t(\rho))}\leq \rho^{-\sum_{\beta\leq\eta}(p^{m(\beta,t)}-1)}\|x\|_\rho
\end{equation*}
by induction on $\eta$. This induction has in fact finitely many steps since
$|\Phi^+|<\infty$. At first we write
$b_{\eta}^{k_{\eta}}=\sum_{j_{\eta}=0}^{p^{m(\eta,t)}-1}n_\eta^{j_\eta}f_{{\bf
  k},j_\eta}(\varphi_t(b_\eta))$ for each ${\bf k}\in\mathbb{N}^{\Phi^+}$. Note
that---by the choice of the ordering on $\Phi^+$---for any fixed $\eta$ the set
$\prod_{\beta<\eta}N_{0,\beta}$ is a normal subgroup of $N_0$. Moreover,
the conjugation by any element of $N_0$ preserves the $\rho$-norm on
$D(N_0)$. Therefore we may write
\begin{equation*}
\prod_{\beta\leq\eta}b_\beta^{k_\beta}=\sum_{j_\eta=0}^{p^{m(\eta,t)}-1}n_{\eta}^{j_\eta}x_{{\bf
  k},j_{\eta}}f_{{\bf k},j_\eta}(\varphi_t(b_\eta))
\end{equation*}
such that 
\begin{equation*}
x_{{\bf
    k},j_\eta}:=n_\eta^{-j_\eta}\left(\prod_{\beta<\eta}b_\beta^{k_\beta}\right)
n_\eta^{j_\eta}\in D(\prod_{\beta<\eta}N_{0,\beta})\ .
\end{equation*}
By \eqref{est} we have
\begin{equation*}
\|f_{{\bf k},j_\eta}(\varphi_t(b_\eta))\|_{q_t(r_t(\rho))}=\|f_{{\bf k},j_\eta}(\varphi_t(b_\eta))\|_\rho
\leq \|b_\eta^{k_\eta}\|_{q_t(r_t(\rho))}\leq
\rho^{-p^{m(\eta,t)}+1}\|b_\eta^{k_\eta}\|_\rho\ .
\end{equation*}
Since the $r_t(\rho)$-norm is multiplicative on $D(\varphi_t(N_0))$, for any $a\in D(N_0)$ and $b\in D(\varphi_t(N_0))$ we
 also have $\|a\iota_{1,t}(b)\|_{q_t(r_t(\rho))}=\|a\|_{q_t(r_t(\rho))}\|b\|_{r_t(\rho)}$. Indeed, if we
 decompose $a$ as $a=\sum_{n\in J(N_0/\varphi_t(N_0))}n\iota_{1,t}(a_n)$ then
 we have $a\iota_{1,t}(b)=\sum_{n\in
 J(N_0/\varphi_t(N_0))}n\iota_{1,t}(a_nb)$. Now $f_{{\bf 
    k},j_\eta}(\varphi_t(b_\eta))$ lies in $\iota_{1,t}(D(\varphi_t(N_0)))$, so we see that
\begin{equation*}
\|x_{{\bf k},j_\eta}f_{{\bf
    k},j_\eta}(\varphi_t(b_\eta))\|_{q_t(r_t(\rho))}=\|x_{{\bf
    k},j_\eta}\|_{q_t(r_t(\rho))}\|f_{{\bf
    k},j_\eta}(\varphi_t(b_\eta))\|_{q_t(r_t(\rho))}\ .
\end{equation*}
On the other hand, the inductional hypothesis tells us that
\begin{equation*}
\|x_{{\bf k},j_\eta}\|_{q_t(r_t(\rho))}\leq
\rho^{-\sum_{\beta<\eta}(p^{m(\beta,t)}-1)}\|x_{{\bf
    k},j_\eta}\|_\rho=\rho^{-\sum_{\beta<\eta}(p^{m(\beta,t)}-1)}\|\prod_{\beta<\eta}b_\beta^{k_\beta}\|_\rho\ . 
\end{equation*}
Hence we compute 
\begin{align*}
\|x\|_{q_t(r_t(\rho))}=\|\sum_{\bf k}c_{\bf k}\sum_{j_\eta=0}^{p^{m(\eta,t)}-1}n_{\eta}^{j_\eta}x_{{\bf
  k},j_{\eta}}f_{{\bf
    k},j_\eta}(\varphi_t(b_\eta))\|_{q_t(r_t(\rho))}\leq\\
\leq\max_{{\bf k},j_\eta}\left(|c_{\bf
k}|\|x_{{\bf
    k},j_{\eta}}f_{{\bf k},j_\eta}(\varphi_t(b_\eta))\|_{q_t(r_t(\rho))}\right)\leq\\
\leq \max_{\bf k}\left(|c_{\bf
  k}|\rho^{-\sum_{\beta<\eta}(p^{m(\beta,t)}-1)}\|\prod_{\beta<\eta}b_\beta^{k_\beta}\|_\rho\rho^{-p^{m(\eta,t)}+1}\|b_\eta^{k_\eta}\|_\rho\right)
=\rho^{-\sum_{\beta\leq\eta}(p^{m(\beta,t)}-1)}\|x\|_\rho\ .
\end{align*}
\end{proof}

In particular, for each
$p^{-\frac{1}{\max_{\beta}p^{m(\beta,t)}}}<\rho<1$ the completion of
$D(N_0)$ with respect to the topology defined by $\|\cdot\|_\rho$ and by
$\|\cdot\|_{q_t(r_t(\rho))}$ are the same, ie.\
\begin{equation}\label{dec0}
D_{[0,\rho]}(N_0)= \bigoplus_{n\in
  J(N_0/\varphi_t(N_0))}n\iota_{1,t}(D_{r_t([0,\rho])}(\varphi_t(N_0)))
\end{equation}
where $D_{r_t([0,\rho])}(\varphi_t(N_0))$ denotes the completion of
$D(\varphi_t(N_0))$ with respect to the norm $r_t(\rho)$.

Now we turn to the microlocalization and first of all note that
$\varphi_t(b_\alpha)=(b_\alpha+1)^{\alpha(t)}-1$ is divisible by $b_\alpha$. So 
if $\varphi_t(b_\alpha)$ is invertible in a ring then so is $b_\alpha$. On the
other hand, if $p^{-\frac{1}{p^{m(\alpha,t)}}}<\rho<1$ then by Lemma \ref{trivest} we have
\begin{equation*}
\|\varphi_t(b_\alpha)-\binom{\alpha(t)}{p^{m(\alpha,t)}}b_\alpha^{p^{m(\alpha,t)}}\|_\rho<\|\binom{\alpha(t)}{p^{m(\alpha,t)}}b_\alpha^{p^{m(\alpha,t)}}\|_\rho\ .
\end{equation*}
Hence $\varphi_t(b_\alpha)$ is invertible in the Banach algebra
$D_{[\rho_1,\rho_2]}(N_0,\alpha)$ for any
$p^{-\frac{1}{p^{m(\alpha,t)}}}<\rho_1<\rho_2<1$ since it is close to
the invertible element
$\binom{\alpha(t)}{p^{m(\alpha,t)}}b_\alpha^{p^{m(\alpha,t)}}$ (as the
binomial coefficient $\binom{\alpha(t)}{p^{m(\alpha,t)}}$ is not
divisible by $p$). This shows that the microlocalisation of $D_{[0,\rho_2]}(N_0)$
with respect to the multiplicative set $\varphi_t(b_\alpha)^{\mathbb{N}}$ and
norm $\max(\rho_1,\rho_2)$ equals $D_{[\rho_1,\rho_2]}(N_0,\alpha)$. Therefore
for each $p^{-\frac{1}{p^{m(\alpha,t)}}}<\rho_1<\rho_2<1$ we obtain
\begin{equation*}
D_{[\rho_1,\rho_2]}(N_0,\alpha)= \bigoplus_{n\in
  J(N_0/\varphi_t(N_0))}n\iota_{0,1}(D_{r_t([\rho_1,\rho_2])}(\varphi_t(N_0),\alpha))
\end{equation*}
by microlocalizing both sides of \eqref{dec0}. Now letting $\rho_2$ tend to $1$ and then also $\rho_1\to 1$ we get
\begin{equation}
\mathfrak{R}_0(N_0,\alpha)=\bigoplus_{n\in
  J(N_0/\varphi_t(N_0))}n\iota_{1,t}(\mathfrak{R}_{0,r_t(\cdot)}(\varphi_t(N_0),\alpha))\label{dec1}
\end{equation}
for all $t\in T_+$. Here we define
\begin{equation*}
\mathfrak{R}_{0,r_t(\cdot)}(\varphi_t(N_0),\alpha):=\varinjlim_{\rho_1\to
  1}\varprojlim_{\rho_2\to 1}D_{r_t([\rho_1,r(\rho_2)])}(\varphi_t(N_0),\alpha)
\end{equation*}
which is in general different from $\mathfrak{R}_0(\varphi_t(N_0),\alpha)$ (in which by
definition we use norms $\rho$ such that
$\|\varphi_t(b_\beta)\|_\rho=\|\varphi_t(b_\alpha)\|_\rho$) by Example
\ref{ex}. Indeed, for $t=s$ the sum $\sum_{n=1}^{\infty}\varphi(b_\beta^nb_\alpha^{-n})$ converges in $\mathfrak{R}_{0,r_t(\cdot)}(\varphi_t(N_0),\alpha)$, but not in $\mathfrak{R}_0(\varphi_t(N_0),\alpha)$.

By the entirely same proof we also obtain
\begin{equation}
\mathfrak{R}_{0,r_{t_1}(\cdot)}(\varphi_{t_1}(N_0),\alpha)=\bigoplus_{n\in
  J(\varphi_{t_1}(N_0)/\varphi_{t_1t_2}(N_0))}n\iota_{t_1,t_1t_2}(\mathfrak{R}_{0,r_{t_1t_2}(\cdot)}(\varphi_{t_1t_2}(N_0),\alpha))\label{dec2}
\end{equation}
for each pair $t_1,t_2\in T_+$ where $\iota_{t_1,t_1t_2}$ is the
inclusion of the rings above induced by the natural inclusion
$\varphi_{t_1t_2}(N_0)\hookrightarrow \varphi_{t_1}(N_0)$.

Now we would like to define continuous homomorphisms
\begin{eqnarray*}
\varphi_{t_2t_1,t_1}\colon
\mathfrak{R}_{0,r_{t_1}(\cdot)}(\varphi_{t_1}(N_0),\alpha)&\to&
\mathfrak{R}_{0,r_{t_1t_2}(\cdot)}(\varphi_{t_1t_2}(N_0),\alpha)\\
\varphi_{t_1}(b_\beta)&\mapsto&\varphi_{t_1t_2}(b_\beta)
\end{eqnarray*}
induced by the group isomorphism $\varphi_{t_2}\colon \varphi_{t_1}(N_0)\to
\varphi_{t_1t_2}(N_0)$ so that we can take the injective limit 
\begin{equation*}
\mathfrak{R}(N_1,\ell):=\varinjlim_{t}\mathfrak{R}_{0,r_t(\cdot)}(\varphi_t(N_0),\alpha)
\end{equation*}
with respect to the maps $\varphi_{t_2t_1,t_1}$. This is not possible for all $t_2$
since the map $\varphi_{t_2}$ will not always be norm-decreasing on
monomials $\bb^{\bk}$ for
$\bk\in\mathbb{Z}^{\{\alpha\}}\times\mathbb{N}^{\Phi^+\setminus\{\alpha\}}$.
To overcome this we define the 
pre-ordering $\leq_\alpha$ (depending on the choice of the simple root $\alpha$)
on $T_+$ the following way: $t_1\leq_\alpha t_2$ if and only if
$|\beta(t_2t_1^{-1})|\leq |\alpha(t_2t_1^{-1})|\leq 1$ for all
$\beta\in\Phi^+$. (In other words if and only if we have
$m(\beta,t_2t_1^{-1})\geq m(\alpha,t_2t_1^{-1})\geq 0$.) In particular, $t_1\leq_\alpha t_2$ implies $t_2t_1^{-1}\in
T_+$ and it is equivalent to $1\leq_\alpha t_2t_1^{-1}$. We also have
$1\leq_\alpha s$ for any $\alpha\in \Delta$. It is clear that $\leq_\alpha$ is
transitive and reflexive. Moreover, if $t_2\leq_\alpha t_1\leq_\alpha t_2$
then $|\beta(t_2t_1^{-1})|=1$ for all $\beta\in\Phi^+$ whence $t_2t_1^{-1}$
lies in $T_0$. Therefore $\leq_\alpha$ defines a partial ordering on the
quotient monoid $T_+/T_0$. 
\begin{lem}\label{rightfilt}
The partial ordering $\leq_\alpha$ on $T_+/T_0$ is right filtered, ie.\ any finite subset of $T_+/T_0$ has a common upper bound with respect to $\leq_\alpha$. 
\end{lem}
\begin{proof}
Let $t_1,t_2\in T_+$ be
arbitrary with $|\alpha(t_1)|\leq |\alpha(t_2)|$. Since the simple roots $\beta\in\Delta$ are linearly independent in $X^*(T)=\Hom_{alg}(T,\mathbb{G}_m)$, and the pairing $X^*(T)\times X_*(T)\to\mathbb{Z}$ is perfect, we may choose
$s_{\overline{\alpha}}\in T$ so that
$|\beta(s_{\overline{\alpha}})|<|\alpha(s_{\overline{\alpha}})|=1$ for all $\alpha\neq\beta\in\Delta$. Since all the positive roots are positive linear combinations of the simple roots, we see immediately that $s_{\overline{\alpha}}\in T_+$. Moreover, if $\alpha\neq\gamma\in\Phi^+$ then $\gamma$ is not a scalar multiple of $\alpha$ hence writing $\gamma=\sum_{\beta\in\Delta}m_{\beta,\gamma}\beta$ there is a $\alpha\neq\beta\in\Delta$ with $m_{\beta,\gamma}>0$ whence $|\gamma(s_{\overline{\alpha}})|<1$. So we have
$t_1\leq_\alpha t_1s_{\overline{\alpha}}^k$ for any $k\geq 0$ and
$t_2\leq_\alpha t_1s_{\overline{\alpha}}^k$ for $k$ big enough.
\end{proof}

Fix an element $1\leq_\alpha t\in T_+$ and let
$p^{-\frac{1}{\max_{\beta\in \Phi^+} p^{m(\beta,t)+m(\alpha,t)}}}<\rho_1<\rho_2<1$ be a real numbers in 
$p^{\mathbb{Q}}$. Note that $\varphi_t\colon
N_0\to\varphi_t(N_0)$ is an isomorphism of pro-$p$ 
groups. Hence it induces an isometric isomorphism
\begin{eqnarray*}
\varphi_t\colon D_{[0,\rho_2^{p^{m(\alpha,t)}}]}(N_0)&\to& D_{[0,\rho_2^{p^{m(\alpha,t)}}]}(\varphi_t(N_0))\\
\sum_{\bf k} c_{\bf k}\prod_\beta b_\beta^{k_\beta}&\mapsto& \sum_{\bf k}
c_{\bf k}\prod_\beta \varphi_t(b_\beta)^{k_\beta}
\end{eqnarray*}
of Banach algebras where $D_{[0,\rho_2^{p^{m(\alpha,t)}}]}(\varphi_t(N_0))$ denotes the completion of $D(\varphi_t(N_0))$ with respect to the $\rho_2^{p^{m(\alpha,t)}}$-norm defined by the set of generators $\{\varphi_t(n_\beta)\}_{\beta\in\Phi^+}$ of $\varphi_t(N_0)$. To avoid confusion, from now on we denote by the subscript $\rho,N_0$ the $\rho$-norm (as before) on $D(N_0)$ and by the subscript $\rho,\varphi_t(N_0)$ the $\rho$-norm on $D(\varphi_t(N_0))$. By Lemma \ref{trivest} we have
\begin{equation*}
\|\varphi_t(b_\beta)\|_{\rho,N_0}=\rho^{p^{m(\beta,t)}}\leq\rho^{p^{m(\alpha,t)}}=\|\varphi_t(b_\alpha)\|_{\rho,N_0}
\end{equation*}
for any $\beta\in\Phi^+$ and $\rho=\rho_1$ or $\rho=\rho_2$ because of our assumption $1\leq_\alpha t$.
This shows that for any monomial
$\prod_{\beta\in\Phi^+}\varphi_t(b_{\beta})^{k_\beta}$ 
(with $k_\beta\geq 0$
for all $\beta\in\Phi^+$) we have
\begin{equation*}
\|\prod_{\beta\in\Phi^+}\varphi_t(b_{\beta})^{k_\beta}\|_{r_t(\rho)}=\|\prod_{\beta\in\Phi^+}\varphi_t(b_{\beta})^{k_\beta}\|_{\rho,N_0}\leq\rho^{p^{m(\alpha,t)}\bk}=\|\prod_{\beta\in\Phi^+}\varphi_t(b_{\beta})^{k_\beta}\|_{\rho^{p^{m(\alpha,t)}},\varphi_t(N_0)}
\end{equation*}
since both norms are multiplicative on $D(\varphi_t(N_0))$. Hence we
obtain a norm decreasing homomorphism
\begin{equation*}
D_{[0,\rho_2^{p^{m(\alpha,t)}}]}(N_0)\overset{\sim}{\rightarrow}
D_{[0,\rho_2^{p^{m(\alpha,t)}}]}(\varphi_t(N_0))\rightarrow
D_{r_t([0,\rho_2])}(\varphi(N_0))\hookrightarrow
D_{r_t([\rho_1,\rho_2])}(\varphi(N_0),\alpha)\ . 
\end{equation*}

Moreover, the element $\varphi_t(b_\alpha)$ is invertible in
$D_{r_t([\rho_1,\rho_2])}(\varphi_t(N_0),\alpha)$ and for each $\rho_1\leq
\rho\leq\rho_2$ and $x\in D_{[0,\rho_2]}(N_0)$ we have
\begin{equation*}
\|\varphi_t(x)\varphi_t(b_\alpha)^{-k}\|_{r_t(\rho)}\leq\|x\|_{\rho^{p^{m(\alpha,t)}},N_0}\|b_\alpha^{-k}\|_{\rho^{p^{m(\alpha,t)}},N_0}\ .
\end{equation*}
Therefore by the universal
property of microlocalisation (Prop.\ A.18 in \cite{SZ}) we obtain a norm
decreasing homomorphism 
\begin{eqnarray}
\varphi_{t,1}\colon D_{[\rho_1^{p^{m(\alpha,t)}},\rho_2^{p^{m(\alpha,t)}}]}(N_0,\alpha)&\to&
D_{r_t([\rho_1,\rho_2])}(\varphi_t(N_0),\alpha)\notag\\ 
b_\beta &\mapsto& \varphi_t(b_\beta)\ .\label{phit1}
\end{eqnarray}
Note that this above map is not surjective in general by Example
\ref{ex}.
\begin{lem}
The map \eqref{phit1} is injective.
\end{lem}
\begin{proof}
Take an element $f(\bb)=\sum_{\bk}d_{\bk}\bb^\bk\in D_{[\rho_1^{p^{m(\alpha,t)}},\rho_2^{p^{m(\alpha,t)}}]}(N_0,\alpha)$ and pairwise distinct  $\bk_1,\dots,\bk_r\in\mathbb{Z}^{\{\alpha\}}\times\mathbb{N}^{\Phi^+\setminus\{\alpha\}}$. Note that $\|\varphi_t(b_\beta)\|_{\rho,N_0}>\|\varphi_t(b_\beta)-\binom{\beta(t)}{p^{m(\beta,t)}}b_\beta^{p^{m(\beta,t)}}\|_{\rho,N_0}$ hence we obtain
\begin{align*}
\|\sum_{j=1}^r d_{\bk_j}\varphi_t(\bb)^{\bk_j}\|_{r_t(\rho_1),r_t(\rho_2)}=\|\sum_{j=1}^r d_{\bk_j}\prod_{\beta\in\Phi^+}\left(\binom{\beta(t)}{p^{m(\beta,t)}}b_\beta^{p^{m(\beta,t)}}\right)^{k_{j,\beta}}\|_{\rho_1,\rho_2}=\\
\max_j\| d_{\bk_j}\prod_{\beta\in\Phi^+}\left(\binom{\beta(t)}{p^{m(\beta,t)}}b_\beta^{p^{m(\beta,t)}}\right)^{k_{j,\beta}}\|_{\rho_1,\rho_2}=\max_j \|d_{\bk_j}\varphi_t(\bb)^{\bk_j}\|_{r_t(\rho_1),r_t(\rho_2)}
\end{align*}
using Prop.\ \ref{4} as we have $\prod_{\beta\in\Phi^+}b_\beta^{p^{m(\beta,t)}k_{j_1,\beta}}\neq \prod_{\beta\in\Phi^+}b_\beta^{p^{m(\beta,t)}k_{j_2,\beta}}$ for $1\leq j_1\neq j_2 \leq r$.

Since the map $\varphi_{t,1}$ is norm decreasing, we have $\|d_\bk\varphi_t(\bb)^\bk\|_{r_t(\rho_1),r_t(\rho_2)}\to 0$ as $\bk\to\infty$. Therefore we also have $\|\sum_{\bk}d_{\bk}\varphi_t(\bb)^\bk\| _{r_t(\rho_1),r_t(\rho_2)}=\max_{\bk} \|d_{\bk}\varphi_t(\bb)^{\bk}\|_{r_t(\rho_1),r_t(\rho_2)}$ which is nonzero if there exists a $\bk$ with $d_{\bk}\neq 0$. Therefore the injectivity.
\end{proof}

Taking projective and injective limits we obtain an injective ring 
homomorphism
\begin{equation*}
\varphi_{t,1}\colon\mathfrak{R}_0(N_0,\alpha)\hookrightarrow\mathfrak{R}_{0,r_t(\cdot)}(\varphi_t(N_0),\alpha)
\end{equation*}
for any $1\leq_\alpha t\in T_+$.

\begin{rem}
Note that $\mathfrak{R}_{0,r_t(\cdot)}(\varphi_t(N_0),\alpha)$ is a subring of $\mathfrak{R}_0(N_0,\alpha)$ via the map $\iota_{1,t}$ (for all $t\in T_+$). Hence for $1\leq_\alpha t$ we obtain a ring homomorphism $\varphi_t=\iota_{1,t}\circ\varphi_{t,1}\colon \mathfrak{R}_0(N_0,\alpha)\to \mathfrak{R}_0(N_0,\alpha)$. However, if $1\not\leq_\alpha t$ for some $t\in T_+$ then we in fact do not have a continuous ring homomorphism $\varphi_t\colon
\mathfrak{R}_0(N_0,\alpha)\to\mathfrak{R}_0(N_0,\alpha)$. Indeed, in this case there exists a $\beta\in \Phi^+$ such that $|\beta(t)|>|\alpha(t)|$ so there exist integers $k_\beta>k_\alpha$ such that $\|\varphi_t(b_\beta^{k_\beta}b_\alpha^{-k_\alpha})\|_{\rho}=\rho^{k_\beta/|\beta(t)|-k_\alpha/|\alpha(t)|}>1$ for any $p^{-|\alpha(t)|}<\rho<1$ therefore $\sum_{n=1}^{\infty}\varphi_t(b_\beta^{nk_\beta}b_\alpha^{-nk_\alpha})$ does not converge in $\mathfrak{R}_0(N_0,\alpha)$ even though $\sum_{n=1}^{\infty}b_\beta^{nk_\beta}b_\alpha^{-nk_\alpha}$ does. 
\end{rem}

\begin{rem}\label{notetale}
If $\Phi^+\neq \Delta$ (e.g.\ if $G=\GL_n(\mathbb{Q}_p)$, $n>2$) then we have
\begin{equation*}
\mathfrak{R}_0(N_0,\alpha)=\bigoplus_{u\in
  J(N_0/\varphi_t(N_0))}u\iota_{1,s}(\mathfrak{R}_{0,r_s(\cdot)}(\varphi(N_0),\alpha))\supsetneq \bigoplus_{n\in
  J(N_0/\varphi(N_0))}u\varphi(\mathfrak{R}_{0}(N_0,\alpha))
\end{equation*}
by \eqref{dec1} (with the choice $t=s$) and Example \ref{ex} (which shows that $\varphi_{s,1}$ is not surjective).
\end{rem}

In a similar fashion we get for $t_1\in T_+$ (and $1\leq_\alpha t\in T_+$) an injective
homomorphism
\begin{equation*}
\varphi_{tt_1,t_1}\colon\mathfrak{R}_{0,r_{t_1}(\cdot)}(\varphi_{t_1}(N_0),\alpha)\to\mathfrak{R}_{0,r_{tt_1}(\cdot)}(\varphi_{tt_1}(N_0),\alpha)\ .
\end{equation*}
In view of Lemma \ref{rightfilt} we define
\begin{equation*}
\mathcal{R}(N_1,\ell):=\varinjlim_{t\in T_+}\mathfrak{R}_{0,r_t(\cdot)}(\varphi_t(N_0),\alpha)
\end{equation*}
with respect to the maps $\varphi_{t_1,t_2}$ for $t_2\leq_\alpha t_1$.

Now take any $t\in T_+$ (not necessarily satisfying $1\leq_\alpha t$). The map 
\begin{equation*}
\varphi_t:=\varinjlim_{t_1}\iota_{t_1,tt_1}\colon\mathcal{R}(N_1,\ell)\to\mathcal{R}(N_1,\ell) 
\end{equation*}
is defined as the direct limit of the inclusion maps
\begin{equation*}
\iota_{t_1,tt_1}\colon \mathfrak{R}_{0,r_{tt_1}(\cdot)}(\varphi_{tt_1}(N_0),\alpha)\hookrightarrow 
\mathfrak{R}_{0,r_{t_1}(\cdot)}(\varphi_{t_1}(N_0),\alpha)
\end{equation*}
induced by $\varphi_{tt_1}(N_0)\subseteq \varphi_{t_1}(N_0)$. By definition, for any $t\in T_+$ the ring $\mathfrak{R}_{0,r_t(\cdot)}(\varphi_t(N_0),\alpha)$ consists of formal power series $\sum_{\bf k}c_{\bf k}\varphi_t({\bf b})^{\bf k}$ that converge in $\mathfrak{R}_0(N_0,\alpha)$. Therefore the map
\begin{align*}
\bigcup_{t\in T_+}\left\{
\sum_{{\bf k}\in\mathbb{Z}^{\{\alpha\}}\times\mathbb{N}^{\Phi^+\setminus\{\alpha\}}}c_{\bf
  k}{\bf b}^{\bf k}
        \text{ }{\Big |} \text{ }
\sum_{\bf k}c_{\bf k}\varphi_t({\bf b})^{\bf k}\text{
convergent in }\mathfrak{R}_0(N_0,\alpha)\right\}\to\mathcal{R}(N_1,\ell)\\
\sum_{\bf k}c_{\bf
  k}{\bf b}^{\bf k}\mapsto\sum_{\bf k}c_{\bf k}\varphi_t({\bf b})^{\bf k}\in \mathfrak{R}_{0,r_t(\cdot)}(\varphi_t(N_0),\alpha)\hookrightarrow\mathcal{R}(N_1,\ell)
\end{align*}
is well-defined and bijective since $\sum_{\bf k}c_{\bf k}\varphi_t({\bf b})^{\bf k}$ converges for some $t\in T_+$ and the connecting homomorphisms in the injective limit defining $\mathcal{R}(N_1,\ell)$ are injective and given by $\varphi_{t_1,t_2}$ for $t_2\leq_\alpha t_1$.

Hence we may identify
\begin{equation}
\mathcal{R}(N_1,\ell)=\bigcup_{t\in T_+}\left\{
\sum_{{\bf k}\in\mathbb{Z}^{\{\alpha\}}\times\mathbb{N}^{\Phi^+\setminus\{\alpha\}}}c_{\bf
  k}{\bf b}^{\bf k}
        \text{ }{\Big |} \text{ }
\sum_{\bf k}c_{\bf k}\varphi_t({\bf b})^{\bf k}\text{
convergent in }\mathfrak{R}_0(N_0,\alpha)\right\}
\label{expandR}
\end{equation}
and obtain
\begin{pro}
The natural map $\varphi_t\colon\mathcal{R}(N_1,\ell)\to \mathcal{R}(N_1,\ell)$
is injective for all $t\in T_+$ and we have the decomposition
\begin{equation*}
\mathcal{R}(N_1,\ell)=\bigoplus_{n\in J(N_0/\varphi_t(N_0))}n\varphi_t(\mathcal{R}(N_1,\ell))\ .
\end{equation*}
In particular, $\mathcal{R}(N_1,\ell)$ is a free (right) module over itself
via $\varphi_t$ and it is a $\varphi$-ring over $N_0$ with $\varphi=\varphi_s$ in the sense of Definition \ref{phiringH_0}.
\end{pro}
\begin{proof}
By \eqref{dec2} we have 
\begin{equation*}
\mathfrak{R}_{0,r_{t_1}(\cdot)}(\varphi_{t_1}(N_0),\alpha)=\bigoplus_{n\in J(N_0/\varphi_t(N_0))}\varphi_{t_1}(n)\iota_{t_1,tt_1}(\mathfrak{R}_{0,r_{tt_1}(\cdot)}(\varphi_{tt_1}(N_0),\alpha))
\end{equation*}
for any $t_1\in T_+$. The statement follows by taking the injective limit of
both sides (with respect to $t_1$) and noting that $\varphi_{t_1,1}(n)=\varphi_{t_1}(n)\in\varphi_{t_1}(N_0)\subseteq \mathfrak{R}_{0,r_{t_1}(\cdot)}(\varphi_{t_1}(N_0),\alpha)$ for $n\in N_0\subseteq \mathfrak{R}_0(N_0)$ and $1\leq_\alpha t_1$ therefore $n$ corresponds to $\varinjlim_{1\leq_\alpha t_1}(\varphi_{t_1}(n))_{t_1}$ via the identification \eqref{expandR}. 
\end{proof}

\begin{rem}\label{expandRell}
The ring $\mathcal{R}(N_1,\ell)$ via the description \eqref{expandR} consists of exactly those Laurent-series $$x=\sum_{{\bf k}\in\mathbb{Z}^{\{\alpha\}}\times\mathbb{N}^{\Phi^+\setminus\{\alpha\}}}c_{\bf
  k}{\bf b}^{\bf k}$$ that converge on the open annulus of the form
\begin{equation}
\left\{\rho_2<|z_\alpha|<1,\ |z_{\beta}|\leq |z_\alpha|^r\text{ for
}\beta\in\Phi^+\setminus\{\alpha\}\right\}\ .\label{annulus}
\end{equation}
for some $p^{-1}<\rho_2<1$ and $1\leq r\in\mathbb{Z}$.
\end{rem}
\begin{proof}
If $x\in \mathcal{R}(N_1,\ell)$ then there exists a $t\in T_+$ such that $\varphi_t(x)$ converges in $\mathfrak{R}_0(N_0)$, ie.\ it converges in the norm $\|b_\beta\|_\rho=\rho$ for all $\beta\in \Phi^+$ for some fixed $p^{-1}<\rho_0<1$ and all $\rho\in(\rho_0,1)$. By Lemma \ref{rightfilt} we may assume that $|\alpha(t)|=1$ whence $\|\varphi_t(b_\alpha)\|_\rho=\rho$ for all $\rho<1$ as we may take $t=s_{\overline{\alpha}}^k$ for $k$ large enough. Now let $\rho_2:=\rho_0$ and $r:=\max_{\beta\in\Phi^+}([|1/\beta(t)|]+1)\in\mathbb{Z}$. Then $x$ converges on the annulus \eqref{annulus} as we have $\rho^r\leq \rho^{1/|\beta(t)|}\leq\|\varphi_t(b_{\beta})\|_\rho$ for all $\beta\in\Phi^+\setminus\{\alpha\}$ by Lemma \ref{trivest}.

Conversely for any fixed $p^{-1}<\rho_2<1$ and integer $r\geq 1$ we need to find a $t\in T_+$ and a $\rho_0\in (p^{-1},1)$ such that for all $\rho\in (\rho_0,1)$ we have $\rho_2<\|\varphi_t(b_\alpha)\|_{\rho}<1$ and $\|\varphi_t(b_\beta)\|_{\rho}\leq \|\varphi_t(b_\alpha)\|_{\rho}^r$. We take $t:=s_{\overline{\alpha}}^k$ and $\rho_0:=\max(\rho_2,p^{-|\beta(t)|}\mid \beta\in\Phi^+\setminus\{\alpha\})$ where $k:=\max_{\beta\in\Phi^+\setminus\{\alpha\}}([-\frac{\log r}{\log |\beta(s_{\overline{\alpha}})|}]+1)$ (for the definition of $s_{\overline{\alpha}}$ see the proof of Lemma \ref{rightfilt}). Indeed, since $|\alpha(s_{\overline{\alpha}}^k)|$ equals $1$, we have $\rho_2<\rho=\|\varphi_{s_{\overline{\alpha}}^k}(b_\alpha)\|_{\rho}<1$ (for any $k$). On the other hand, we have $|\beta(s_{\overline{\alpha}})|<1$ for all $\alpha\neq\beta\in\Phi^+$ (whence, in particular, the definition of $k$ makes sense), so we obtain $$\|\varphi_{t}(b_\beta)\|_{\rho}=\max\limits_{0\leq j\leq \mathrm{val}_p(\beta(t))}(\rho^{p^j}p^{j-\mathrm{val}_p(\beta(t))})=\rho^{p^{\mathrm{val}_p(\beta(t))}}=\rho^{1/|\beta(s_{\overline{\alpha}})|^k}\leq\rho^r$$ for all $\beta\in\Phi^+\setminus\{\alpha\}$ by Lemma \ref{trivest}, the choice of $k$ so that we have $r\leq 1/|\beta(s_{\overline{\alpha}})|^k$, and the choice of $p^{-|\beta(t)|}<\rho$ so that we have $\max\limits_{0\leq j\leq \mathrm{val}_p(\beta(t))}(\rho^{p^j}p^{j-\mathrm{val}_p(\beta(t))})=\rho^{p^{\mathrm{val}_p(\beta(t))}}$.
\end{proof}

\subsection{Bounded rings}

Let us denote by $\mathfrak{R}^b_0$ (resp.\ by $\mathfrak{R}^{int}_0$) the set of elements $x\in\mathfrak{R}_0$ such that $\lim_{\rho\to1}\|x\|_{\rho,\rho}$ exists (resp.\ exists and is at most $1$). These are subrings of $\mathfrak{R}_0$ by Prop.\ A.28 in \cite{SZ}. Moreover, since $\varphi_t$ is norm-decreasing for any $1\leq_\alpha t$ (see \eqref{phit1}), these subrings are stable under the action of $\varphi_t$ ($1\leq_\alpha t\in T_+$). We put 
\begin{eqnarray*}
\mathfrak{R}_{0,r_t(\cdot)}^b(\varphi_t(N_0),\alpha)&:=&\mathfrak{R}_{0,r_t(\cdot)}(\varphi_t(N_0),\alpha)\cap\mathfrak{R}_{0}^b(N_0,\alpha)\ ,\quad\text{and}\\
\mathfrak{R}_{0,r_t(\cdot)}^{int}(\varphi_t(N_0),\alpha)&:=&\mathfrak{R}_{0,r_t(\cdot)}(\varphi_t(N_0),\alpha)\cap\mathfrak{R}_{0}^{int}(N_0,\alpha) 
\end{eqnarray*}
where the intersection is taken inside $\mathfrak{R}_0$ under the inclusion $\iota_{1,t}\colon\mathfrak{R}_{0,r_t(\cdot)}(\varphi_t(N_0),\alpha)\hookrightarrow\mathfrak{R}_0$. Hence
\begin{equation*}
\mathcal{R}^b(N_1,\ell):=\varinjlim_t\mathfrak{R}^b_{0,r_t(\cdot)}(\varphi_t(N_0),\alpha)\text{ and }\mathcal{R}^{int}(N_1,\ell):=\varinjlim_t\mathfrak{R}^{int}_{0,r_t(\cdot)}(\varphi_t(N_0),\alpha)
\end{equation*}
are $T_+$-stable subrings of $\mathcal{R}(N_1,\ell)$ (the injective limit is taken with repsect to the maps $\varphi_{t_1,t_2}$ for $t_1\leq_\alpha t_2\in T_+$ as in the construction of $\mathcal{R}(N_1,\ell)$). Further, Lemma \ref{equivnorm} shows that for any $t\in T_+$ and $x\in \mathfrak{R}_0$ we have 
\begin{equation}\label{lims}
\lim_{\rho\to1}\|x\|_{\rho}=\lim_{\rho\to1}\|x\|_{q_t(r_t(\rho))}\ . 
\end{equation}
Indeed, we may use Lemma \ref{equivnorm} in the context of $\mathfrak{R}_0$ the following way. The elements of $D_{[\rho_1,\rho_2]}(N_0,\alpha)$ are Cauchy sequences $(a_n\varphi_t(b_\alpha)^{-k_n})_{n\in\mathbb{N}}$ (in the norm $\max(\|\cdot\|_{\rho_1},\|\cdot\|_{\rho_2})$) with $a_n\in D_{[0,\rho_2]}(N_0)$ and $k_n\geq 0$. Since $\|\cdot\|_\rho$ is multiplicative for any $\rho_1\leq\rho\leq\rho_2$ in $p^{\mathbb{Q}}$ and so is its restriction to $D(\varphi_t(N_0))$ we compute
\begin{align*}
\|a_n\varphi_t(b_\alpha)^{-k_n}\|_{\rho}\rho^{\sum_{\beta\in\Phi^+}(p^{m(\beta,t)}-1)}=\frac{\|a_n\|_\rho}{\|\varphi_t(b_\alpha)^{k_n}\|_{\rho}\rho^{-\sum_{\beta\in\Phi^+}(p^{m(\beta,t)}-1)}}\leq\frac{\|a_n\|_{q_t(r_t(\rho))}}{\|\varphi_t(b_\alpha)^{k_n}\|_{q_t(r_t(\rho))}}=\\
=\|a_n\varphi_t(b_\alpha)^{-k_n}\|_{q_t(r_t(\rho))}\leq	\frac{\|a_n\|_{\rho}\rho^{-\sum_{\beta\in\Phi^+}(p^{m(\beta,t)}-1)}}{\|\varphi_t(b_\alpha)^{k_n}\|_{\rho}}
\leq \|a_n\varphi_t(b_\alpha)^{-k_n}\|_{\rho}\rho^{-\sum_{\beta\in\Phi^+}(p^{m(\beta,t)}-1)}\ .
\end{align*}
If $\rho\to 1$ and $n\to\infty$ we obtain \eqref{lims}. Combining this observation with \eqref{dec1} we obtain
\begin{align*}
\mathfrak{R}^b_0(N_0,\alpha)=\bigoplus_{n\in
  J(N_0/\varphi_t(N_0))}n\iota_{1,t}(\mathfrak{R}^b_{0,r_t(\cdot)}(\varphi_t(N_0),\alpha))\ ;\\
\mathfrak{R}^{int}_0(N_0,\alpha)=\bigoplus_{n\in
  J(N_0/\varphi_t(N_0))}n\iota_{1,t}(\mathfrak{R}^{int}_{0,r_t(\cdot)}(\varphi_t(N_0),\alpha))\ .
\end{align*}
So by a similar argument as for $\mathcal{R}(N_1,\ell)$ we also obtain
\begin{align*}
\mathcal{R}^b(N_1,\ell)=\bigoplus_{n\in J(N_0/\varphi_t(N_0))}n\varphi_t(\mathcal{R}^b(N_1,\ell))\ ;\\
\mathcal{R}^{int}(N_1,\ell)=\bigoplus_{n\in J(N_0/\varphi_t(N_0))}n\varphi_t(\mathcal{R}^{int}(N_1,\ell)),
\end{align*}
in other words these are $\varphi$-rings over $N_0$ in the sense of Definition \ref{phiringH_0}.

\begin{rem}\label{int}
Note that by Lemma A.27 in \cite{SZ} an element $\sum_{{\bf k}\in\mathbb{N}^{\Phi^+\setminus\{\alpha\}}\times\mathbb{Z}}c_{\bf
  k}{\bf b}^{\bf k}\in \mathcal{R}(N_1,\ell)$ (under the parametrization \eqref{expandR}) lies in $\mathcal{R}^b(N_1,\ell)$ (resp.\ in $\mathcal{R}^{int}(N_1,\ell)$) if and only if $|c_{\bf k}|$ is bounded (resp.\ $\leq 1$) for ${\bf k}\in\mathbb{Z}^{\{\alpha\}}\times\mathbb{N}^{\Phi^+\setminus\{\alpha\}}$.
\end{rem}

\subsection{Relation with the completed Robba ring and overconvergent ring}\label{relate}

\begin{lem}\label{intmult}
There exists a continuous (in the weak topology of $\Lambda_\ell(N_0)$) injective ring homomorphism $j_{int}\colon\mathcal{R}^{int}(N_1,\ell)\to\Lambda_\ell(N_0)$ respecting Laurent series expansions. The image of $j_{int}$ is contained in $\mathcal{O}_{\mathcal{E}}^\dagger\bs N_1,\ell\js\subset \Lambda_\ell(N_0)$.
\end{lem}
\begin{proof}
We proceed in 3 steps. In Step 1 we construct a map $j_{int,0}= j_{int\mid\mathfrak{R}_0^{int}}\colon\mathfrak{R}^{int}_0\to \Lambda_\ell(N_0)$ which is a priori continuous and $o_K$-linear. In Step 2 we show that $j_{int,0}$ is multiplicative hence a ring homomorphism. In Step 3 we extend it to $\mathcal{R}^{int}(N_1,\ell)$ and show that the image lies in $\mathcal{O}_\mathcal{E}^\dagger\bs N_1,\ell\js\subset\mathcal{O_E}\bs N_1,\ell\js=\Lambda_\ell(N_0)$.

\emph{Step 1.} By Lemma \ref{4} and Remark \ref{int} we may write any element in $\mathfrak{R}^{int}_0$ in a Laurent series expansion $\sum_{{\bf k}\in\mathbb{N}^{\Phi^+\setminus\{\alpha\}}\times\mathbb{Z}}c_{\bf
  k}{\bf b}^{\bf k}$ with coefficients $c_{\bf k}$ in $o_K$. So we may collect all the terms containing $b_\alpha^{k_\alpha}$ for some fixed $k_\alpha$ into an element of the Iwasawa algebra $\Lambda(N_1)$ to obtain an expansion $\sum_{n\in\mathbb{Z}}b_\alpha^nf_n$ with $f_n\in\Lambda(N_1)$. These power series satisfy the convergence property that there exists a real number $p^{-1}<\rho_1<1$ such that $\rho^n\|f_n\|_{\rho}\to0$ as $|n|\to\infty$ for all $\rho_1<\rho<1$. In particular, if $n\to-\infty$ then $f_n\to 0$ in the compact topology of $\Lambda(N_1)$. Hence the sum $\sum_nb_\alpha^nf_n$ also converges in $\Lambda_\ell(N_0)$. This way we obtained a right $\Lambda(N_0)$-linear injective map $j_{int,0}\colon\mathfrak{R}_0^{int}\to\Lambda_\ell(N_0)$. 

Recall that the weak topology (see \cite{SVe}, \cite{SVi}, \cite{SVZ} for instance) on $\Lambda_\ell(N_0)$ is defined by the open neighbourhoods of $0$ of the form $\mathcal{M}(r)=\mathcal{M}_\ell(N_0)^r+\mathcal{M}(N_0)^r$ where $\mathcal{M}_\ell(N_0)=\Lambda_\ell(N_0)\mathcal{M}(N_1)$ denotes the maximal ideal of $\Lambda_\ell(N_0)$ and $\mathcal{M}(N_i)$ denotes the maximal ideal of $\Lambda(N_i)\subseteq \Lambda_\ell(N_0)$ ($i=0,1$). For any fixed $p^{-1}<\rho_1<\rho<1$ the preimage of $\mathcal{M}(r)$ in $\mathfrak{R}^{int}_{0}\cap D_{[\rho_1,1)}(N_1,\alpha)$ contains the open ball $\{x\mid \|x\|_\rho<p^{-r}\}$. Indeed, if $x=\sum_{n\in\mathbb{Z}}b_\alpha^nf_n$ then for any $n<0$ we have $\|f_n\|_\rho<p^{-r}$ hence $f_n\in \mathcal{M}(N_1)^r$ and $b_\alpha^nf_n\in\mathcal{M}_\ell(N_0)$. On the other hand, the positive part $\sum_{n\geq 0}b_\alpha^nf_n$ lies in $\Lambda(N_0)$ and has $\rho$-norm smaller than $p^{-r}$ therefore lies in $\mathcal{M}(N_0)^r$. Hence the continuity.

\emph{Step 2.} Now by the continuity and linearity of $j_{int,0}$ it suffices to show that it is multiplicative on monomials ${\bf b}^{\bf k}$. Moreover, each monomial is a linear combination of elements of the form $b_\alpha^ng$ with $g\in N_0$. In order to expand the product $(b_\alpha^{n_1}g_1)(b_\alpha^{n_2}g_2)$ into a skew Laurent series it suffices to expand $g_1b_\alpha^{n_2}$ with $n_2<0$. However, if $g_1b_\alpha^{n_2}=\sum_n b_\alpha^nh_n$ is the expansion in $\mathcal{R}^{int}(N_1,\ell)$ then $\sum_{|n|<n_0}b_\alpha^nh_nb_\alpha^{-n_2}$ tends to $g_1$ (as $n_0\to+\infty$) in the topology of $\mathfrak{R}^{int}_0$ (induced by the norms) hence also in the weak topology. Therefore the expansion in $\Lambda_\ell(N_0)$ is also $g_1b_\alpha^{n_2}=\sum_n b_\alpha^nh_n$. So the above constructed map $j_{int,0}$ is indeed a ring homomorphism as claimed.

\emph{Step 3.} Finally, take an element $x\in\mathcal{R}^{int}(N_1,\ell)$. There exists an element $1\leq_\alpha
t\in T_+$ such that $\varphi_t(x)$ lies in the image of the composite map 
\begin{equation*}
\mathfrak{R}^{int}_{0,r_t(\cdot)}(\varphi_t(N_0),\alpha)\hookrightarrow
\mathfrak{R}^{int}_0(N_0,\alpha)\hookrightarrow \mathcal{R}^{int}(N_1,\ell)
\end{equation*}
where the first arrow is induced by the inclusion $\varphi_t(N_0)\subseteq
N_0$. Now if we reduce
$j_{int,0}(\varphi_t(x))\in\Lambda_\ell(N_0)$
modulo the ideal generated by $N_l-1$ for some integer $l\geq 1$ then we
obtain an element in $\varphi_t(\mathcal{O}_\mathcal{E}^\dagger[N_1/N_l,\ell])$. Indeed, $\varphi_t(\mathcal{O_E}[N_1/N_l,\ell])$ is a closed subspace in $\mathcal{O_E}[N_1/N_l,\ell]$ and all the monomials $j_{int,0}(\varphi_t({\bf b}^{\bf k}))$ map into this subspace under the reduction modulo $(N_l-1)$. Hence the image lies in $\varphi_t(\mathcal{O_E}[N_1/N_l,\ell]$. Moreover, by the convergence property of elements in $\mathfrak{R}^{int}_0$, we may expand 
\begin{equation*}
\varphi_t(x)=\sum_{n\in\mathbb{Z}}b_\alpha^nf_n 
\end{equation*}
with $f_n\in\Lambda(N_1)$ and $\rho^n\|f_n\|_\rho\to0$ as $n\to\infty$ for all $\rho_1<\rho<1$ and a fixed $p^{-1}<\rho_1<1$ depending on $x$. Since the reduction map $\Lambda(N_1)\to o[N_1/N_l]$ is continuous in the $\rho$-norm, we obtain that the reduction of $j_{int,0}(\varphi_t(x))$ modulo $(N_l-1)$ also lies in $\mathcal{O}_\mathcal{E}^\dagger[N_1/N_l,\ell]$. Hence we have $j_{int,0}(\varphi_t(x))\pmod{N_l-1}\in\varphi_t(\mathcal{O}_\mathcal{E}^\dagger[N_1/N_l,\ell])=\varphi_t(\mathcal{O_E}[N_1/N_l,\ell])\cap \mathcal{O}_\mathcal{E}^\dagger[N_1/N_l,\ell]$. Taking the limit we see (using \eqref{lim}) that $j_{int,0}(\varphi_t(x))$ lies in
\begin{equation*}
\varprojlim_l\varphi_t(\mathcal{O}_\mathcal{E}^\dagger[N_1/N_l,\ell])=\varphi_t(\mathcal{O}_\mathcal{E}^\dagger\bs
N_1,\ell\js)\ .
\end{equation*}
So we put $j_{int}(x):=\varphi_t^{-1}(j_{int,0}(\varphi_t(x)))$. This extends the ring homomorphism $j_{int,0}$ to a continuous ring homomorphism $j_{int}\colon\mathcal{R}^{int}(N_1,\ell)\hookrightarrow\mathcal{O}_\mathcal{E}^\dagger\bs N_1,\ell\js\subset\Lambda_\ell(N_0)$ by Lemma \ref{rightfilt}. Moreover, $j_{int}$ is $T_+$-equivariant as it respects power series expansions.
\end{proof}

Now the following proposition compares $\mathcal{R}(N_1,\ell)$ with the previous
construction $\mathcal{R}\bs N_1,\ell\js$. 

\begin{pro}\label{j}
There exists a natural $T_{+}$-equivariant ring homomorphism
\begin{equation*}
j\colon\mathcal{R}(N_1,\ell)\to \mathcal{R}\bs N_1,\ell\js
\end{equation*}
with dense image.
\end{pro}
\begin{proof}
At first we construct the map $j_0=j_{\mid\mathfrak{R}_0}$ on $\mathfrak{R}_0\subset \mathcal{R}(N_1,\ell)$ with dense image. We are
going to show that for any open characteristic subgroup $H\leq N_1$ we have an
isomorphism
$\mathfrak{R}_0/\mathfrak{R}_0(H-1)\cong\mathcal{R}[N_1/H,\ell]$. Note that
$N_1$ being a compact $p$-adic Lie group, $N_1$ has a system of neighbourhoods
of $1$ consisting of open uniform characteristic subgroups (in fact $N_1$ is 
uniform---since so is $N_0$ by assumption---and one can take repeatedly the
Frattini subgroups of $N_1$ which are characteristic subgroups, ie.\ stable under all the continuous automorphisms of $N_1$). So we may assume without loss of generality
that $H$ is uniform with topological generators $h_1,h_2,\dots,h_d$ with $d=\dim N_1$ as a $p$-adic Lie group.

Under the parametrization in Prop.\ \ref{4} the elements of $\mathfrak{R}_0$
can be written as power series $\sum_{n\in\mathbb{Z}}b_\alpha^nf_n$
with $f_n\in D(N_1,K)$ and the convergence property that there exists a real number $\rho_1<1$ such that
$\rho^n\|f_n\|_\rho\rightarrow 0$ (as $|n|\rightarrow\infty$) for all
$\rho_1\leq\rho<1$. Now note that we have $D(N_1,K)=\bigoplus_{u\in J(N_1/H)}uD(H,K)$. Hence the right ideal $D(N_1,K)(H-1)$ in $D(N_1,K)$ is generated by the elements $h_i-1$ for $1\leq i\leq d$ and it is the kernel of the natural projection $\pi_H\colon D(N_1,K)\to D(N_1/H)=K[N_1/H]$. Moreover, this quotient map factors through the inclusion $D(N_1,K)\hookrightarrow D_{[0,\rho]}(N_1,K)$ for any $p^{-1}<\rho<1$. Hence $\rho^n\|\pi_H(f_n)\|\to 0$ where $\|x\|:=\max_u |x_u|$ with $x=\sum_{u\in N_1/H}x_uu$, $x_u\in K$. Therefore we obtain a map
\begin{eqnarray*}
\pi_H\quad\colon\quad\mathfrak{R}_0&\to&\mathcal{R}[N_1/H,\ell]=\bigoplus_{u\in N_1/H}\mathcal{R}u\\
\sum_{n\in\mathbb{Z}}b_\alpha^nf_n&\mapsto&\sum_{u\in N_1/H}\sum_{n\in\mathbb{Z}}\pi_H(f_n)_uT^nu\ .
\end{eqnarray*}
A priori this map is only known to be $K$-linear, continuous, and surjective between topological $K$-vectorspaces. So for the multiplicativity it suffices to show that $\pi_H({\bf b}^{{\bf k}_1}{\bf b}^{{\bf k}_2})=\pi_H({\bf b}^{{\bf k}_1})\pi_H({\bf b}^{{\bf k}_2})$ for monomials ${\bf b}^{{\bf k}_i}$ with ${\bf k}_i\in\mathbb{N}\times\mathbb{Z}^{d}$ ($i=1,2$). On the other hand, these monomials are contained in the subring $\mathfrak{R}_0^{int}$. By Lemma \ref{intmult} we have a commutative diagram 
\begin{equation*}
\xymatrix{
\mathfrak{R}_0^{int}\ar[r] \ar[d]_{\pi_{H,\mathcal{O}_\mathcal{E}^\dagger\bs N_1,\ell\js}} &\mathfrak{R}_0 \ar[d]_{\pi_{H}}\\
\mathcal{O}_\mathcal{E}^\dagger [N_1/H,\ell]\ar[r] & \mathcal{R}[N_1/H,\ell]
}
\end{equation*}
of $o$-modules such that all the maps are ring homomorphisms except possibly for $\pi_H$. However, from the commutativity of the diagram it follows that also $\pi_H$ is multiplicative on monomials therefore a ring homomorphism. By taking the projective limit of maps $\pi_H$ we obtain a ring homomorphism $j_0\colon\mathfrak{R}_0\to \mathcal{R}\bs N_1,\ell\js$ with dense image and extending $j_{int,0}\colon\mathfrak{R}_0^{int}\hookrightarrow \mathcal{O}_\mathcal{E}^\dagger$.

Finally, the homomorphism $j_0$ is extended to $\mathcal{R}(N_1,\ell)$ as in the proof of Lemma \ref{intmult}. The $T_+$-equivariance is clear on monomials by Lemma \ref{intmult} and follows in general from the continuity and linearity.
\end{proof}

\begin{rem}\label{log}
The above constructed map $j\colon\mathcal{R}(N_1,\ell)\to\mathcal{R}\bs N_1,\ell\js$ is not injective in general. Indeed, for any root $\beta\neq\alpha$ in $\Phi^+$ the element $\log(n_\beta)=\log(1+b_\beta)$ lies in $D(N_1)\subset \mathcal{R}(N_1,\ell)$. It is easy to see that $\log(1+b_\beta)$ is divisible by $\varphi^r(b_\beta)$ for any nonnegative integer $r$. Indeed, we clearly have $b_\beta\mid\log(1+b_\beta)$. Applying $\varphi^r$ on the both sides of the divisibility we obtain 
\begin{equation*}
\varphi^r(b_\beta)\mid\varphi^r(\log(1+b_\beta))=\log(1+b_\beta)^{p^{rm_\beta}}=p^{rm_\beta}\log(1+b_\beta)\mid\log(1+b_\beta)
\end{equation*}
as $p^{rm_\beta}$ is invertible in $\mathcal{R}$. Therefore $\log(1+b_\beta)$ lies in the kernel of $\pi_H$ for all $H=N_r$ hence also in the kernel of $j$.
\end{rem}

\begin{rem}
Note that via the inclusion $\mathcal{O}_{\mathcal{E}}^\dagger\subseteq
\mathcal{R}$ we also have $\mathcal{O}_{\mathcal{E}}^\dagger\bs
N_1,\ell\js\subseteq \mathcal{R}\bs N_1,\ell\js$. However, if $N_1\neq 1$ then we have $j_{int}(\mathcal{R}^{int}(N_1,\ell))\neq j(\mathcal{R}(N_1,\ell))\cap \mathcal{O}_\mathcal{E}^\dagger\bs N_1,\ell\js\subset\mathcal{R}\bs N_1,\ell\js$.
\end{rem}
\begin{proof}
Assume $N_1\neq 1$, so we have a positive root $\beta\neq\alpha\in\Phi^+$. We proceed in 3 steps. In Step 1 we are going to construct an element $x\in\mathcal{R}(N_1,\ell)$ with several properties. In Step 2 we are going to show that $j(x)$ lies in $\mathcal{O}_\mathcal{E}^\dagger\bs N_1,\ell\js\subset\mathcal{R}\bs N_1,\ell\js$. In Step 3 we prove that $j(x)$ does not lie in $j_{int}(\mathcal{R}^{int}(N_1,\ell))$. Note that the other inclusion $j_{int}(\mathcal{R}^{int}(N_1,\ell))\subset j(\mathcal{R}(N_1,\ell))\cap \mathcal{O}_\mathcal{E}^\dagger\bs N_1,\ell\js$ is obvious. 

\emph{Step 1.} We denote by $s_n:=\sum_{i=1}^n(-1)^{i+1}\frac{b_\beta^i}{i}$ the $n$th estimating sum of $\log(1+b_\beta)\in\mathcal{R}(N_1,\ell)$. Note that $k_n:=[\log_pn]$ is the smallest positive integer such that 
\begin{equation}\label{k_n}
p^{k_n}s_n\in\mathbb{Z}_p[N_{\beta,0}]\subseteq \mathcal{R}(N_1,\ell)
\end{equation}
where $[\cdot]$ denotes the integer part of a real number. We further choose a sequence of real numbers $p^{-1}<\rho_1<\dots<\rho_n<\dots<1$ in $p^{\mathbb{Q}}$ such that $\lim_{n\to\infty}\rho_n=1$. Now for any fixed positive integer $n$ let $i_n$ be the smallest positive integer satisfying the following properties
\begin{align}
\log_{\rho_{n-1}}(\|p^{k_{i_{n-1}}}\log(1+b_\beta)\|_{\rho_{n-1}})+1<\log_{\rho_n}(\|p^{k_{i_n}}\log(1+b_\beta)\|_{\rho_n})\ ;\notag\\
\frac{\|\log(1+b_\beta)\|_{\rho_n}}{p^{n}}>\|\log(1+b_\beta)-s_{i_n}\|_{\rho_n}\ ;\label{i_n}\\
p^{k_{i_n}/2}>\|\log(1+b_\beta)\|_{\rho_n}\ ;\notag\\
\|\varphi^i(\log(1+b_\beta))\|_{\rho_j}>\|\varphi^i(\log(1+b_\beta)-s_{i_n})\|_{\rho_j}\notag
\end{align}
for all $1\leq i,j\leq n$. Such an $i_n$ exists as for any fixed $1\leq i,j\leq n$ we have $\lim_{k\to\infty}\|\varphi^i(\log(1+b_\beta)-s_{k})\|_{\rho_j}=0$. The first condition in \eqref{i_n} makes the definition of $i_n$ inductive. As a consequence, we have $\|\log(1+b_\beta)\|_{\rho_n}=\|s_{i_n}\|_{\rho_n}$ by the ultrametric inequality. Now define $j_n\in\mathbb{Z}$ so that
\begin{equation}\label{j_n}
\rho_n^{j_n+1}<\frac{\|s_{i_n}\|_{\rho_n}}{p^{k_{i_n}}}=\|p^{k_{i_n}}s_{i_n}\|_{\rho_n}=\|p^{k_{i_n}}\log(1+b_\beta)\|_{\rho_n}\leq \rho_n^{j_n}\ .
\end{equation}
(In other words $j_n=[\log_{\rho_n}(\|p^{k_{i_n}}\log(1+b_\beta)\|_{\rho_n})]$.) By \eqref{k_n} we have $j_n\geq 0$. Moreover, by the first condition in \eqref{i_n} the sequence $(j_n)_n$ is strictly increasing: $j_{n-1}<j_n$ for all $n>1$. On the other hand, $(-1)^{p^{k_{i_n}}}b_\beta^{p^{k_{i_n}}}$ is a summand in $p^{k_{i_n}}s_{i_n}$, therefore we have $\rho_n^{p^{k_{i_n}}}\leq \|p^{k_{i_n}}s_{i_n}\|_{\rho_n}\leq \rho_n^{j_n}$ whence 
\begin{equation}\label{j_ni_n}
j_n\leq p^{k_{i_n}}\leq i_n\ . 
\end{equation}
Put $x:=\sum_{n=1}^{\infty}p^{k_{i_n}}(\log(1+b_\beta)-s_{i_n})b_\alpha^{-j_n}$. Our goal in this step is to show that the sum $x$ converges in $\mathfrak{R}_0(N_0,\alpha)\subset \mathcal{R}(N_1,\ell)$. For this it suffices to verify that for any fixed $k\geq 1$ we have $\|p^{k_{i_n}}(\log(1+b_\beta)-s_{i_n})b_\alpha^{-j_n}\|_{\rho_k}\to 0$ as $n\to\infty$. Note that in the power series expansion of $\log(1+b_\beta)-s_{i_n}$ all the terms have degree $>i_n\geq j_n$ by \eqref{j_ni_n}. Therefore in the power series expansion of $x$ all the terms have positive degree. In particular, for $k<n$ we have $\|y\|_{\rho_k}\leq \|y\|_{\rho_n}$ whenever $y$ is a monomial in the expansion of $x$. By \eqref{i_n} and \eqref{j_n} we obtain
\begin{align*}
\|p^{k_{i_n}}(\log(1+b_\beta)-s_{i_n})b_\alpha^{-j_n}\|_{\rho_k}\leq \|p^{k_{i_n}}(\log(1+b_\beta)-s_{i_n})b_\alpha^{-j_n}\|_{\rho_n}<\\
<\frac{\|p^{k_{i_n}}\log(1+b_\beta)\|_{\rho_n}}{p^n}\rho_n^{-j_n}\leq \frac{1}{p^n}
\end{align*}
for $k<n$. Hence we have $x\in\mathfrak{R}_0(N_0,\alpha)\subset \mathcal{R}(N_1,\ell)$.

\emph{Step 2.} Note that by Remark \ref{log} $\log(1+b_\beta)$ lies in the kernel of $\pi_H$ for all open normal subgroup $H\leq N_1$. Hence by the continuity of $\pi_H$ we obtain $\pi_H(x)=\sum_{n=1}^{\infty}\pi_H(-p^{k_{i_n}}s_{i_n}b_\alpha^{-j_n})\in\mathcal{O}_{\mathcal{E}}^\dagger[N_1/H,\ell]\subseteq \mathcal{R}[N_1/H,\ell]$ as we have $-p^{k_{i_n}}s_{i_n}\in\mathbb{Z}_p[N_1]$ and $\mathcal{O}_{\mathcal{E}}^\dagger$ is closed in $\mathcal{R}$.

\emph{Step 3.} Assume finally that $j_{int}(z)=j(x)$ for some $z\in \mathcal{R}^{int}(N_1,\ell)$. Note that both $z$ and $j(x)\in\mathcal{O}_{\mathcal{E}}^\dagger\bs N_1,\ell\js\subset \mathcal{O_E}\bs N_1,\ell\js$ have a power series expansion. By the injectivity of $j_{int}$ these expansions are equal. Hence put $z=\sum_{{\bf k}\in\mathbb{Z}\times\mathbb{N}^{\Phi^+\setminus\{\alpha\}}}d_{\bf k}{\bf b}^{\bf k}$ with $d_{\bf k}\in\mathbb{Z}_p$. By the definition of $\mathcal{R}^{int}(N_1,\ell)$ there exists an element $t\in T_+$ such that $\varphi_t(z)$ lies in $\mathfrak{R}_0^{int}$. This means that there exists a positive integer $K_0$ such that for all fixed $k\geq K_0$ and $\varepsilon>0$ we have $\|\varphi_t(d_{\bf k}{\bf b}^{\bf k})\|_{\rho_k}<\varepsilon$ for all but finitely many ${\bf k}\in\mathbb{Z}\times\mathbb{N}^{\Phi^+\setminus\{\alpha\}}$. In particular, for any fixed $k\geq K_0$ we have 
\begin{equation*}
\|\varphi_t(-p^{k_{i_n}}s_{i_n}b_\alpha^{-j_n})\|_{\rho_k}<\varepsilon
\end{equation*}
for all but finitely many positive integers $n$ since the sequence $j_n$ is strictly increasing by construction therefore the terms in $x=\sum_{n=1}^{\infty}p^{k_{i_n}}(\log(1+b_\beta)-s_{i_n})b_\alpha^{-j_n}$ cannot cancel each other. Now we clearly have $\|\varphi_t(b_\alpha)\|_{\rho_k}\leq\rho_k$. On the other hand, we compute (for $n>\max(k,m(\beta,t))$ large enough) 
\begin{equation*}
\|\varphi_t(-p^{k_{i_n}}s_{i_n})\|_{\rho_k}=\frac{\|\varphi^{m(\beta,t)}(s_{i_n})\|_{\rho_k}}{p^{k_{i_n}}}=\frac{\|\varphi^{m(\beta,t)}(\log(1+b_\beta))\|_{\rho_k}}{p^{k_{i_n}}}=\frac{\|\log(1+b_\beta)\|_{\rho_k}}{p^{m(\beta,t)+k_{i_n}}}\ .
\end{equation*}
Hence we obtain
\begin{align*}
\varepsilon>\|\varphi_t(-p^{k_{i_n}}s_{i_n}b_\alpha^{-j_n})\|_{\rho_k}\geq\frac{\|\log(1+b_\beta)\|_{\rho_k}}{p^{m(\beta,t)+k_{i_n}}\rho_k^{j_n}}>\frac{\rho_k\|\log(1+b_\beta)\|_{\rho_k}}{p^{m(\beta,t)+k_{i_n}}\|p^{k_{i_n}}\log(1+b_\beta)\|_{\rho_n}^{\log_{\rho_n}\rho_k}}=\\
=\frac{\rho_k\|\log(1+b_\beta)\|_{\rho_k}}{p^{m(\beta,t)}}\frac{p^{k_{i_n}(\log_{\rho_n}\rho_k-1)}}{\|\log(1+b_\beta)\|_{\rho_n}^{\log_{\rho_n}\rho_k}}>\frac{\rho_k\|\log(1+b_\beta)\|_{\rho_k}}{p^{m(\beta,t)}}p^{k_{i_n}(1/2\cdot\log_{\rho_n}\rho_k-1)}
\end{align*}
using \eqref{i_n} and \eqref{j_n}. This is a contradiction as the right hand side above tends to $\infty$ as $n\to\infty$. Therefore $j(x)$ is not in the image of $j_{int}$ as claimed.
\end{proof}

\begin{rem}
The elements of $\mathcal{R}\bs N_1,\ell\js$ cannot be expanded as a skew Laurent series of the form $\sum_{{\bf k}\in\mathbb{Z}^{\Phi^+}}d_{\bf k}{\bf b^k}$ in general. Indeed, the sum $\sum_{n=1}^{\infty}\varphi^n(b_\beta)/p^{2^n}=\sum_{n=1}^{\infty}((b_\beta+1)^{p^n}-1)/p^{2^n}$ converges in $\mathcal{R}\bs N_1,\ell\js$ for any simple root $\beta\neq\alpha$ but does not have a skew Laurent-series expansion as the coefficient of $b_\beta$ in its expansion would be the non-convergent sum $\sum_{n=1}^\infty p^{n-2^n}$.
\end{rem}

We end this section by a diagram showing all the rings constructed.
\begin{equation*}
\xymatrix{
\mathcal{O_E}\ar@{^{(}->}[rrr] & & &\mathcal{O_E}\bs N_1,\ell\js=\Lambda_\ell(N_0)\\
\mathcal{O}_\mathcal{E}^\dagger\ar@{^{(}->}[u]\ar@{^{(}->}[r]\ar@{^{(}->}[d] & \mathfrak{R}_0^{int}(N_0,\alpha)\ar@{^{(}->}[r]\ar@{^{(}->}[d] & \mathcal{R}^{int}(N_1,\ell)\ar@{^{(}->}[r]^{j_{int}}\ar@{^{(}->}[d] & \mathcal{O}_\mathcal{E}^\dagger\bs N_1,\ell\js\ar@{^{(}->}[u]\ar@{^{(}->}[d]\\
\mathcal{E}^\dagger\ar@{^{(}->}[r]\ar@{^{(}->}[d] & \mathfrak{R}_0^{bd}(N_0,\alpha)\ar@{^{(}->}[r]\ar@{^{(}->}[d] & \mathcal{R}^{bd}(N_1,\ell)\ar@{^{(}->}[r]^{j_{int}\otimes_{\mathbb{Z}_p}\mathbb{Q}_p}\ar@{^{(}->}[d] & \mathcal{E}^{\dagger}\bs N_1,\ell\js\ar@{^{(}->}[d]\\
\mathcal{R}\ar@{^{(}->}[r] & \mathfrak{R}_0(N_0,\alpha)\ar@{^{(}->}[r] & \mathcal{R}(N_1,\ell)\ar[r]^{j} & \mathcal{R}\bs N_1,\ell\js
}
\end{equation*}

Here $ \mathcal{R}(N_1,\ell)$ consists of Laurent series $\sum_{\bf k}c_{\bf
  k}{\bf b}^{\bf k}$ with $c_\bk\in K$ that converge on the open annulus of the form
\begin{equation*}
\left\{\rho_2<|z_\alpha|<1,\ |z_{\beta}|\leq |z_\alpha|^r\text{ for
}\beta\in\Phi^+\setminus\{\alpha\}\right\}
\end{equation*}
for some $0<\rho_2<1$ and $1\leq r\in\mathbb{Z}$. The elements of $\mathfrak{R}_0(N_0,\alpha)$ are exactly those for which we can take $r=1$. Their analogous integral (resp.\ bounded) versions consist of those Laurent series having the same convergence condition for which $c_\bk\in o_K$ for all $\bk\in\mathbb{Z}^{\{\alpha\}}\times\mathbb{N}^{\Phi^+\setminus\{\alpha\}}$ (resp.\ for which $\{c_\bk\mid \bk\in\mathbb{Z}^{\{\alpha\}}\times\mathbb{N}^{\Phi^+\setminus\{\alpha\}}\}\subset K$ bounded).

\subsection{Towards an equivalence of categories for overconvergent and Robba rings}\label{overconv}

Note that Propositions \ref{equivcat} and \ref{T_+equiv} apply in both the
cases $R=\mathcal{O_E}$ and $R=\mathcal{O}^\dagger_{\mathcal{E}}$. In both
cases the category $\mathfrak{M}(R,\varphi)$ is the category of \emph{\'etale}
$\varphi$-modules over $R$. Moreover, by the main result of \cite{CC} (see also \cite{Ke2}), we also have an
equivalence of categories between finite free \'etale
$(\varphi,\Gamma)$-modules over $\mathcal{O}_{\mathcal{E}}^\dagger$ and finite
free \'etale $(\varphi,\Gamma)$-modules over $\mathcal{O_E}$ given by the base
change $\mathcal{O}_{\mathcal{E}}\otimes_{\mathcal{O}_{\mathcal{E}}^\dagger}\cdot$. On
the other hand, $T_\ell$ acts by automorphisms on an object $D$ in
$\mathfrak{M}(\mathcal{O}_{\mathcal{E}},T_+)$ and also on an object
$D^\dagger$ in $\mathfrak{M}(\mathcal{O}_{\mathcal{E}}^\dagger,T_+)$. Since
automorphisms correspond to automorphism in an equivalence of categories, we
obtain

\begin{pro}
The functors
\begin{eqnarray*}
\mathcal{O}_{\mathcal{E}}\otimes_{\mathcal{O}_{\mathcal{E}}^\dagger}\cdot\colon
\mathfrak{M}(\mathcal{O}_{\mathcal{E}}^\dagger,T_+)&\rightarrow&
\mathfrak{M}(\mathcal{O}_{\mathcal{E}},T_+)\\
\cdot^{\dagger}\colon \mathfrak{M}(\mathcal{O}_{\mathcal{E}},T_+)&\rightarrow&\mathfrak{M}(\mathcal{O}_{\mathcal{E}}^\dagger,T_+)
\end{eqnarray*}
are quasi-inverse equivalences of categories.
\end{pro}

Note that for the Robba ring $\mathcal{R}$ \'etaleness is stronger than what
we assumed for a module $D_{rig}^{\dagger}$ to belong to
$\mathfrak{M}(\mathcal{R},\varphi)$. The category
$\mathfrak{M}(\mathcal{R},\varphi)$ is just the category of $\varphi$-modules over
the Robba ring. Recall that an object $D^{\dagger}_{rig}$ in
$\mathfrak{M}(\mathcal{R},\varphi)$ is \emph{\'etale} (or unit-root, or pure of
slope zero) whenever it comes from an
overconvergent \'etale $\varphi$-module $D^{\dagger}$ over the ring of
``overconvergent'' power series $\mathcal{O}_{\mathcal{E}}^\dagger$ by base extension. We
denote by $\mathfrak{M}^0(\mathcal{R},\varphi)$ the category of \'etale
$\varphi$-modules over the Robba ring $\mathcal{R}$. We consequently define
the categories $\mathfrak{M}^0(\mathcal{R},T_{+})$,
$\mathfrak{M}^0(\mathcal{R}\bs N_1,\ell\js,\varphi)$, and
$\mathfrak{M}^0(\mathcal{R}\bs N_1,\ell\js,T_{+})$ as the full subcategory of \'etale
objects in the corresponding categories without superscript $0$. Note that via
the equivalence of categories \ref{T_+equiv}, \'etale objects correspond to
each other. Combining this observation with the main result of \cite{B} leads to

\begin{cor}\label{etaleRobba}
We have a commutative diagram of equivalences of categories
\begin{equation*}
\begin{matrix}
\mathfrak{M}^0(\mathcal{R},T_+)&\leftarrow&\mathfrak{M}(\mathcal{E}^\dagger,T_+)&\rightarrow&\mathfrak{M}(\mathcal{E},T_+)\\
\downarrow&&\downarrow&&\downarrow\\
\mathfrak{M}^0(\mathcal{R}\bs N_1,\ell\js,T_+)&\leftarrow&\mathfrak{M}(\mathcal{E}^\dagger\bs N_1,\ell\js,T_+)&\rightarrow&\mathfrak{M}(\mathcal{E}\bs
N_1,\ell\js,T_+).
\end{matrix}
\end{equation*}
\end{cor}
\begin{proof}
The left horizontal arrows are also equivalences of categories by \cite{B}
noting that $T_\ell$ acts via automorphisms on both types of objects in the
upper row.
\end{proof}

\begin{rem} The category $\mathfrak{M}^0(\mathcal{R}\bs N_1,\ell\js,T_+)$ of \'etale
  $T_+$-modules is embedded into the bigger category
  $\mathfrak{M}(\mathcal{R}\bs N_1,\ell\js,T_+)$. So we may 
speak of \emph{trianguline} objects in $\mathfrak{M}^0(\mathcal{R}\bs
N_1,\ell\js,T_+)$ as in the classical case (see for instance \cite{B1}). Indeed, we call an object
$M_{rig}^{\dagger}$ in $\mathfrak{M}^0(\mathcal{R}\bs N_1,\ell\js,T_+)$ trianguline
if it becomes a successive extension of objects in
$\mathfrak{M}(\mathcal{R}\bs N_1,\ell\js,T_+)$ of rank $1$ after a
finite base extension $L\otimes_K\cdot$. It is clear that trianguline objects
correspond to trianguline objects via the first vertical arrow in Corollary \ref{etaleRobba}.
\end{rem}
\begin{rem}
It would be interesting to construct a noncommutative version of the ``big'' rings $\tilde{\bf{A}}_{\mathbb{Q}_p}$ and $\tilde{\bf{A}}_{\mathbb{Q}_p}^\dagger$ in \cite{Ke2} and generalize (the proofs of) Theorems 2.3.5, 2.4.5, and 2.6.2 to this noncommutative setting. For this, one would need a generalization for results in the present paper to base fields other than $\mathbb{Q}_p$.
\end{rem}
\begin{rem}
Since we have the natural inclusions $\mathcal{O}_{\mathcal{E}}^\dagger\hookrightarrow
\mathcal{R}^{int}(N_1,\ell)\hookrightarrow
\mathcal{O}_{\mathcal{E}}^\dagger\bs N_1,\ell\js$, we have a fully faithful functor
\begin{equation*}
\Theta:=\left(\mathcal{R}^{int}(N_1,\ell)\otimes_{\mathcal{O}_{\mathcal{E}}^\dagger}\cdot\right)\circ
\left(\mathcal{O}_{\mathcal{E}}^\dagger\otimes_{\ell,\mathcal{O}_{\mathcal{E}}^\dagger[[
N_1,\ell]]}\cdot\right)\circ
\left(\mathcal{O}_{\mathcal{E}}^\dagger\bs N_1,\ell\js\otimes_{\mathcal{R}^{int}(N_1,\ell)}\cdot\right)
\end{equation*}
from the category $\mathfrak{M}(\mathcal{R}^{int}(N_1,\ell),T_+)$
to itself. Whether or not it is essentially surjective (or equivalently that
it is naturally isomorphic to the identity functor) is not clear. However, we
have $\Theta\cong\Theta\circ\Theta$ naturally.
\end{rem}
\begin{proof}
The faithfulness is clear since the objects in the category
$\mathfrak{M}(\mathcal{R}^{int}(N_1,\ell),T_+)$ are free modules, the maps $\mathcal{O}_{\mathcal{E}}^\dagger\hookrightarrow
\mathcal{R}^{int}(N_1,\ell)\hookrightarrow
\mathcal{O}_{\mathcal{E}}^\dagger\bs N_1,\ell\js$ are injective,
and the functor $\mathcal{O}_{\mathcal{E}}^\dagger\otimes_{\ell,\mathcal{O}_{\mathcal{E}}^\dagger[[
N_1,\ell]]}\cdot$ in the middle is an equivalence of categories by
Prop.\ \ref{T_+equiv}. The assertion $\Theta\cong\Theta\circ\Theta$ is also
clear by Prop.\ \ref{T_+equiv}. For the fullyness let
$f\colon\Theta(\mathcal{M}_1)\to\Theta(\mathcal{M}_2)$ be a morphism in
$\mathfrak{M}(\mathcal{R}^{int}(N_1,\ell))$. Then we have
$\Theta(f-\Theta(f))=0$ and by the faithfulness of $\Theta$ obtain $f=\Theta(f)$.
\end{proof}


\begin{thebibliography}{99}
\bibitem{B} L.\ Berger, Repr\'esentations $p$-adiques et equations differentielles,
    \emph{Invent. Math.} \textbf{148} (2002), 219--284.
\bibitem{B1} L.\ Berger, Trianguline representations, \emph{Bull.\ London Math.\ Soc.} \textbf{43} (2011), no. 4, 619--635. 
\bibitem{Be2} L.\ Berger, Multivariable Lubin-Tate $(\varphi,\Gamma)$-modules and filtered $\varphi$-modules, preprint, \url{http://perso.ens-lyon.fr/laurent.berger/articles/article23.pdf}
\bibitem{Br} C.\ Breuil, Invariant $L$ et s\'erie sp\'eciale $p$-adique, \emph{Ann.\ Scient.\ de l'E.N.S.}\ \textbf{37} (2004), 559--610.
\bibitem{CFKSV} J.\ Coates, T.\ Fukaya, K.\ Kato, R.\ Sujatha and O.\ Venjakob,
The $\GL_2$ main conjecture for elliptic curves without complex multiplication,
\emph{Publ. Math.} IHES \textbf{101} (2005), 163--208.
\bibitem{CC} P.\ Colmez, F.\ Cherbonnier, Repr\'esentations $p$-adiques
  surconvergentes, \emph{Invent. Math.} \textbf{133} (1998), 581--611.
\bibitem{Co3} P.\ Colmez, La s\'erie principale unitaire de
  $\GL_2(\mathbb{Q}_p)$, \emph{Ast\'erisque} \textbf{330} (2010), 213--262.
\bibitem{Co1} P.\ Colmez, Repr\'esentations de $\GL_2(\mathbb{Q}_p)$ et $(\varphi,\Gamma)$-modules, \emph{Ast\'erisque} \textbf{330} (2010), 281--509.
\bibitem{Co2} P.\ Colmez, $(\varphi,\Gamma)$-modules et repr\'esentations du mirabolique de $\GL_2(\mathbb{Q}_p)$, \emph{Ast\'erisque} \textbf{330} (2010), 61--153.
\bibitem{DDMS} Dixon, J. D., du Sautoy, M. P. F., Mann, A., Segal, D., \emph{Analytic pro-$p$ groups}.\ Second edition.\ Cambridge Studies in Advanced Mathematics, 61.\ Cambridge University Press, Cambridge, 1999.
\bibitem{F} J.-M.\ Fontaine, Repr\'esentations $p$-adiques des corps locaux, in ``The Grothendieck
Festschrift'', vol 2, Prog.\ in Math. 87, 249--309, Birkh\"auser 1991.
\bibitem{FvdP} J.\ Fresnel, M.\ van der Put, Rigid analytic geometry and its applications, \emph{Birkh\"auser}, Boston (2004).
\bibitem{Ke2} K.\ Kedlaya, New methods for $(\varphi,\Gamma)$-modules, preprint, \url{http://math.mit.edu/~kedlaya/papers/new-phigamma.pdf}
\bibitem{Ki} M.\ Kisin, Deformations of $G_{\mathbb{Q}_p}$ and $\GL_2(\mathbb{Q}_p)$-representations, \emph{Ast\'erisque} \textbf{330} (2010), 511--528. 
\bibitem{P} V.\ Pa\v{s}k$\mathrm{\bar{u}}$nas, The image of Colmez's Montreal functor, \emph{Publ.\ Math.\ }IHES, to appear.
\bibitem{ST1} P.\ Schneider, J.\ Teitelbaum, Algebras of $p$-adic distributions and admissible
representations, \emph{Invent. Math.} \textbf{153} (2003), 145--196.
\bibitem{SVe} P.\ Schneider, O.\ Venjakob, Localisations and
completions of skew power series rings,
\emph{American J.\ Math.} \textbf{132} (2010), 1--36.
\bibitem{SVi} P.\ Schneider, M.-F.\ Vigneras, A functor from smooth $o$-torsion
  representations to $(\varphi,\Gamma)$-modules, \emph{Clay Mathematics
    Proceedings} Volume \textbf{13} (2011), 525--601.
\bibitem{SVZ} P.\ Schneider, M.-F.\ Vigneras, G.\ Z\'abr\'adi, From \'etale
  $P_+$-representations to $G$-equivariant sheaves on $G/P$, to appear in \emph{Automorphic forms and Galois representations}, LMS Lecture Notes Series, Cambridge Univ.\ Press, \url{http://www.maths.dur.ac.uk/events/Meetings/LMS/2011/GRAF11/papers/schneider.pdf}
\bibitem{SZ} G.\ Z\'abr\'adi, Generalized Robba rings (with an
  Appendix by Peter Schneider), \emph{Israel J.\ Math.\ }\textbf{191}(2) (2012), 817--887.
\bibitem{Z} G.\ Z\'abr\'adi, Exactness of the reduction on \'etale modules, \emph{J. of Algebra} \textbf{331} (2011), 400--415.
 
\end{thebibliography}
\end{document}